\documentclass[12pt]{article}
\usepackage{}
\usepackage{amsfonts}
\usepackage{pifont}
\usepackage[english]{babel}
\RequirePackage{fix-cm}

\usepackage[pagewise]{lineno}
\usepackage{array}
\usepackage{longtable}
\usepackage{amsmath}
\usepackage{amsthm}
\usepackage{amssymb}
\usepackage{graphics}
\usepackage{graphicx}
\usepackage{subfigure}
\usepackage{epsfig}
\usepackage{delarray}
\usepackage{epstopdf}
\usepackage{color}
\usepackage{rotating}

\usepackage{graphics}
\usepackage{graphicx}
\usepackage{subfigure}
\usepackage{epsfig}
\usepackage{delarray}
\usepackage{multirow}

\title{Chenciner bifurcation, strong resonances and Arnold tongues of a discrete time SIR epidemic model
\thanks{This work is supported by National Natural Science
Foundation of China (Grant Nos.12171337 and 11701476), Central Government Guided Local Science and Technology Development Projects (Grant No.2024ZYD0059), Sichuan Provincial Natural Science Foundation of China (Grant No.2023NSFSC0064) and Open Research Fund Program of Data Recovery Key Laboratory of Sichuan Province (Grant No.DRN2405)}}

\author{
{\sc Jiangqiong Yu}$^{a}$,~~
{\sc Jiyu Zhong}$^{b}$,~~{\sc Lingling Liu}$^{c}$,
~~{\sc Zhiheng Yu}$^d$\footnote{Corresponding author: yuzhiheng9@163.com
}
\\
$^a${\small School of Mathematical Sciences,}
\\
{\small Chongqing Normal University, Chongqing 401331, P. R. China}
\\
$^b${\small School of Mathematics and Statistics, Hainan University,}
\\
{\small Haikou, Hainan, 570228, P. R. China}
\\
$^c${\small School of Sciences, Southwest Petroleum University,}\\
{\small Chengdu, Sichuan 610500, P. R. China}
\\
$^d${\small School of Mathematics, Southwest Jiaotong University,}
\\
{\small Chengdu, Sichuan 611756, P. R. China}
\\
}
\date{}

\begin{document}
\maketitle
\vskip -1.2cm
\begin{abstract}
In this paper, we mainly study the dynamic properties of a class of three-dimensional SIR models. Firstly, we use the {\it complete discriminant theory} of polynomials to obtain the parameter conditions for the topological types of each fixed point. Secondly, by employing the center manifold theorem and bifurcation theory, we prove that the system can undergo codimension 1 bifurcations, including transcritical, flip and Neimark-Sacker bifurcations, and codimension 2 bifurcations which contain Chenciner bifurcation, 1:3 and 1:4 strong resonances. Besides, by the theory of normal form, we give theoretically the Arnold tongues in the weak resonances such that the system possesses two periodic orbits on the stable invariant closed curve generated from the Neimark-Sacker bifurcation. Finally, in order to verify the theoretical results, we detect all codimension 1 and 2 bifurcations by using MatcontM and numerically simulate all bifurcation phenomena and the Arnold tongues in the weak resonances.

\vskip 0.2cm
{\bf AMS(2010) Subject Classifications:} 37G10; 39A28; 58K50; 68W30

\end{abstract}

\newcommand{\tl}[1]{\multicolumn{1}{l}{#1}} %
\renewcommand{\appendixname}{\appendix~{\thechapter}~}
\renewcommand{\theequation}{\thesection.\arabic{equation}}
\newtheorem{lm}{Lemma}
\newtheorem{defi}{Definition}
\newtheorem{prob}{Problem}
\newtheorem{thm}{Theorem}
\newtheorem{pro}{Proposition}
\newtheorem{exmp}{Example}
\newtheorem{rmk}{Remark}
\newtheorem{cor}{Corollary}
\newtheorem{con}{Conjecture}
\newcommand\T{\rule{-0.5pt}{2.6ex}}
\newcommand\B{\rule[-1.2ex]{0pt}{0pt}}
\vskip 0.5cm
\baselineskip 14pt
\parskip 8pt
\allowdisplaybreaks[4]
\section{Introduction}
It has long been known that infectious diseases always threatened the survival and development of humanity. Every occurrence of an infectious disease brings enormous disasters to humanity (\cite{Brauer, Hethcote, Johnson, Stone}). Therefore, how to control infectious diseases is an important issue that has received widespread attention from people. A basic method is to establish a mathematical model corresponding to infectious diseases and discuss their dynamic behaviors. Kermark and McKendrick proposed a basic epidemic model in 1927 as follow (see \cite{Kermack})
\begin{eqnarray}\label{firmod}
\left\{
\begin{array}{l}
\frac{dS}{dt}=-\hat{\beta} S I,\\
\frac{dI}{dt}=\hat{\beta} S I-\hat{\alpha} I,\\
\frac{dR}{dt}=\hat{\alpha} I,
\end{array}
\right.
\end{eqnarray}
where $S(t)$, $I(t)$ and $R(t)$ denote the numbers of susceptible, infective and recovered individuals at time $t$ respectively, the parameters $\hat{\alpha}$ and $\hat{\beta}$ represent the probability of patient recovery per unit time and the rate of infection per unit time of contact with others respectively. System \eqref{firmod} is called SIR model (see \cite{Brauer}), which is one of the early triumphs of mathematical epidemiology. After that, more and more people have begun to establish and study such kind of models to control the infectious diseases (\cite{Feng,LongJIM,Lu,Taylor,Yang2022}).
In addition, comparing with the continuous models, the discrete models not only preserve the basic features of the continuous model, but also are much more intuitive than continuous cases to see the numbers of susceptible, infected and immune individuals with richer dynamic properties and more simple to use statistical data for numerical simulations (\cite{Schreiber,Xiang}).

Consider the following discrete time SIR epidemic model, which was described by the difference equation mentioned in \cite{DinQ},
\begin{eqnarray}
\left\{
\begin{array}{rcl}
S(n+1)&=&S(n)-\frac{\alpha\,S(n)\,I(n)}{N}+\beta\,(I(n)+R(n)),\\
I(n+1)&=&(1-\beta-r)\,I(n)+\frac{\alpha\,S(n)\,I(n)}{N},\\
R(n+1)&=&(1-\beta)\,R(n)+r\,I(n),\\
\end{array}
\right.
\label{OM-0}
\end{eqnarray}
where $\alpha$ ($>0$) represents the average number of successful contacts (resulting in infections) which made by one infected (and infectious) individual during the time $n$ to $n+1$,
$\beta$ ($0<\beta<1$) represents the probability of a birth and
the probability of a death (where assuming that the probability of a birth equals the probability of a death),
$r$ ($0<r<1$) is the probability of recovery so that the ratio $1/r$ is the average length of the infectious period when there are no deaths (assuming the infected individual is also infectious, i.e. the infected individual can transmit the disease), and $N=S+I+R$ represents total population size in SIR model.
Furthermore, it is assumed that the initial conditions $S(0)$, $I(0)$ and $R(0)$ are positive real numbers such that $S(0) +I(0) +R(0)=N$.
Besides, in \cite{DinQ}, Din supposed that the total number of people $N$ unchanged, i.e., $S(n)+I(n)+R(n)=N$, and replaced $R(n)$ by $N-S(n)-I(n)$ so that system \eqref{OM-0} can be translated to following equivalent system
\begin{eqnarray}
\left\{
\begin{array}{l}
S(n+1)=S(n)-\frac{\alpha\,S(n)\,I(n)}{N}+\beta\,(N-S(n)),\\
I(n+1)=(1-\beta-r)\,I(n)+\frac{\alpha\,S(n)\,I(n)}{N},\\
R(n+1)=(1-\beta)\,R(n)+r\,I(n).\\
\end{array}
\right.
\label{OM-1}
\end{eqnarray}
Since the first two equations of \eqref{OM-1} are only about $S(n)$ and $I(n)$, for the following system
\begin{eqnarray}
\left\{
\begin{array}{l}
S(n+1)=S(n)-\frac{\alpha\,S(n)\,I(n)}{N}+\beta\,(N-S(n)),\\
I(n+1)=(1-\beta-r)\,I(n)+\frac{\alpha\,S(n)\,I(n)}{N}.\\
\end{array}
\right.
\label{OM-2}
\end{eqnarray}
Din investigated asymptotic stability of both disease-free and the endemic fixed points, and obtained the global stability of these fixed points under certain parametric conditions by using comparison method.

Later, Li and Eskandari  (\cite{LiB}) analysed the codimensions 1 and 2 bifurcations of the system \eqref{OM-2} by computing their critical normal form coefficients. They further detected the flip and Neimark-Sacker bifurcations, and strong resonances of system \eqref{OM-2}.
However, for the three-dimensional system \eqref{OM-1}, the stability of its fixed points and bifurcation phenomena have not been studied yet.
Although the first two equations are decoupled from the third one, the third equation is coupled with the first two.
A natural question is what are the dynamic properties of system \eqref{OM-1} near its fixed points when the numbers of susceptible, infective and recovered individuals changes simultaneously, so that we can obtain some strategies for controlling infectious diseases from these properties.
Furthermore, even if we can determine the numbers of recovered individuals by virtue of the numbers of susceptible and infective individuals and the relationship $S(n)+I(n)+R(n)=N$ between them, it remains extremely difficult to regulate the disease condition. This is due to the excessively large workload of computational analysis. However, we can regulate the disease more conveniently by directly studying the dynamic properties of system system \eqref{OM-1}.
Additionally, in this study, we not only establish the critical conditions for bifurcation occurrence but also presents rigorous theoretical proofs and numerical simulations of the resulting bifurcation phenomena, thereby offering substantial theoretical support for infectious disease prevention and control strategies.

This paper is organized as follows. In the second section, we first use {\it complete discrimination system} of polynomial (see \cite{YangHouXia2001,YangXia2005}) to determine the parameter conditions corresponding to the topological structure of the orbits near each fixed point. In the third section, center manifold theorem (\cite{Carr}) is used to prove the existence of codimension 1 bifurcations, and we prove that system \eqref{OM-1} undergos transcritical bifurcation at the fixed point $E_1$, flip and Neimark-Sacker bifurcations at $E_2$.
Then we show that system \eqref{OM-1} undergoes codimension 2 bifurcations at $E_2$ in the fourth section, including Chenciner bifurcation, $1:3$ and $1:4$ strong resonance bifurcations. In addition, we not only prove the existence of the codimension 2 bifurcations, but also demonstrate the existence of all dynamic phenomena (including the codimension 1 bifurcations and homoclinic structures) near bifurcation parameters.
Furthermore, in the fourth section, we give theoretically the Arnold tongues in the weak resonances such that the system possesses periodic orbits on the stable invariant closed curve generated from the Neimark-Sacker bifurcation.
Finally, we give a simulation to the dynamic phenomena of system \eqref{OM-1}, which verified our results in the previous sections.

\section{Qualitative properties of fixed points}
\allowdisplaybreaks[4]
\setcounter{equation}{0}
In this section, we first discuss the stability of the fixed points of system \eqref{OM-1}, and then present the parametric conditions in non-hyperbolic cases, so that the bifurcations of system \eqref{OM-1} can be studied in the next three sections.

Setting $S(n):=x(n)$, $I(n):=y(n)$ and $R(n):=z(n)$, system \eqref{OM-1} is regarded as a three-dimensional mapping
$$F:{\mathbb{R}}^{3}_0:=\{(x, y ,z)\in {\mathbb{R}}^{3}|x\geq0,y\geq0,z\geq0\} \rightarrow {{\mathbb{R}}^3_0}, $$
\begin{eqnarray}
F\left(
\begin{array}{l}
x \\
y \\
z
\end{array}
\right)
=
\left(
\begin{array}{l}
x-\frac{\alpha\,x\,y}{N}+\beta\,(N-x)\\
(1-\beta-r)y+\frac{\alpha\,x\,y}{N}\\
(1-\beta)\,z+r\,y
\end{array}
\right).
\label{eq2.1}
\end{eqnarray}

\begin{pro}
For parameters $N>0$, $0<\beta<1$, $0<r<1$ and $\alpha>0$, mapping \eqref{eq2.1} has at most two fixed points, i.e., a fixed point $E_1:(N, 0, 0)$ always exists, and the other fixed point
\begin{eqnarray*}
&&E_2:({\frac {N \left( \beta+r \right) }{\alpha}}, {\frac {\beta\,N \left( \alpha-\beta-r \right) }{\alpha\, \left( \beta+r \right) }},
{\frac {rN \left( \alpha-\beta-r \right) }{\alpha\, \left( \beta+r\right) }})
\end{eqnarray*}
exists in the case that $\alpha>\beta+r$.
Furthermore, the topological classifications for fixed points $E_1$ and $E_2$, are shown in Tables \ref{Table1} and \ref{Table2}, respectively.
\label{pro2.1}
\end{pro}
\begin{proof}
The fixed points of mapping \eqref{eq2.1} are given by equation
\begin{eqnarray*}
\left\{\begin{array}{ll}
x=x-\frac{\alpha\,x\,y}{N}+\beta\,(N-x),\\
y=(1-\beta-r)y+\frac{\alpha\,x\,y}{N},\\
z=(1-\beta)\,z+r\,y.
\end{array}\right.
\end{eqnarray*}
One can solve the above equations and get the following fixed points
\begin{eqnarray*}
&&E_1:(N, 0, 0)~~\mbox{and}~~E_2:({\frac {N \left( \beta+r \right) }{\alpha}}, {\frac {\beta\,N \left( \alpha-\beta-r \right) }{\alpha\, \left( \beta+r \right) }}, {\frac {rN \left( \alpha-\beta-r \right) }{\alpha\, \left( \beta+r\right) }}).
\end{eqnarray*}
Since $x\geq0$, $y\geq0$, $z\geq0$, and $E_2$ does not coincide with $E_1$, it follows that the conditions of existence for $E_2$ is $N>0$, $0<\beta<1$, $0<r<1$ and $\alpha>\beta+r$.
Next, we discuss the topological types of $E_1$ and $E_2$. Firstly, the Jacobian matrices at $E_1$ and $E_{2}$ are given by
$$
JF(E_1)=\left(
        \begin{array}{ccc}
          1-\beta & -\alpha & 0 \\
          0 & 1-\beta-r+\alpha & 0 \\
          0 & r & 1-\beta
        \end{array}
      \right)
$$
and
$$
JF(E_2)\!=\!\left(\!\!
        \begin{array}{ccc}
          {\frac {\beta+r-\beta\,r }{\beta+r}} & -\beta-r & 0 \\
          {\frac {\beta\, \left( \alpha-\beta-r \right) }{\beta+r}} & 1 & 0 \\
          0 & r & 1-\beta
        \end{array}
     \!\! \right),
$$
respectively.
It is clear that the eigenvalues of $JF(E_1)$ are $1+\alpha-\beta-r$ and $1-\beta$ with multiplicity 2, which are all real. Hence,
\begin{description}
\item[(I)] If $E_1$ is a stable node, then $|1-\beta|<1$
 and $|1-\beta-r+\alpha|<1$. Thus we obtain the corresponding semi-algebraic system
    \begin{eqnarray*}
    &&PS_1:=\{N>0, 0<\beta<1, 0<r<1, \alpha>0, 1-\beta>-1, 1-\beta<1, \\
    &&~~~~~~~~~~~~1-\beta-r+\alpha>-1, 1-\beta-r+\alpha<1\}.
    \end{eqnarray*}
    Solving $PS_1$, we get the following parameter conditions
    \begin{description}
    \item[$\bullet$]~$N>0,0<\beta<1, 0<r<1, 0<\alpha<\beta+r$
    \end{description}
    (referred to cases $\mathbf{D}_{1}$ in Table \ref{Table1}).
\item[(II)] If $E_1$ is an unstable node, then $|1-\beta|>1$ and $|1-\beta-r+\alpha|>1$,
and we see that the solution set of corresponding semi-algebraic system is the empty.
\item[(III)] If $E_1$ is a saddle point, then either $|1-\beta|>1$, $|1-\beta-r+\alpha|<1$, or $|1-\beta|<1$, $|1-\beta-r+\alpha|>1$. Applying the {\it complete discrimination system} as done in the above two cases, we get
    \begin{description}
    \item[$\bullet$]~$N>0,0<\beta<1, 0<r<1, \alpha>\beta+r$
    \end{description}
    (referred to cases $\mathbf{D}_{2}$ in Table \ref{Table1}).
\item[(IV)] If $E_1$ is non-hyperbolic, then $|1-\beta|=1$ or $|1-\beta-r+\alpha|=1$. It follows that
    \begin{description}
    \item[$\bullet$]~$N>0,0<\beta<1, 0<r<1, \alpha=\beta+r$
    \end{description}
    (referred to cases $\mathbf{L}_{1}$ in Table \ref{Table1}).
\end{description}
\begin{table*}[htbp]
 \centering
 \caption{Topological types of fixed point $E_1$.}
\label{Table1}
{\fontsize{12pt}{\baselineskip}\selectfont\begin{tabular}{llllll}
\hline\noalign{\smallskip}
\multicolumn{4}{c}{\multirow{2}*{Parameters}}  & \multicolumn{1}{c}{Properties} & \multirow{2}*{Cases}\\ \noalign{\smallskip}\cline{5-5}\noalign{\smallskip}
\multicolumn{4}{c}{}& ~~~~~~~~$E_1$ &  \\ \noalign{\smallskip}\hline\noalign{\smallskip}
 $N>0$  & $0<\beta<1$  & $0<r<1$ & $0<\alpha<\beta+r$ & stable node   & $\mathbf{D}_{1}$\\
        &              &         & $\alpha=\beta+r$   & non-hyperbolic& $\mathbf{L}_{1}$\\
        &              &         & $\alpha>\beta+r$   & saddle point  & $\mathbf{D}_{2}$\\
 \noalign{\smallskip}\hline
 \end{tabular}}
\end{table*}

In what follows we get down to the topological type of fixed point $E_2$. It is clear that the characteristic polynomial of $JF(E_2)$ is
\begin{eqnarray*}
\mathcal{P}_{E_2}(t)&:=&\big(t-(1-\beta)\big)\Big(-{t}^{2}+\frac { \left(  \left( 2-\alpha \right) \beta+2\,r \right) t}{\beta+r}\\
&&~+\frac {{\beta}^{3}+ \left( -\alpha+2\,r \right) {\beta}^{2}+ \left( r-1 \right)  \left( r-\alpha+1 \right) \beta-r}{\beta+r}\Big).
\end{eqnarray*}
We are apt to see $JF(E_2)$ has a real eigenvalue which is equal to $1-\beta$, and the other two eigenvalues can be given by the zeros of the polynomial
\begin{eqnarray*}
&&{P}_{E_2}(t):=\frac{\mathcal{P}_{E_2}(t)}{t-(1-\beta)}=-{t}^{2}+\frac { \left(  \left( 2-\alpha \right) \beta+2\,r \right) t}{\beta+r}\\
&&~~~~~~~~~~~~~~~~~~~~~~~~~~~~~~~~~+\frac {{\beta}^{3}+ \left( -\alpha+2\,r \right) {\beta}^{2}+ \left( r-1 \right)  \left( r-\alpha+1 \right) \beta-r}{\beta+r}.
\end{eqnarray*}
Solving ${P}_{E_2}(t)=0$, we obtain the other eigenvalues of $JF(E_2)$, i.e.,
\begin{eqnarray}
t_{1,2}=\frac {-\beta\,\alpha+2\,\beta+2\,r\pm\sqrt{\Delta}}{2\,\beta+2\,r},
\label{solu-1}
\end{eqnarray}
where
{\small
\begin{eqnarray*}
\Delta:={\alpha}^{2}{\beta}^{2}-4
\,\alpha\,{\beta}^{3}-8\,\alpha\,{\beta}^{2}r-4\,\alpha\,\beta\,{r}^{2
}+4\,{\beta}^{4}+12\,{\beta}^{3}r+12\,{\beta}^{2}{r}^{2}+4\,\beta\,{r}
^{3}.
\label{eq2.3}
\end{eqnarray*}}
\!\!Thereupon,
\begin{description}
\item[(i)] If $E_2$ is a stable node, then $|1-\beta|<1$ and $|t_{12}|<|t_{11}|<1$, i.e., $|1-\beta|<1$, ${P_{E_2}}(1)<0$, ${P_{E_2}}(-1)<0$, $\Delta\geq0$ and $-1<t_*<1$, where $t_*:=-(\beta\,\alpha-2\,\beta-2\,r)/(2\,\beta+2\,r)$, denotes the axis symmetry of $P_{E_2}(t)$. Thus, the corresponding semi-algebraic system is
    \begin{eqnarray*}
    &&\!\!\!\!\!\!\!\!\!\!\!\!\!\!\!PS_2:=\{N>0, 0<\beta<1, 0<r<1, \alpha>0, 1-\beta>-1, 1-\beta<1, \\
    &&~~~\left( -\alpha+\beta+r \right) \beta<0, -{\frac {\alpha\,\beta-2\,\beta-2\,r}{2\,r+2\,\beta}}>-1, -{\frac{\alpha\,\beta-2\,\beta-2\,r}{2\,r+2\,\beta}}<1,\\
    &&~~~\frac{{\beta}^{3}+ \left( -\alpha+2\,r \right) {\beta}^{2}+ \left( r-2\right)  \left( r-\alpha+2 \right) \beta-4\,r}{\beta+r}<0,\\
    &&~~~{\frac{\beta\,\left({\alpha}^{2}\beta-4\,\alpha\,{\beta}^{2}-8\,\alpha\,\beta\,r-4\,\alpha\,{r}^{2}+4\,{\beta}^{3}+12\,r{\beta}^{2}+12\,{r}^{2}\beta+4\,{r}^{3} \right) }{ \left( r+\beta \right) ^{2}}}\geq0\}.
    \end{eqnarray*}
   Employing the {\it  complete discrimination system} theory to solve $PS_2$, we obtain the following parameter conditions to guarantee the stability of $E_2$:
    \begin{description}
    \item[$\bullet$]~$N>0$, $0<\beta<1$, $0<r<1$, $r+\beta<\alpha\leq\Upsilon_{1}$, or
    \item[$\bullet$]~$N>0$, $0<\beta<1$, $0<r<\Psi_1$, $\Upsilon_{2}\leq\alpha<\Psi_2$,
    \end{description}
    where
    \begin{eqnarray*}
    &&\!\!\!\!\!\!\!\!\!\!\Upsilon_{1,2}:={\frac {2\,({\beta}^{2}+2\,r\beta+{r}^{2}\mp\sqrt {r{\beta}^{3}+3\,{r}^{2}{\beta}^{2}+3\,{r}^{3}\beta+{r}^{4}})}{\beta}},
    \Psi_1:=-{\frac {{\beta}^{2}-4\,\beta+4}{\beta-4}}\\
    &&\!\!\!\!\!\!\!\!\!\!\mbox{and}~~ \Psi_2:={\frac {{\beta}^{3}+2\,r{\beta}^{2}+{r}^{2}\beta-4\,\beta-4\,r}{\beta\, \left( r+\beta-2 \right) }}
    \end{eqnarray*}
    (referred to cases $\mathfrak{D}_{1i}$ in Table \ref{Table2}, where $i= 1,2,3,4$).
\item[(ii)] If $E_2$ is an unstable node, then $|1-\beta|>1$ and $|t_{1,2}|>1$. Moreover, $|t_{1,2}|>1$ implies that one of the following conditions is satisfied,
    \begin{description}
    \item[$\lozenge$]~${P_{E_2}}(1)>0$, ${P_{E_2}}(-1)>0$, $\Delta>0$;
    \item[$\lozenge$]~${P_{E_2}}(1)<0$, ${P_{E_2}}(-1)<0$, $\Delta\geq0 $, $t_*<-1$;
    \item[$\lozenge$]~${P_{E_2}}(1)<0$, ${P_{E_2}}(-1)<0$, $\Delta\geq0 $, $t_*>1$.
    \end{description}
    Using the same idea as in case (i), we obtain the null sets.
\item[(iii)] If $E_2$ is a saddle point, then $1-\beta$, $t_1$ and $t_2$ are all real, and one of the following conditions holds:
    \begin{description}
    \item[$\lozenge$]~$|1-\beta|>1$, ${P_{E_2}}(1)<0$, ${P_{E_2}}(-1)<0$, $\Delta\geq0$, $-1<t_*<1$,
    \item[$\lozenge$]~$|1-\beta|<1$, ${P_{E_2}}(1)>0$, ${P_{E_2}}(-1)<0$, $\Delta\geq0 $,
    \item[$\lozenge$]~$|1-\beta|<1$, ${P_{E_2}}(1)<0$, ${P_{E_2}}(-1)>0$, $\Delta\geq0 $,
    \item[$\lozenge$]~$|1-\beta|>1$, ${P_{E_2}}(1)>0$, ${P_{E_2}}(-1)<0$, $\Delta\geq0 $,
    \item[$\lozenge$]~$|1-\beta|>1$, ${P_{E_2}}(1)<0$, ${P_{E_2}}(-1)>0$, $\Delta\geq0 $,
    \item[$\lozenge$]~$|1-\beta|<1$, ${P_{E_2}}(1)>0$, ${P_{E_2}}(-1)>0$, $\Delta>0$,
    \item[$\lozenge$]~$|1-\beta|<1$, ${P_{E_2}}(1)<0$, ${P_{E_2}}(-1)<0$, $\Delta\geq0 $, $t_*<-1$,
    \item[$\lozenge$]~$|1-\beta|<1$, ${P_{E_2}}(1)<0$, ${P_{E_2}}(-1)<0$, $\Delta\geq0 $, $t_*>1$.
    \end{description}
    Hence, from above conditions, we obtain
    \begin{description}
    \item[$\bullet$]~$N>0$, $0<\beta<1$, $0<r<1$, $\alpha>\Psi_2$, or
    \item[$\bullet$]~$N>0$, $0<\beta<1$, $\Psi_1<r<1$, $\Upsilon_{2}\leq\alpha<\Psi_2$
    \end{description}
    (referred to cases $\mathfrak{D}_{3i}$ in Table \ref{Table2}, where $i= 1,2,3,4$).
\item[(iv)] If $E_2$ is a stable focus-node, then $|1-\beta|<1$, $\Delta<0$ and $|t_1\,t_2|<1$. It follows that
    \begin{description}
    \item[$\bullet$]~$N>0$, $0<\beta<1$, $0<r<\Psi_1$, $\Upsilon_{1}<\alpha<\Upsilon_{2}$, or
    \item[$\bullet$]~$N>0$, $0<\beta<1$, $r=\Psi_1$, $\Upsilon_{1}<\alpha<\Psi_2(=\Psi_3=\Upsilon_2)$, or
    \item[$\bullet$]~$N>0$, $0<\beta<1$, $\Psi_1<r<1$, $\Upsilon_{1}<\alpha<\Psi_3$,
    \end{description}
    where
    \begin{eqnarray*}
    \Psi_3:={\frac {{\beta}^{2}+2\,r\beta+{r}^{2}}{r+\beta-1}}
    \end{eqnarray*}
    (referred to cases $\mathfrak{D}_{2i}$ in Table \ref{Table2}, where $i= 1,2,3$).
\item[(v)] If $E_2$ is an unstable focus-node, then $|1-\beta|>1$, $\Delta<0$ and $|t_1\,t_2|>1$. Using the {\it  complete discrimination system} theory to solve the corresponding semi-algebraic system, we find that its solution set is an empty set.
\item[(vi)] If $E_2$ is a saddle-focus, then one of the following conditions holds:
    \begin{description}
    \item[$\lozenge$]~$|1-\beta|>1$, $\Delta<0$, $|t_1\,t_2|<1$,
    \item[$\lozenge$]~$|1-\beta|<1$, $\Delta<0$, $|t_1\,t_2|>1$.
    \end{description}
    Solving the corresponding semi-algebraic system, we obtain the following parameter conditions:
   \begin{description}
    \item[$\bullet$]~$N>0$, $0<\beta<1$, $\Psi_1<r<1$, $\Psi_3<\alpha<\Upsilon_{2}$,
    \end{description}
    (referred to cases $\mathfrak{D}_{4}$ in Table \ref{Table2}).
\item[(vii)] If $E_2$ is non-hyperbolic, then one of the following conditions is fulfilled:
    \begin{description}
    \item[$\lozenge$]~$|1-\beta|=1$,
    \item[$\lozenge$]~$\Delta\geq0$, ${P_{E_2}}(1)=0$,
    \item[$\lozenge$]~$\Delta\geq0$, ${P_{E_2}}(-1)=0$,
    \item[$\lozenge$]~$\Delta<0$, $|t_1\,t_2|=1$.
    \end{description}
    As done in the previous cases, we obtain that if
    \begin{description}
    \item[$\bullet$]~$N>0$, $0<\beta<1$, $0<r<\Psi_1$, $\alpha=\Psi_2$, or
    \item[$\bullet$]~$N>0$, $0<\beta<1$, $r=\Psi_1$, $\alpha=\Psi_2(=\Psi_3=\Upsilon_2)$, or
    \item[$\bullet$]~$N>0$, $0<\beta<1$, $\Psi_1<r<1$, $\alpha=\Psi_2$, or
    \item[$\bullet$]~$N>0$, $0<\beta<1$, $\Psi_1<r<1$, $\alpha=\Psi_3$,
    \end{description}
    (referred to cases $\mathfrak{L}_{1i}$ and $\mathfrak{L}_{2}$ in Table \ref{Table2}, where $i= 1,2,3$),
    then $E_2$ is non-hyperbolic.
\end{description}
This completes the whole proof.
\end{proof}
\begin{table*}[htbp]
 \centering
 \caption{Topological types of fixed point $E_2$.}
\label{Table2}
{\fontsize{12pt}{\baselineskip}\selectfont\begin{tabular}{llllll}
\hline\noalign{\smallskip}
\multicolumn{4}{c}{\multirow{2}*{Parameters}}  & \multicolumn{1}{c}{Properties} & \multirow{2}*{Cases}\\ \noalign{\smallskip}\cline{5-5}\noalign{\smallskip}
\multicolumn{4}{c}{}& ~~~~~~~~$E_2$ &  \\ \noalign{\smallskip}\hline\noalign{\smallskip}
$N>0$ & $0<\beta<1$  & $0<r<\Psi_1$ & $r+\beta<\alpha\leq\Upsilon_1$          & stable node       & $\mathfrak{D}_{11}$\\
      &              &              & $\Upsilon_1<\alpha<\Upsilon_2$          & stable focus-node & $\mathfrak{D}_{21}$\\
      &              &              & $\Upsilon_2\leq\alpha<\Psi_2$           & stable node       & $\mathfrak{D}_{12}$\\
      &              &              & $\alpha=\Psi_2$                         & non-hyperbolic    & $\mathfrak{L}_{11}$\\
      &              &              & $\alpha>\Psi_2$                         & saddle point      & $\mathfrak{D}_{31}$\\
      &              & $r=\Psi_1$   & $r+\beta<\alpha\leq\Upsilon_1$          & stable node       & $\mathfrak{D}_{13}$\\
      &              &              & $\Upsilon_1<\alpha<\Psi_2(=\Psi_3=\Upsilon_2)$ & stable focus-node & $\mathfrak{D}_{22}$\\
      &              &              & $\alpha=\Psi_2$                         & non-hyperbolic    & $\mathfrak{L}_{12}$\\
      &              &              & $\alpha>\Psi_2$                         & saddle point      & $\mathfrak{D}_{32}$\\
      &              & $\Psi_1<r<1$ & $r+\beta<\alpha\leq\Upsilon_1$          & stable node       & $\mathfrak{D}_{14}$\\
      &              &              & $\Upsilon_1<\alpha<\Psi_3$              & stable focus-node & $\mathfrak{D}_{23}$\\
      &              &              & $\alpha=\Psi_3$                         & non-hyperbolic    & $\mathfrak{L}_{2}$\\
      &              &              & $\Psi_3<\alpha<\Upsilon_2$              & saddle-focus     &$\mathfrak{D}_{4}$\\
      &              &              & $\Upsilon_2\leq\alpha<\Psi_2$           & saddle point     &$\mathfrak{D}_{33}$\\
      &              &              & $\alpha=\Psi_2$                         & non-hyperbolic   &$\mathfrak{L}_{13}$\\
      &              &              & $\alpha>\Psi_2$                         & saddle point     & $\mathfrak{D}_{34}$\\
 \noalign{\smallskip}\hline
 \end{tabular}}
\end{table*}

In next sections, we will study the bifurcation phenomena of fixed points $E_1$ and $E_2$ under the above non-hyperbolic conditions, mentioned in Proposition \ref{pro2.1}.

\section{Codimension 1 bifurcations}
\allowdisplaybreaks[4]
\setcounter{equation}{0}
In this section, we will investigate all local codimension 1 bifurcations of mapping \eqref{eq2.1} at the fixed points $E_1$ and $E_2$.
\subsection{Transcritical bifurcation of $E_1$}
\allowdisplaybreaks[4]
In this subsection, we will discuss the bifurcations happened near the fixed point $E_1:(N,0,0)$. From Proposition 1, we see that mapping \eqref{eq2.1} has one eigenvalue 1 at the fixed point $E_1$ if $N>0,0<\beta<1,0<r<1,\alpha=\beta+r$, which implies the mapping may produce one of the following three codimension 1 bifurcations, i.e, transcritical bifurcation, fold bifurcation and pitchfork bifurcation. To figure out which bifurcation can happen, we have the following results.

\begin{thm}
For $N>0$, $0<\beta<1$ and $0<r<1$,
mapping \eqref{eq2.1} undergoes a transcritical bifurcation as the parameter $\alpha$ crosses
$\beta+r$.
More specifically, for $\alpha$ is in a small neighborhood of $\beta+r$, the fixed points $E_1$ is unstable and $E_2$ is stable if $\alpha>\beta+r$, and $E_1$, $E_2$ are coincident if $\alpha=\beta+r$. Besides, when $\alpha<\beta+r$, $E_1$ also exists but $E_2$ disappears for biological significance, and in this case $E_1$ is stable.
\label{th3.1}
\end{thm}
\begin{proof}
Let $\delta:=\alpha-(\beta+r)$. Translating $E_1$ to the origin $O$ in mapping \eqref{eq2.1} through the transformation $(x,y,z)^T=(u+N,v,w)^T$
and expanding it in Taylor series, we obtain the mapping of the following suspended form $F:\mathbb{R}^3\rightarrow\mathbb{R}^3$,
\begin{eqnarray}
F:
\left[
\begin{array}{cc}
   u \\
   v \\
   w  \\
\end{array}
\right]
\mapsto
\left[\begin{array}{cc}
 -u \left( \beta-1 \right) - \left( \beta+r+\delta \right) v-{\frac {uv\left( \beta+r+\delta \right) }{N}}\\
 \left( \delta+1 \right) v+{\frac {uv \left( \beta+r+\delta \right) }{N}}\\
 - \left( \beta-1 \right) w+vr\\
 \end{array}\right].
\label{eq3.1}
\end{eqnarray}
Clearly, the Jacobian matrix $JF(E_1)$ has eigenvectors
$$(-\frac{\beta+r+\delta}{r},\frac{\delta+\beta}{r},1)^T,~~(0,0,1)^T~~\mbox{and}~~(1,0,0)^T,$$
and their corresponding eigenvalues are
$$\delta+1,~~1-\beta~~\mbox{and}~~1-\beta.$$
Moreover, applying the invertible transformation
\begin{eqnarray*}
\left[
\begin{array}{cc}
   u \\
   v \\
   w
\end{array}
\right]
=\left[
   \begin{array}{ccc}
     -\frac{\beta+r+\delta}{r} & 0 & 1\\
     \frac{\delta+\beta}{r}    & 0 & 0\\
     1                         & 1 & 0\\
   \end{array}
 \right]
\left[\begin{array}{cc}
   u_1 \\
   v_1 \\
   w_1
 \end{array}\right],
\end{eqnarray*}
we change the linear part of mapping \eqref{eq3.1} into the normalized form, i.e.
\begin{eqnarray}
\left[
\begin{array}{ccc}
   u_1 \\
   v_1 \\
   w_1
\end{array}
\right]
\mapsto
\left[
\begin{array}{ccc}
\left( \delta+1 \right) u_1-{\frac { \left( {\beta}^{2}+2\,\delta\,\beta+2\,r\beta+{\delta}^{2}+2\,r\delta+{r}^{2} \right) u^2_1}{rN}}+
{\frac {\left( \beta+r+\delta \right) u_1\,w_1}{N}}\\
- \left( \beta-1 \right) v_1+{\frac { \left( \beta+r+\delta \right) ^{2}u^2_1}{rN}}-{\frac {\left( \beta+r+\delta \right) u_1\,w_1}{N}}\\
-\left( \beta-1 \right)\,w_1 -{\frac { \left( \beta+r+\delta \right) ^{2}u^2_1}{rN}}+{\frac { \left( \beta+r+\delta \right) u_1\,w_1}{N}}
\end{array}\right].
\label{eq3.2}
\end{eqnarray}
Further, choose $\delta$ as the bifurcation coefficient, then mapping \eqref{eq3.2} is converted into the following form,
\begin{eqnarray}
\left[
\begin{array}{ccc}
   u_1 \\
   v_1 \\
   w_1\\
   \delta
\end{array}
\right]
\mapsto
\left[
\begin{array}{ccc}
\left( \delta+1 \right) u_1-{\frac { \left( {\beta}^{2}+2\,\delta\,\beta+2\,r\beta+{\delta}^{2}+2\,r\delta+{r}^{2} \right) u^2_1}{rN}}+
{\frac {\left( \beta+r+\delta \right) u_1\,w_1}{N}}\\
- \left( \beta-1 \right) v_1+{\frac { \left( \beta+r+\delta \right) ^{2}u^2_1}{rN}}-{\frac {\left( \beta+r+\delta \right) u_1\,w_1}{N}}\\
-\left( \beta-1 \right)\,w_1 -{\frac { \left( \beta+r+\delta \right) ^{2}u^2_1}{rN}}+{\frac { \left( \beta+r+\delta \right) u_1\,w_1}{N}}\\
\delta
\end{array}\right].
\label{eq3.3}
\end{eqnarray}
Since mapping (\ref{eq3.3}) has exact two eigenvalues on the unit circle $S^1$, by center manifold theorem (see \cite[pp.33-35]{Carr}),
it has a two dimensional $C^2$ center manifold
\begin{eqnarray}
\begin{array}{l}
W_1^c(O)=\{(u_1,v_1,w_1,\delta)\in \mathbb{R}^4:v_1=h_{11}(u_1,\delta),\\
~~~~~~~~~~~~~~~w_1=h_{12}(u_1,\delta),|u_1|<\zeta_1, |\delta|<\zeta_2\},
\end{array}
\label{mf1}
\end{eqnarray}
where $h_{11}$ and $h_{12}$ are $C^2$ functions near $(0,0)$ such that
$$
h_{11}(0,0)=0,~ Dh_{11}(0,0)=0,~ h_{12}(0,0)=0,~Dh_{12}(0,0)=0,
$$
and $\zeta_1$, $\zeta_2$ are both sufficiently small positive constants.
Due to the $C^2$ smoothness, we can assume that $h_{11}(u_1,\delta)$ and $h_{12}(u_1,\delta)$ are of the following forms
\begin{eqnarray}
\begin{array}{c}
v_1=h_{11}(u_1,\delta)=d_{11}u_1^2+d_{12}u_1\,\delta+d_{13}\delta^2+O(|(u_1,\delta)|^3),\\
w_1=h_{12}(u_1,\delta)=d_{21}u_1^2+d_{22}u_1\,\delta+d_{23}\delta^2+O(|(u_1,\delta)|^3),
\end{array}
\label{mf2}
\end{eqnarray}
where $d_{11}$, $d_{12}$, $d_{13}$, $d_{21}$, $d_{22}$ and $d_{23}$ are indeterminate.
By the invariance of center manifold, from (\ref{mf1}) and (\ref{mf2}) we obtain
\begin{eqnarray}
\begin{array}{c}
h_{11}(\left( \delta+1 \right) u_1-{\frac { \left( {\beta}^{2}+2\,\delta\,\beta+2\,r\beta+{\delta}^{2}+2\,r\delta+{r}^{2} \right) u^2_1}{rN}}+
{\frac {\left( \beta+r+\delta \right) u_1\,h_{11}(u_1,\delta)}{N}},\delta)\\
~~~~~~~~~~~~=- \left( \beta-1 \right) h_{11}(u_1,\delta)+{\frac { \left( \beta+r+\delta \right) ^{2}u^2_1}{rN}}-{\frac {\left( \beta+r+\delta \right) u_1\,h_{12}(u_1,\delta)}{N}}\
\end{array}
\label{coe11}
\end{eqnarray}
and
\begin{eqnarray}
\begin{array}{c}
h_{12}(\left( \delta+1 \right) u_1-{\frac { \left( {\beta}^{2}+2\,\delta\,\beta+2\,r\beta+{\delta}^{2}+2\,r\delta+{r}^{2} \right) u^2_1}{rN}}+
{\frac {\left( \beta+r+\delta \right) u_1\,h_{11}(u_1,\delta)}{N}},\delta)\\
~~~~~~~~~~~~=-\left( \beta-1 \right)\,h_{12}(u_1,\delta) -{\frac { \left( \beta+r+\delta \right) ^{2}u^2_1}{rN}}+{\frac { \left( \beta+r+\delta \right) u_1\,h_{12}(u_1,\delta)}{N}}.
\end{array}
\label{coe12}
\end{eqnarray}
Comparing the coefficients of $w_1^2$, $w_1\delta$ and $\delta^2$ in \eqref{coe11} and \eqref{coe12}, we obtain
$$d_{11}=-d_{21}={\frac {{\beta}^{2}+2\,r\beta+{r}^{2}}{rN\beta}}~~\mbox{and}~~d_{12}=d_{13}=d_{22}=d_{23}=0.$$
Therefore, from (\ref{mf2}) we get
\begin{eqnarray}
\begin{array}{c}
u_1=h_{11}(w_1,\delta)={\frac {({\beta}^{2}+2\,r\beta+{r}^{2})w_1^2}{rN\beta}}+O(|(u_1,\delta)|^3),\\
v_1=h_{12}(w_1,\delta)=-{\frac {({\beta}^{2}+2\,r\beta+{r}^{2})w_1^2}{rN\beta}}+O(|(u_1,\delta)|^3).
\end{array}
\label{coe2}
\end{eqnarray}
Substituting (\ref{coe2}) into the last two mappings of (\ref{eq3.3}), we get
\begin{eqnarray*}
\left[\!
\begin{array}{cc}
   w_1 \\
   \delta
\end{array}
\!\right]
\!\mapsto\!
\left[\!\begin{array}{cc}
u_1-{\frac { \left( {\beta}^{2}+2r\beta+{r}^{2} \right) u^2_1}{Nr}}+
u_1\delta-{\frac { \left( {\beta}^{3}+3r{\beta}^{2}+3{r}^{2}\beta+
{r}^{3} \right)u^3_1}{{N}^{2}r\beta}}-{\frac { 2\,\left( \beta+r
 \right) u^2_1\delta}{Nr}}
+O(|(u_1,\delta)|^3)\\
\delta
 \end{array}\!\right],
\end{eqnarray*}
which defines a one-dimensional mapping
\begin{eqnarray*}
&&w_1\mapsto g_{1}(w_1):=u_1-{\frac { \left( {\beta}^{2}+2\,r\beta+{r}^{2} \right) u^2_1}{Nr}}+
u_1\,\delta-{\frac { \left( {\beta}^{3}+3\,r{\beta}^{2}+3\,{r}^{2}\beta+
{r}^{3} \right)u^3_1}{{N}^{2}r\beta}}\\
&&~~~~~~~~~~~~~~~~~~~~~~-{\frac { 2\,\left( \beta+r
 \right) u^2_1\,\delta}{Nr}}+O(|(u_1,\delta)|^3).
\end{eqnarray*}

One can check that
$$
\left.\frac{\partial^2g_{1}}{\partial w_1^2}\right|_{(u_1,\delta)=(0,0)}={\frac { -2\,\left( \beta+r \right) ^{2}}{Nr}}<0
$$
and
$$
\left.\frac{\partial^2 g_{1}}{\partial u_1 \partial \delta}\right|_{(w_1,\delta)=(0,0)}=1>0,
$$
implying that the non-degeneracy and transversality conditions
for transcritical bifurcation (see \cite[pp.504-508]{Wiggins}) are satisfied.
Therefore, mapping \eqref{eq2.1} undergoes a transcritical bifurcation as the parameter crossing $\alpha=\beta+r$.
This completes the proof.
\end{proof}

From the proof of Theorem \ref{th3.1}, we see that the transcritical bifurcation occurs on the center manifold, where the fixed point $E_1$ exchanges its stability, the other fixed point $E_2$ exists and disappears in both side of $\delta=0$,  which provides for the original mapping \eqref{eq2.1} a threshold, i.e., $N>0$, $0<\beta<1$, $0<r<1$ and $\alpha=\beta+r$ of a transit between stable
and unstable.
The threshold gives a scheme controlling the numbers of susceptible, infective and recovered individuals:
1) if $N>0$, $0<\beta<1$, $0<r<1$ and $\alpha<\beta+r$, then the number of susceptible individuals will gradually tend to $N$, the numbers of infective and recovered individuals will gradually tend to $0$. This implies that infectious diseases will not break out.
2) for $N>0$, $0<\beta<1$, $0<r<1$ and $\alpha<\beta+r$, if the initial number of susceptible individuals is near $N$, the initial numbers of infective and recovered individuals are near $0$, then the number of susceptible individuals will decrease to $ N \left( \beta+r \right)/\alpha$, and then remains stable; the numbers of infective and recovered individuals will increase to $\beta\,N \left( \alpha-\beta-r \right)/\alpha\,\left( \beta+r \right)$ and $rN \left( \alpha-\beta-r \right) /\alpha\,\left( \beta+r\right)$, respcetively, and then become stable. It implies that infectious diseases can be well controlled.

\subsection{Flip bifurcation of $E_2$}
\allowdisplaybreaks[4]
In subsection 3.1, we have revealed that transcritical can occur near the fixed point $E_1$. In this subsection, we will find that
mapping \eqref{eq2.1} can also undergo a flip bifurcation near $E_2$ in cases $\mathfrak{L}_{11}$ and $\mathfrak{L}_{13}$.

\begin{thm}
For the case $N>0$, $0<\beta<1$ and $0<r<\Psi_1$,
if
\begin{eqnarray*}
&&\Theta_1:=-{\frac { \left( \beta+r-2 \right) ^{2}\beta}{ \left( \beta+r \right)
 \left( {\beta}^{2}+ \left( r-4 \right) \beta-4\,r+4 \right) }}\neq0, \\
&&\Theta_2:={\frac {2\, \left( \beta+r \right)  \left( \beta+r-2 \right)  \left( {
\beta}^{2}+\beta\,r-4 \right) ^{2} \left( {\beta}^{2}+\beta\,r-2\,
\beta-r \right) }{{\beta}^{2}{r}^{2}{N}^{2} \left( {\beta}^{2}+\beta\,
r-4\,\beta-4\,r+4 \right) }}\neq0,
\end{eqnarray*}
then as the parameter $\alpha$ crosses $\Psi_2$,
mapping \eqref{eq2.1} undergoes a flip bifurcation near the fixed point $E_2$ (The expressions of $\Psi_1$ and $\Psi_2$ here can be found on page 6). More concretely, if $\Theta_2>0$ {\rm(}resp. $\Theta_2<0${\rm)}, then mapping \eqref{eq2.1} undergoes a supercritical {\rm(}resp. subcritical{\rm)} flip bifurcation and produces a stable {\rm(}resp. unstable{\rm)} period-two cycle when $\alpha$ crosses $\Psi_2$ from the region $\mathfrak{D}_{12}$ {\rm(}resp. $\mathfrak{D}_{31}${\rm)} to $\mathfrak{D}_{31}$ {\rm(}resp. $\mathfrak{D}_{12}${\rm)}.
\label{th4.1}
\end{thm}
\begin{proof}
Let
$$\vartheta:=\alpha-\Psi_2.
$$
As done in the proof of Theorem \ref{th3.1}, by translating $E_{2}$ to the origin $O$ and restricting mapping \eqref{eq2.1} into a two dimensional $C^{2}$ center manifold
\begin{eqnarray*}
\begin{array}{ll}
u_2=h_{21}(v_2,\vartheta)=O(|(v_2,\vartheta)|^3),\\
w_2=h_{22}(v_2,\vartheta)={\frac { \left( \beta+r-2 \right)  \left( {\beta}^{3}+2\,{\beta}^{2}r+
\beta\,{r}^{2}-2\,{\beta}^{2}-2\,\beta\,r-4\,\beta-4\,r+8 \right) \,v^2_2}{
\beta\,N \left( {\beta}^{2}+\beta\,r-4\,\beta-4\,r+4 \right) r}}+O(|(v_2,\vartheta)|^3),
\end{array}
\end{eqnarray*}
we obtain a one-dimensional mapping
{\fontsize{10pt}{11pt}\selectfont
\begin{eqnarray*}
&&\!\!\!\!\!\!\!\!\!\!\!\!\!\!\!v_3\mapsto g_{2}(v_2):=-v_2-{\frac { \left( \beta-2 \right)  \left( \beta+r-2 \right)  \left( {
\beta}^{2}+\beta\,r-4 \right)  \left( \beta+r \right)\,v^2_2 }{Nr\beta\,
 \left( {\beta}^{2}+\beta\,r-4\,\beta-4\,r+4 \right) }}\\
&&~~~~~~~~~~~~-{\frac { \left( \beta+r-2 \right) ^{2}\beta\,v_2\,\vartheta}{ \left( \beta+r \right)
 \left( {\beta}^{2}+ \left( r-4 \right) \beta-4\,r+4 \right) }}-{\frac { \left( \beta+r \right)  \left( \beta+r-2 \right)  \left( {
\beta}^{2}+\beta\,r-4 \right) ^{3}\,v^3_2}{ \left( {\beta}^{2}+\beta\,r-4\,
\beta-4\,r+4 \right) ^{2}{\beta}^{2}r{N}^{2}}}\\
&&~~~~~~~~~~~~-{\frac {\mathfrak{F}}{N \left( {
\beta}^{2}+ \left( r-4 \right) \beta-4\,r+4 \right) ^{3}r}}-{\frac { \left( \beta+r-2 \right) ^{3}r{\beta}^{2}\,v_2\,\vartheta^2}{ \left( \beta+r
 \right)  \left( {\beta}^{2}+ \left( r-4 \right) \beta-4\,r+4 \right)
^{3}}}+O(|(v_2,\vartheta)|^4),
\label{SIR02-1}
\end{eqnarray*}}where
{\fontsize{10pt}{11pt}\selectfont
\begin{eqnarray*}
&&\mathfrak{F}:=(( r-4 ) {\beta}^{4}+ ( 2\,{r}^{2}-8\,r+32 ) {\beta}^{3}+ ( {r}^{3}-4\,{r}^{2}+48\,r-96) {\beta}^{2}+( 16\,{r}^{2}-128\,r+128) \beta-32\,
{r}^{2}\\
&&~~~~~~~+112\,r-64 ) ( \beta+r-2 ) ^{2}\,v^2_2\,\vartheta.
\end{eqnarray*}}It can be calculated that
{\fontsize{10pt}{11pt}\selectfont
\begin{eqnarray*}
\left.\frac{\partial^2g_{2}}{\partial v_2 \partial\vartheta}\right|_{(v_2,\vartheta)=(0,0)}=-{\frac { \left( \beta+r-2 \right) ^{2}\beta}{ \left( \beta+r \right)
 \left( {\beta}^{2}+ \left( r-4 \right) \beta-4\,r+4 \right) }}
\end{eqnarray*}}and
{\fontsize{10pt}{11pt}\selectfont
\begin{eqnarray}\label{SIR02-2}
\!\!\!\left.{\left(\frac{1}{2}\,\left(\frac{\partial^2g_{2}}{\partial v_2^2}\right)^2+\frac{1}{3}\,\frac{\partial^3g_{2}}{\partial v_2 ^3}\right)}\right|_{(v_2,\vartheta)=(0,0)}={\frac { 2\,\left( \beta+r \right)  \left( \beta+r-2 \right)  \left( {
\beta}^{2}+\beta\,r-4 \right) ^{2} \left( {\beta}^{2}+\beta\,r-2\,
\beta-r \right) }{{\beta}^{2}{r}^{2}{N}^{2} \left( {\beta}^{2}+\beta\,
r-4\,\beta-4\,r+4 \right) }}.
\end{eqnarray}}

Since $\Theta_2\neq0$, the non-degeneracy and transversality conditions of Theorem 3.5.1 in \cite[pp.158]{Guckenheimer} are satisfied, implying that the mapping $g_{2}$ undergoes a flip bifurcation. Therefore, mapping \eqref{eq2.1} undergoes a flip bifurcation as the  parameters cross $\alpha=\Psi_2$. Specifically, if $\Theta_2>0$, then mapping \eqref{eq2.1} undergoes a supercritical flip bifurcation and produces a stable period-two cycle from the region $\mathbf{D}_{12}$ to $\mathbf{D}_{22}$ and from the region $\mathbf{D}_{24}$ to $\mathbf{D}_{31}$ respectively. However, if $\Theta_2<0$, then the period-two cycle located in $\mathbf{D}_{12}$ and $\mathbf{D}_{24}$ is unstable.
In fact, in the case of $\Theta_2>0$, mapping \eqref{SIR02-1} is topologically equivalent to the following mapping
\begin{eqnarray}
\begin{array}{l}
v_2\mapsto g^*_{2}(v_2):=-v_2-{\frac { \left( \beta-2 \right)  \left( \beta+r-2 \right)  \left( {
\beta}^{2}+\beta\,r-4 \right)  \left( \beta+r \right)\,v^2_2 }{Nr\beta\,
 \left( {\beta}^{2}+\beta\,r-4\,\beta-4\,r+4 \right) }}
 \\
~~~~~~~~~~~~~~~~~~~~-{\frac { \left( \beta+r-2 \right) ^{2}\beta\,v_2\,\vartheta}{ \left( \beta+r \right)
 \left( {\beta}^{2}+ \left( r-4 \right) \beta-4\,r+4 \right) }}-{\frac { \left( \beta+r \right)  \left( \beta+r-2 \right)  \left( {
\beta}^{2}+\beta\,r-4 \right) ^{3}\,v^3_2}{ \left( {\beta}^{2}+\beta\,r-4\,
\beta-4\,r+4 \right) ^{2}{\beta}^{2}r{N}^{2}}}
\\
~~~~~~~~~~~~~~~~~~~~-{\frac {\mathfrak{F}}{N \left( {
\beta}^{2}+ \left( r-4 \right) \beta-4\,r+4 \right) ^{3}r}}
-{\frac { \left( \beta+r-2 \right) ^{3}r{\beta}^{2}\,v_2\,\vartheta^2}{ \left( \beta+r
 \right)  \left( {\beta}^{2}+ \left( r-4 \right) \beta-4\,r+4 \right)
^{3}}}.
\label{SIR-3}
\end{array}
\end{eqnarray}
We re-expand $g^{\ast}_{2}$ in Taylor series with $v_2$ and then obtain
\begin{eqnarray}
g^{\ast}_{2}(v_2)=f_1(\vartheta)\,v_2+f_2(\vartheta)\,v_2^2+f_3(\vartheta)\,v_2^3+O(v_2^4),
\label{eq05261}
\end{eqnarray}
where $f_1,f_2$ and $f_3$ are $C^1$ functions, satisfying
{\fontsize{11pt}{11pt}\selectfont
\begin{eqnarray*}
&&f_1(\vartheta):=-(1+g(\vartheta)),\\
&&g(\vartheta):={\frac { \left( \beta+r-2 \right) ^{2}\beta\,\vartheta}{ \left( \beta+r \right)
 \left( {\beta}^{2}+ \left( r-4 \right) \beta-4\,r+4 \right) }}+{\frac { \left( \beta+r-2 \right) ^{3}r{\beta}^{2}\,\vartheta^2}{ \left( \beta+r
 \right)  \left( {\beta}^{2}+ \left( r-4 \right) \beta-4\,r+4 \right)
^{3}}},\\
&&f_2(\vartheta):=-{\frac { \left( \beta-2 \right)  \left( \beta+r-2 \right)  \left( {
\beta}^{2}+\beta\,r-4 \right)  \left( \beta+r \right) }{Nr\beta\, \left( {\beta}^{2}+\beta\,r-4\,\beta-4\,r+4 \right) }}-{\frac { \mathfrak{F}}{N \left( {
\beta}^{2}+ \left( r-4 \right) \beta-4\,r+4 \right) ^{3}r}},\\
&&f_3(\vartheta):=-{\frac { \left( \beta+r \right)  \left( \beta+r-2 \right)  \left( {
\beta}^{2}+\beta\,r-4 \right) ^{3}}{ \left( {\beta}^{2}+\beta\,r-4\,
\beta-4\,r+4 \right) ^{2}{\beta}^{2}r{N}^{2}}}.
\end{eqnarray*}}Then $g(0)=0$ and
$$
\left.\frac{\mathrm{d}g}{\mathrm{d}\vartheta}\right|_{\vartheta=0}=-\left.\frac{\partial^2g_2}{\partial w_2\partial\vartheta}\right|_{(w_2, \vartheta)=(0,0)}={\frac { \left( \beta+r-2 \right) ^{2}\beta}{ \left( \beta+r \right)
 \left( {\beta}^{2}+ \left( r-4 \right) \beta-4\,r+4 \right) }},
$$
it follows that the function $g$ is locally invertible. Hence,
mapping \eqref{eq05261} can be changed into
$$
\tilde{w}_2=\kappa(\vartheta_1)+c(\vartheta_1)\,w_2^2+d(\vartheta_1)\,w_2^3+O(w_2^4)
$$
by setting $\vartheta_1:=g(\vartheta)$, where the functions $\kappa(\vartheta_1)=-(1+\vartheta_1)$, $c(\vartheta_1)$ and $d(\vartheta_1)$ are all $C^1$.
Consequently, we get
\begin{eqnarray*}
c(0)=f_2(0)=\frac{1}{2}\,\frac{\partial^2g_{2}}{\partial w_2^2}, ~~d(0)=f_3(0)=\frac{1}{6}\frac{\partial^3g_{2}}{\partial w_2 ^3}.
\label{eq05162}
\end{eqnarray*}

Furthermore, we employ a smooth change of coordinate
\begin{eqnarray}
w_2=w_6+\vartheta_2w_6^3,
\label{eq05263}
\end{eqnarray}
where $\vartheta_2=\vartheta_2(\vartheta_1)$ is a $C^1$ function to be determined. Clearly, the transformation \eqref{eq05263} is invertible in a sufficiently small neighborhood of the origin $O$, and its inverse can be obtained by the method of undetermined coefficients as
\begin{eqnarray}
w_6=w_2-\vartheta_2w_2^2+2\vartheta_2^2w_2^3+O(w_2^4).
\label{eq05264}
\end{eqnarray}
Using \eqref{eq05263}-\eqref{eq05264}, we get
{\fontsize{11pt}{11pt}\selectfont
\begin{eqnarray}\label{karpa}
\!\!\!\!\!\tilde{w}_6=\kappa w_6+(c+\vartheta_2\kappa-\vartheta_2\kappa^2)w_6^2+(d+2\vartheta_2c-2\vartheta_2\kappa(\vartheta_2\kappa+c)+2\vartheta_2^2\kappa^3)w_6^3+O(w_6^4).
\end{eqnarray}}
\!Thus, the quadratic term of (\ref{karpa}) can be cancelled via sufficiently small $|\vartheta_1|$ by setting
$$
\vartheta_2(\vartheta_1):=\frac{c(\vartheta_1)}{\kappa^2(\vartheta_1)-\kappa(\vartheta_1)}.
$$
Actually, it can be done because $\kappa^2(0)-\kappa(0)=2\neq0$. Then (\ref{karpa}) becomes
\begin{eqnarray*}
\tilde{w}_6=\kappa w_6+(d+\frac{2c^2}{\kappa^2-\kappa})w_6^3+O(w_6^4)=-(1+\vartheta_1)w_6+e(\vartheta_1)w_6^3+O(w_6^4)
\end{eqnarray*}
for certain $C^1$ function $e$ such that
$$
e(0)=c^2(0)+d(0)=\frac{1}{4}\,\left.\left(\frac{\partial^2g_{2}}{\partial w_2^2}\right)^2\right|_{(w_2,\vartheta)=(0,0)}+\frac{1}{6}\,\left({\left.\frac{\partial^3g_{2}}{\partial w_2 ^3}\right)}\right|_{(w_2,\vartheta)=(0,0)}.
$$
From \eqref{SIR02-2}, $e(0)>0$. Rescaling $w_6=w_7/\sqrt{|e(\vartheta_1)|}$ we get
\begin{eqnarray}
w_7\mapsto g_{2\tau}(w_7):=-(1+\vartheta_1)w_7+w_7^3+O(w_7^4).
\label{eq05265}
\end{eqnarray}
Hence, from \cite[Theorem~4.4, p.129]{Kuznetsov}, mapping \eqref{eq05265} is topologically equivalent to the following form
\begin{eqnarray}
w_7\mapsto g_{2*}(w_7)=-(1+\vartheta_1)w_7+w_7^3,
\label{eq05266}
\end{eqnarray}
near the origin $O$.
Finally, we consider the second iterate $g_{2*}$ of mapping \eqref{eq05266} and obtain
\begin{eqnarray*}
\!\!\!\!\!\!\!\!\!\!g_{2*}^2(w_7)=\left( -\vartheta_1-1 \right) ^{2}w_7+ \left(  \left( -\vartheta_1-1 \right)\left( 1+\vartheta_1 \right) ^{2}-\vartheta_1-1 \right) w_7^{3}+O(w_7^5).
\end{eqnarray*}
It is clear that $g_{2*}^2$ has a trivial fixed point $w_{70}=0$ and two nontrivial fixed points
$$
w_{71,72}=\pm\sqrt{\vartheta_1}
$$
for $0<\vartheta_1\ll1$.
One can check that $|Dg_{2*}^2(w_{71})|<1$ (resp. $|Dg_{2*}^2(w_{72})|<1$).
Hence, $w_{71}$ and $w_{72}$ are stable and constitute a stable cycle of period-two for the original mapping $g_{2*}$.
It implies that the mapping \eqref{eq05266} generates a stable period-two cycle near origin $O$ for $0<\vartheta_1\ll1$. Furthermore, from $\vartheta_1:=g(\vartheta)$, we can see that the mapping \eqref{SIR-3} generates a stable period-two cycle near the fixed point $E_1$ for $-1\ll\vartheta<0$.
Therefore, by the parameter transformation $\vartheta:=\alpha-\Psi_2$, mapping \eqref{eq2.1} undergoes a supercritical flip bifurcation and produces a stable period-two cycle when $\alpha$ crosses $\Psi_2$ from the region $\mathfrak{D}_{12}$ to $\mathfrak{D}_{31}$. In the same way, when $\Theta_2<0$ one can also prove that mapping \eqref{eq2.1} undergoes a subcritical flip bifurcation and generates an unstable period-two cycle when $\alpha$ crosses $\Psi_2$ from the region $\mathfrak{D}_{31}$ to $\mathfrak{D}_{12}$.
This completes the proof.
\end{proof}
Using the same idea as done in Theorem \ref{th4.1}, one can obtain
\begin{thm}
For the case $N>0$, $0<\beta<1$ and $0<r<\Psi_1$,
if
\begin{eqnarray*}
&&\Theta_3:=-{\frac {\beta\, \left( r+\beta-2 \right) ^{2}}{ \left( {\beta}^{2}+
 \left( r-4 \right) \beta-4\,r+4 \right)  \left( \beta+r \right) }}\neq0, \\
&&\Theta_4:={\frac {2\, \left( \beta+r \right)  \left( \beta+r-2 \right)  \left( {
\beta}^{2}+\beta\,r-2\,\beta-r \right)  \left( {\beta}^{2}+\beta\,r-4
 \right) ^{2}}{{N}^{2}{r}^{2}{\beta}^{2} \left( {\beta}^{2}+\beta\,r-4
\,\beta-4\,r+4 \right) }}
\neq0,
\end{eqnarray*}
then mapping \eqref{eq2.1} undergoes a flip bifurcation near the fixed point $E_2$ as the parameter $\alpha$ crosses $\Psi_2$. More specifically,
if $\Theta_5>0$ {\rm(}resp. $\Theta_5<0${\rm)}, then mapping \eqref{eq2.1} undergoes a supercritical {\rm(}resp. subcritical{\rm)} flip bifurcation and produces a stable {\rm(}resp. unstable{\rm)} period-two cycle when $\alpha$ crosses $\Psi_2$ from the region $\mathfrak{D}_{34}$ {\rm(}resp. $\mathfrak{D}_{33}${\rm)} to $\mathfrak{D}_{33}$ {\rm(}resp. $\mathfrak{D}_{34}${\rm)}.
\label{th4.2}
\end{thm}

According to Theorem \ref{th4.1}, As the parameters cross $\alpha=\Psi_2$, mapping \eqref{eq2.1} undergoes a flip bifurcation, which implies that the numbers of susceptible, infective and recovered individuals will coexist in a state of fluctuations with periods of 2, 4, and 8 $\cdots$.
When $\Theta_2$ is greater than 0, this periodic fluctuation is in a stable state; when $\Theta_2$ is less than 0, the periodic fluctuation is in an unstable state.
However, whether the fluctuations are stable or not, as the parameters changes, the numbers of susceptible, infected and recovered individuals will gradually evolve into a chaotic state through periodic fluctuations, leading to the loss of disease control.
Therefore, from a biological perspective, once it is detected that the numbers of susceptible, infected and recovered individuals exhibit changes with periods of 2, 4, and 8 $\cdots$, the parameters should be adjusted in a timely manner to prevent the numbers of susceptible, infected and recovered individuals from experiencing larger-period changes, thus preventing the disease from entering an uncontrollable situation. Similar biological implications can also be derived from Theorem \ref{th4.2}.


\subsection{Neimark-Sacker bifurcation of $E_{2}$}
\allowdisplaybreaks[4]
In this section, we consider the case that the eigenvalues of $JF(E_2)$ are a pair of conjugate complex numbers and lie on the unit disk. Then, another codimensional $1$ bifurcation (i.e., Neimark-Sacker bifurcation, see \cite{Neimark,Sacker})of mapping \eqref{eq2.1} near the fixed point $E_2$ will be studied.
From \eqref{solu-1}, in case
\begin{eqnarray*}
\begin{array}{llll}
\mathfrak{L}_2:=\{
N>0, 0<\beta<1, \Psi_1<r<1, \alpha=\Psi_3
\},
\end{array}
\end{eqnarray*}
the two eigenvalues of $JF(E_2)$ are
\begin{eqnarray*}
\begin{array}{llll}
\!\!\!\!\!\!\!\!\!\!\!\!\!t_{1}={\frac {-{\beta}^{2}-\beta\,r+2\,\beta+2\,r-2+{\bf{i}}\,\sqrt {-\beta\, \left(
{\beta}^{2}+ \left( r-4 \right) \beta-4\,r+4 \right)  \left( \beta+r
 \right) }}{2\,\beta+2\,r-2}},\\
\noalign{\medskip}
\!\!\!\!\!\!\!\!\!\!\!\!\!t_{2}={\frac {-{\beta}^{2}-\beta\,r+2\,\beta+2\,r-2-{\bf{i}}\,\sqrt {-\beta\, \left(
{\beta}^{2}+ \left( r-4 \right) \beta-4\,r+4 \right)  \left( \beta+r
 \right) }}{2\,\beta+2\,r-2}}
,
\end{array}
\end{eqnarray*}
where ${\bf i}$ denotes the imaginary unit.
Hence, when the parameters cross $\mathfrak{L}_2$, mapping \eqref{eq2.1} may have a Neimark-Sacker bifurcation.
Actually, we have the following result.
\begin{thm}
If the parameters cross $\mathfrak{L}_2$ with
\begin{eqnarray*}
&&r\neq-{\frac {{\beta}^{2}-3\,\beta+3}{\beta-3}},-{\frac {{\beta}^{2}-2\,\beta+2}{\beta-2}}, 1-\beta~~\mbox{and}~~-\frac{\beta^2-1}{\beta},
\end{eqnarray*}
then mapping \eqref{eq2.1} undergoes a Neimark-Sacker bifurcation near $E_2$. More specifically, if the parameters $\beta$ and $r$ lie in the set
\begin{eqnarray*}
&&\mathfrak{E}_1:=\{0<\beta<\frac{4}{5},~~\Psi_1<r<-\frac{\beta^2-1}{\beta}\},
\end{eqnarray*}
\vskip -1.2cm
\begin{eqnarray*}
&&\!\!\!\!\!\!\!\!\!\!\!\!\!({\rm resp.}~\mathfrak{E}_2:=\{\frac{\sqrt{5}}{2}-1<\beta<\frac{4}{5},~~\Psi_1<r<1\}\cup\{N>0,~~\frac{4}{5}\leq\beta<1,~~\Psi_1<r<1\}),
\end{eqnarray*}
i.e.,
the first Lyapunov quantity $\mathfrak{A}>0$
(resp. $\mathfrak{A}<0$),
where
\begin{eqnarray*}
&&\mathfrak{A}:=-{\frac { \left( {\beta}^{2}+\beta\,r-1 \right)  \left( \beta+r
 \right) ^{4}}{8\,{N}^{2} \left( \beta+r-1 \right) }},
\end{eqnarray*}
mapping \eqref{eq2.1} undergoes a subcritical (resp. supercritical) Neimark-Sacker bifurcation and produces a unique unstable (resp. stable) invariant circle surrounding $E_2$ when $\alpha$ crosses $\Psi_3$ from the region $\mathfrak{D}_{4}$ (resp. $\mathfrak{D}_{23}$) to $\mathfrak{D}_{23}$ (resp. $\mathfrak{D}_{4}$).
\label{th5.1}
\end{thm}
\begin{proof}
Let
{\footnotesize
\begin{eqnarray*}
&&\!\!\!\!\!\!\!\!\!\!\!\!\!\!\!\!\!\!\!\mu_0:=t_{1}={\frac {-{\beta}^{2}-\beta\,r+2\,\beta+2\,r-2+{\bf{i}}\,\sqrt {-\beta\, \left(
{\beta}^{2}+ \left( r-4 \right) \beta-4\,r+4 \right)  \left( \beta+r
 \right) }}{2\,\beta+2\,r-2}}.
\end{eqnarray*}
}In order to obtain the conclusion in the theorem, we now verify the conditions (C.1), (C.2) and (C.3) offered in Theorem 4.6 of \cite[p.141]{Kuznetsov}. We first check the nondegeneracy condition (C.2), i.e. $\mu_0^{i}\ne 1$ for $i=1,2,3,4$. Since $\mu_0$ is not a real number, then the inequality $\mu_0^{i}\ne 1$ for $i=1,2$ is true. Further, we claim that $\mu_0^{3}\neq1$. Otherwise, from $(\mu_0-1)(\mu_0^2+\mu_0+1)=0$ we get $\mu_0^2+\mu_0+1=0$, implying $$\mu_0=-\frac{1}{2}+\frac{\sqrt{3}}{2}\,{\bf{i}} ~~~\mbox{or}~~~ \mu_0=-\frac{1}{2}-\frac{\sqrt{3}}{2}\,{\bf{i}}, $$
and thus $r=-({\beta}^{2}-3\,\beta+3)/(\beta-3)$, which contradicts to the conditions in Theorem \ref{th5.1}. Then, in the same way, suppose that $\mu_0^{4}=1$, which is equivalent to $\mu_0^2+1=0$, implying the real part of $\mu_0$,
denoted as $\mathfrak{R}(\mu_0)$, is equal to $0$. This means $r=-({\beta}^{2}-2\,\beta+2)/(\beta-2)$, a contradiction to the conditions in Theorem \ref{th5.1}. Therefore, the nondegeneracy condition is proved.

In addition, if the parameters is near $\mathfrak{L}$ but not on $\mathfrak{L}$, then from (\ref{solu-1}) we obtain
$$
|\mu_0|=\sqrt{{\frac {-{\beta}^{3}+ \left( \alpha-2\,r \right) {\beta}^{2}- \left( r
-1 \right)  \left( r-\alpha+1 \right) \beta+r}{\beta+r}}},
$$
which yields
\begin{eqnarray*}
\left.\frac{d|\mu_0|}{d\,\alpha}\right|_{\alpha={\frac {{\beta}^{2}+2\,\beta\,r+{r}^{2}}{\beta+r-1}}}={\frac {\beta\, \left( \beta+r-1 \right) }{2\,\beta+2\,r}}.
\end{eqnarray*}
Combining the condition $r\neq1-\beta$, the transversality condition (C.1) of Theorem 4.6 in \cite[p.141]{Kuznetsov}
is satisfied.

Next, we analyze the non-degeneracy condition (C.3). Firstly,
applying the transformation
$$(x,y,z)=(u_3+{\frac {N \left( \beta+r \right) }{\alpha}},v_3+{\frac {\beta\,N \left( \alpha-\beta-r \right) }{\alpha\, \left( \beta
+r \right) }},w_3+{\frac {rN \left( \alpha-\beta-r \right) }{\alpha\, \left( \beta+r
 \right) })},$$
we translate the fixed point $E_{2}$ to the origin $O$, and let
$$\alpha={\frac {{\beta}^{2}+2\,\beta\,r+{r}^{2}}{\beta+r-1}}.$$
Secondly, by employing the invertible linear transformation $(u_4,v_4,w_4)^T=H_1(u_5,v_5,w_5)^T$, where
\begin{eqnarray*}
H_1=\left[ \begin {array}{ccc}
1&0&0\\
0&1&0\\
-1&-1&1
\end {array}
 \right],
\end{eqnarray*}
mapping \eqref{eq2.1} is converted into the following form
\begin{eqnarray}
\left[
\begin{array}{ccc}
   u_5  \\
   \noalign{\medskip}
   v_5\\
   \noalign{\medskip}
   w_5
\end{array}
\right]
\mapsto
\left[\begin{array}{cc}
-{\frac { \left( {\beta}^{2}+\beta\,r-\beta-r+1 \right)\,u_{{5}}}{\beta
+r-1}}- \left( \beta+r \right)\,v_{{5}} -{\frac { \left(
\beta+r \right) ^{2}\,u_{{5}}\,v_{{5}}}{N \left( \beta+r-1 \right) }}\\
\noalign{\medskip}
  {\frac {\beta\,u_{{5}}}{\beta+r-1}}+v_{{5}}+{\frac { \left( {\beta}^{2
}+2\,\beta\,r+{r}^{2} \right)\, u_{{5}}\,v_{{5}}}{N \left( \beta+r-1
 \right) }} \\
\noalign{\medskip}
- \left( \beta-1 \right) w_{{5}}
 \end{array}\right].
\label{eq4.1}
\end{eqnarray}
As done in proof of Theorem \ref{th3.1}, mapping \eqref{eq4.1} is restricted into a two dimensional $C^{2}$ center manifold
\begin{eqnarray*}
w_5=h_{3}(u_5,v_5)=O(|(u_5,v_5)|^3),
\end{eqnarray*}
and then it is converted into the following two dimensional form
\begin{eqnarray}
\left[
\begin{array}{ccc}
   u_5  \\
   \noalign{\medskip}
   v_5
\end{array}
\right]
\mapsto
\left[\begin{array}{cc}
-{\frac { \left( {\beta}^{2}+\beta\,r-\beta-r+1 \right)\,u_{{5}}}{\beta
+r-1}}- \left( \beta+r \right)\,v_{{5}} -{\frac { \left(
\beta+r \right) ^{2}\,u_{{5}}\,v_{{5}}}{N \left( \beta+r-1 \right) }}\\
\noalign{\medskip}
  {\frac {\beta\,u_{{5}}}{\beta+r-1}}+v_{{5}}+{\frac { \left( {\beta}^{2
}+2\,\beta\,r+{r}^{2} \right)\, u_{{5}}\,v_{{5}}}{N \left( \beta+r-1
 \right) }}
 \end{array}\right].
\label{eq4.2}
\end{eqnarray}
Applying the invertible linear transformation $(u_4,v_4)^T=H_2(\mu,\nu)^T$ for mapping \eqref{eq4.2}, where
\begin{eqnarray*}
H_2=\left[ \begin {array}{cc}
-\frac{\beta+r}{2}   &   -{\frac {\sqrt {-\beta\, \left( {\beta}^{3}+2\,{\beta}^{2}r+\beta
\,{r}^{2}-4\,{\beta}^{2}-8\,\beta\,r-4\,{r}^{2}+4\,\beta+4\,r \right)
}}{2\,\beta}} \\
\noalign{\medskip}
1   &     0
\end {array} \right],
\end{eqnarray*}
we obtain
{\small\begin{eqnarray}
\left[
\begin{array}{cc}
   \mu\\
   \\
   \nu
\end{array}
\right]
\mapsto
\left[\begin{array}{cc}
a_{1,0}\mu_{{1}}+a_{0,1}\nu_{{1}}+a_{2,0}\mu_1^2+a_{1,1}\mu_{{1}}\nu_{{1}}+a_{0,2}\nu_1^2
\\
   \\
b_{1,0}\mu_{{1}}+b_{0,1}v\nu_{{1}}+b_{2,0}\mu_1^2+b_{1,1}\mu_{{1}}\nu_{{1}}+b_{0,2}\nu_1^2
 \end{array}\right],
\label{eq4.3}
\end{eqnarray}}
where
\begin{eqnarray*}
&&a_{1,0}=b_{0,1}:=-{\frac {{\beta}^{2}+\beta\,r-2\,\beta-2\,r+2}{2\,\beta+2\,r-2}},a_{2,0}:=-{\frac {{\beta}^{3}+3\,{\beta}^{2}r+3\,\beta\,{r}^{2}+{r}^{3}}{2\,N \left( \beta+r-1 \right) }},
\\
&&a_{0,1}=-b_{1,0}:=-{\frac {\sqrt {-\beta\, \left( {\beta}^{2}+ \left( r-4 \right) \beta-
4\,r+4 \right)  \left( \beta+r \right) }}{2\,\beta+2\,r-2}},
\\
&&a_{1,1}:=-{\frac { \left( {\beta}^{2}+2\,\beta\,r+{r}^{2} \right) \sqrt {-
\beta\, \left( {\beta}^{2}+ \left( r-4 \right) \beta-4\,r+4 \right)
 \left( \beta+r \right) }}{2\,\beta\,N \left( \beta+r-1 \right) }},
\\
&&b_{1,1}:={\frac {{\beta}^{3}+3\,{\beta}^{2}r+3\,\beta\,{r}^{2}+{r}^{3}-2\,
{\beta}^{2}-4\,\beta\,r-2\,{r}^{2}}{2\,N \left( \beta+r-1 \right) }}, a_{0,2}=b_{0,2}=0,\\
&&b_{2,0}:={\frac { \left( \beta+r \right)  \left( {\beta}^{3}+3\,{\beta}^{2
}r+3\,\beta\,{r}^{2}+{r}^{3}-2\,{\beta}^{2}-4\,\beta\,r-2\,{r}^{2}
 \right) \beta}{2\,\sqrt {-\beta\, \left( {\beta}^{2}+ \left( r-4
 \right) \beta-4\,r+4 \right)  \left( \beta+r \right) }N \left( \beta+
r-1 \right) }}.
\end{eqnarray*}
Let $\omega:=\mu+{\bf i}\nu$. Consider $\omega$ as the variable in Taylor expansion, mapping \eqref{eq4.3} is transformed into the following complex form
\begin{equation}
\omega\mapsto \mu_0\omega+\sum_{i=0}^2\gamma_{2-i,i}\omega^{2-i}\bar{\omega}^i,
\label{eq4.4}
\end{equation}
where
\begin{eqnarray*}
\!\!\!\!&&\gamma_{2,0}:=-{\frac { \left( \beta+r \right) ^{2}}{4\,N \left( \beta+r-1 \right) }}
+{\frac {\left( \beta+r \right) ^{3} \left( 2r+\beta-2 \right)\,{\bf{i}}}{4\,\sqrt {-\beta\, \left( {\beta}^{2}+ \left( r-4 \right) \beta-4\,r+4
 \right)  \left( \beta+r \right) }N \left( \beta+r-1 \right) }},\\
\!\!\!\!&&\gamma_{1,1}:=-{\frac { \left( \beta+r \right) ^{3}}{4\,N \left( \beta+r-1 \right) }}+{\frac {\beta\, \left( \beta+r-2 \right)  \left( \beta+r \right) ^{3}\,{\bf{i}}}{4\,\sqrt {-\beta\, \left( {\beta}^{2}+ \left( r-4 \right) \beta-4\,r+4 \right)  \left( \beta+r \right) }N \left( \beta+r-1 \right) }}
,\\
\!\!\!\!&&\gamma_{0,2}:=-{\frac { \left( \beta+r \right) ^{2}}{4\,N}}
+{\frac { \left( \beta-2 \right)  \left( \beta+r \right) ^{3}\,{\bf{i}}}{4\,\sqrt {-\beta\, \left( {\beta}^{2}+ \left( r-4 \right) \beta-4\,r+4
 \right)  \left( \beta+r \right) }N}}.\\
\end{eqnarray*}
Employing the quadratic near-identity transformation
$$
\omega=\omega+q_{2,0}\omega^2+q_{1,1}\omega\bar{\omega}+q_{0,2}\bar{\omega}^2,
$$
where $q_{2,0}$, $q_{1,1}$ and $q_{0,2}$ are undetermined coefficients,
mapping \eqref{eq4.4} becomes the following form
\begin{eqnarray}
\begin{array}{ll}
\omega\mapsto \mu_0\omega+(\gamma_{2,0}-(\mu_0^2-\mu_0)\,q_{2,0})\omega^2
+(\gamma_{1,1}-(\bar{\mu}_0\mu_0-\mu_0)\,q_{1,1})\omega\bar{\omega}
\\~~~~~~
+(\gamma_{0,2}-(\bar{\mu}_0^2-\mu_0)\,q_{0,2})\bar{\omega}^2+O(|\omega|^3).
\label{eq4.5}
\end{array}
\end{eqnarray}
We further select
$$
q_{2,0}:=\frac{\gamma_{2,0}}{\mu_0^2-\mu_0},~q_{1,1}:=\frac{\gamma_{1,1}}{\bar{\mu}_0\mu_0-\mu_0},
~q_{0,2}:=\frac{\gamma_{0,2}}{\bar{\mu}_0^2-\mu_0},
$$
and then mapping \eqref{eq4.5} is reduced to the mapping
\begin{eqnarray}
&&\omega\mapsto \mu_0\omega+p_{3,0}\omega^3+p_{2,1}\omega^2\bar{\omega}+p_{1,2}\omega\bar{\omega}^2+p_{0,3}\bar{\omega}^3+O(|\omega|^4),
\label{eq4.6}
\end{eqnarray}
where
\begin{eqnarray*}
&&\!\!\!\!\!\!\!p_{2,1}:=\frac{\mathfrak{S_1}}{\mathfrak{S_2}},\\
&&\!\!\!\!\!\!\!\mathfrak{S_1}:=( \beta+r ) ^{4} (  (  ( {\bf{i}}+r{\bf{i}} ) {
\beta}^{6}+ ( 3\,{\bf{i}}{r}^{2}-4\,{\bf{i}} ) {\beta}^{5}+ ( 3\,{\bf{i}}{r
}^{3}-6\,{\bf{i}}{r}^{2}-6\,{\bf{i}}r+4\,{\bf{i}} ) {\beta}^{4}+ ( r-1 ) ( {r}^{3}\\
&&~~~~~-7\,{r}^{2}-7\,r-3 ) {\beta}^{3}{\bf{i}}-3\,{\bf{i}} ( r-1 )  ( {r}^{3}+1/3\,{r}^{2}+r-2 ) {\beta}^{2}
+2\,{\bf{i}} ( r-1 ) ^{3}r ) \\
&&~~~~~\sqrt {-\beta\, ( {\beta}^{2}+ ( r-4 ) \beta-4\,r+4 )  ( \beta+r ) }- ( {\beta}^{2}+ ( r-4 ) \beta-4\,r+4 )  (\beta+r )\\
&&~~~~~(  ( r-1 ) {\beta}^{4}+ ( 2\,{r}^{2}-5\,r+2 ) {\beta}^{3}+ ( {r}^{3}-7\,{r}^{2} +6\,r ) {\beta}^{2}+ ( -3\,{r}^{3}+2\,{r}^{2}+4\,r-3 )\beta\\
&&~~~~~-2\,r ( r-1 ) ^{2} ) \beta ),\\
&&\!\!\!\!\!\!\!\mathfrak{S_2}:=( {\beta}^{2}+\beta\,r-4\,\beta-4\,r+4 ) {N}^{2} ( {\bf{i}}
\sqrt {-\beta\, ( {\beta}^{2}+ ( r-4 ) \beta-4\,r+4 )  ( \beta+r ) }-{\beta}^{2}-\beta\,r )\\
&&~~~~~ ( {\bf{i}}\sqrt {-\beta\, ( {\beta}^{2}+ ( r-4 ) \beta-4\,r+4 )  ( \beta+r ) }-{\beta}^{2}-\beta\,r+2\,\beta
+2\,r-2 ) \beta\, ( {\beta}^{2}+\beta\,r\\
&&~~~~~-3\,\beta-3\,r+3 )  ( {\bf{i}}\sqrt {-\beta\, ( {\beta}^{2}+ ( r-4 ) \beta-4\,r+4 )  ( \beta+r ) }+{\beta}^{2}+
\beta\,r ).
\end{eqnarray*}
Moreover, using the cubic near-identity transformation
$$
\omega=\theta+q_{3,0}\theta^3+q_{2,1}\theta^2\bar{\theta}+q_{1,2}\theta\bar{\theta}^2+q_{0,3}\bar{\theta}^3,
$$
where $q_{3,0}$, $q_{2,1}$, $q_{1,2}$ and $q_{0,3}$ are undetermined coefficients, mapping \eqref{eq4.6} is rewritten as the following form
\begin{eqnarray*}
&&\theta\mapsto\mu_0\theta+(p_{3,0}-(\mu_0^3-\mu)q_{3,0})\theta^3
+(p_{2,1}-(\mu_0|\mu_0|^2-\mu_0)q_{2,1})\theta^2\bar{\theta}
\\
&&~~~~~~+(p_{1,2}-(\bar{\mu}_0|\mu_0|^2-\mu_0)q_{1,2})\theta\bar{\theta}^2+
(p_{0,3}-(\bar{\mu}_0^3-\mu_0)q_{0,3})\bar{\theta}^3+O(|\theta|^4),
\end{eqnarray*}
which can be simplified as
\begin{equation}
\theta\mapsto\mu_0\theta+p_{2,1}\theta\bar{\theta}^2+O(|\theta|^4)
\label{eq4.7}
\end{equation}
by choosing
$$
q_{3,0}:=\frac{p_{3,0}}{\mu_0^3-\mu_0},~q_{2,1}:=0,~q_{1,2}:=\frac{p_{1,2}}{\bar{\mu}_0|\mu_0|^2-\mu_0},
~q_{0,3}:=\frac{p_{0,3}}{\bar{\mu}_0^3-\mu_0}.
$$
Furthermore, let $\mu_0:=e^{{\bf i}\varsigma}$, where
{\footnotesize
\begin{eqnarray*}
\varsigma=
\left\{
\begin{array}{l}
{\rm arctan}(-{\frac {\sqrt {-\beta\, \left( {\beta}^{2}+ \left( r-4 \right) \beta-
4\,r+4 \right)  \left( \beta+r \right) }}{{\beta}^{2}+ \left( r-2
 \right) \beta-2\,r+2}}),~~{\beta}^{2}+ \left( r-2 \right) \beta-2\,r+2>0,\\
\pi+{\rm arctan}(-{\frac {\sqrt {-\beta\, \left( {\beta}^{2}+ \left( r-4 \right) \beta-
4\,r+4 \right)  \left( \beta+r \right) }}{{\beta}^{2}+ \left( r-2
 \right) \beta-2\,r+2}}),~~{\beta}^{2}+ \left( r-2 \right) \beta-2\,r+2<0,
\end{array}
\right.
\end{eqnarray*}
}\!\!\!
and $\theta=\mathfrak{p}e^{{\bf i}\mathfrak{v}}$. Then, mapping \eqref{eq4.7} is changed into the form
$$
\theta=\mathfrak{p}e^{{\bf i}\mathfrak{v}}\mapsto\theta=\mathfrak{p}e^{{\bf i}(\mathfrak{v}+\varsigma)}
(1+\bar{\mu}_0p_{2,1}\mathfrak{p}^2)+O(\mathfrak{p}^4),
$$
which is equivalent to the mapping
\begin{eqnarray}
\left[
\begin{array}{cc}
  \mathfrak{p}  \\
   \mathfrak{v}
\end{array}
\right]
\mapsto
\left[\begin{array}{cc}
\mathfrak{p}|1+\bar{\mu}_0p_{2,1}\mathfrak{p}^2|+O(\mathfrak{p}^4)
\\
\mathfrak{v}+\varsigma+{\rm arg}(1+\bar{\mu}_0p_{2,1}\mathfrak{p}^2)+O(\mathfrak{p}^4)
 \end{array}\right].
\label{eq4.8}
\end{eqnarray}
It follows from the fact
\begin{eqnarray*}
&&|1+\bar{\mu}_0p_{2,1}\mathfrak{p}^2|=((1+{\rm Re}(\bar{\mu}_0p_{2,1})\mathfrak{p}^2)^2+({\rm Im}(\bar{\mu}_0p_{2,1})\mathfrak{p}^2)^2)^{1/2}
\\
&&~~~~~~~~~~~~~~~~~\,=1+{\rm Re}(\bar{\mu}_0p_{2,1})\mathfrak{p}^2+O(\mathfrak{p}^4)
\end{eqnarray*}
and
$$
{\rm arg}(1+\bar{\mu}_0p_{2,1}\mathfrak{p}^2)={\rm arctan}\left(\frac{{\rm Im}(\bar{\mu}_0p_{2,1})\mathfrak{p}^2}{1+{\rm Re}(\bar{\mu}_0p_{2,1})\mathfrak{p}^2}\right)={\rm Im}(\bar{\mu}_0p_{2,1})\mathfrak{p}^2+O(\mathfrak{p}^4)
$$
that for sufficiently small $\mathfrak{p}$ mapping \eqref{eq4.8} can be represented as
\begin{eqnarray*}
\left\{
\begin{array}{ll}
\tilde{\mathfrak{p}}=\mathfrak{p}+\mathfrak{A}\,\mathfrak{p}^3+O(\mathfrak{p}^4),
\\
\tilde{\mathfrak{v}}= \mathfrak{v}+\varsigma+{\rm Im}(\bar{\mu}_0p_{2,1})\mathfrak{p}^2+O(\mathfrak{p}^4),
\end{array}
\right.
\end{eqnarray*}
where the first Lyapunov coefficient
\begin{eqnarray*}
&&\mathfrak{A}:=-{\frac { \left( {\beta}^{2}+\beta\,r-1 \right)  \left( \beta+r
 \right) ^{4}}{8\,{N}^{2} \left( \beta+r-1 \right) }}.
\end{eqnarray*}
We find that $\mathfrak{A}\neq0$ as $r\neq-(\beta^2-1)/\beta$. Thus, condition (C.3) in \cite[p.141]{Kuznetsov} holds.

Therefore, if $(a, b)\in \mathfrak{L}_{2}$ with
\begin{eqnarray*}
&&r\neq-{\frac {{\beta}^{2}-3\,\beta+3}{\beta-3}},~-{\frac {{\beta}^{2}-2\,\beta+2}{\beta-2}}, 1-\beta,~~\mbox{and}~~ -\frac{\beta^2-1}{\beta},
\end{eqnarray*}
all conditions in Theorem 4.6 of \cite[p.141]{Kuznetsov} are satisfied, implying that mapping \eqref{eq2.1} undergoes a Neimark-Sacker bifurcation. Besides, if the parameters $\beta$ and $r$ lie in the set
\begin{eqnarray*}
&&\mathfrak{E}_1:=\{0<\beta<\frac{4}{5},~~\Psi_1<r<-\frac{\beta^2-1}{\beta}\},
\end{eqnarray*}
\vskip -1.2cm
\begin{eqnarray*}
&&\!\!\!\!\!\!\!\!\!\!\!\!\!\!\!\!\!\!\!({\rm or}~\mathfrak{E}_2:=\{\frac{\sqrt{5}}{2}-1<\beta<\frac{4}{5},~~\Psi_1<r<1\}\cup\{N>0,~~\frac{4}{5}\leq\beta<1,~~\Psi_1<r<1\}),
\end{eqnarray*}
i.e., $\mathfrak{A}>0$ (or $\mathfrak{A}<0$), mapping \eqref{eq2.1} undergoes a subcritical (or supercritical) Neimark-Sacker bifurcation and produces a unique unstable (or stable) invariant circle surrounding $E_2$ when $\alpha$ crosses $\Psi_3$ from the region $\mathfrak{D}_{4}$ (resp. $\mathfrak{D}_{23}$) to $\mathfrak{D}_{23}$ (resp. $\mathfrak{D}_{4}$). This completes the proof.
\end{proof}

From Theorem \ref{th5.1}, we know that when $\mathfrak{A}$ is less than 0, mapping \eqref{eq2.1} undergoes an Neimark-Sacker bifurcation and generating a stable invariant circle, that is the numbers of susceptible, infective and recovered individuals will eventually show periodic or quasi-periodic disturbances.
In addition, according to Theorem \ref{th5.1}, we can also see that when $\mathfrak{A}>0$, the mapping \eqref{eq2.1} generates an unstable invariant circle. This indicates that if the initial numbers of susceptible, infective and recovered individuals are inside the invariant cycle, their quantities will gradually stabilize through periodic or quasi-periodic fluctuations, approaching values ${N \left( \beta+r \right) }/{\alpha}$, ${\beta\,N \left( \alpha-\beta-r \right) }/{\alpha\, \left( \beta+r \right) }$ and $ {rN \left( \alpha-\beta-r \right) }/{\alpha\, \left( \beta+r\right) }$ respectively. And so that the disease is effectively under control. Conversely, if the initial numbers of susceptible, infective and recovered individuals are outside the invariant cycle, their quantities will exhibit periodic or quasi-periodic fluctuations and show a gradually increasing trend, and it follows that the disease will get out of control.

\section{Codimension 2 bifurcations}
\allowdisplaybreaks[4]
\setcounter{equation}{0}
\subsection{Chenciner bifurcation at $E_2$}
From the section 3.3, we know that mapping \eqref{eq2.1} can generate Neimark-Sacker bifurcations at the fixed point $E_2$. In this section, we assume that Neimark-Sacker bifurcation conditions hold and further consider the case where the first Lyapunov coefficient $\mathfrak{A}=0$, i.e., the parameters $\beta$ and $r$ belong to the set
$$\mathfrak{E}_3:=\{\frac{\sqrt{5}}{2}<\beta<\frac{4}{5},~~r=-(\beta^2-1)/\beta\}.$$
In this case, we see that mapping \eqref{eq2.1} may undergo a Chenciner bifurcation (or the generalized Neimark-Sacker bifurcation) from \cite{Kuznetsov}, which first appeared in the research work of Chenciner (see \cite{Chenciner1,Chenciner2}). In addition, we can refer \cite{Arrowsmith}, \cite{Kuznetsov} and references therein. In fact, we have

\begin{thm}
For parameters $\beta$ and $r$ satisfy $\mathfrak{E}_3$, if the parameters cross $\mathfrak{L}_2$ with
\begin{eqnarray*}
&&r\notin \{\frac{2}{3}, ~\frac{3}{4}, ~\frac{30+2\,\sqrt{5}}{44}\},~\beta \notin \{\frac{2}{3},\frac{3}{4},\frac{4}{5}\},\mathfrak{S}_5\neq0\\
&&{\mbox{and}}~~31\,{\beta}^{5}+4\,{\beta}^{4}-49\,{\beta}^{3}+11\,{
\beta}^{2}+19\,\beta-8\neq0,
\end{eqnarray*}
where
\begin{eqnarray*}
&&\!\!\!\!\!\!\!\mathfrak{S}_5:= ( -960\,{\beta}^{14}+5968\,{\beta}^{13}-15120\,{\beta}^{12}+
18548\,{\beta}^{11}-7044\,{\beta}^{10}-10704\,{\beta}^{9}\\
&&~~~~~+18024\,{
\beta}^{8}-12748\,{\beta}^{7}+5004\,{\beta}^{6}-1064\,{\beta}^{5}+96\,
{\beta}^{4} ) N+123594\,{\beta}^{12}\\
&&~~~~~-721059\,{\beta}^{11}+
2007411\,{\beta}^{10}-3554451\,{\beta}^{9}+4424911\,{\beta}^{8}-
3996670\,{\beta}^{7}\\
&&~~~~~+2618648\,{\beta}^{6}\-1226063\,{\beta}^{5}+400369
\,{\beta}^{4}-88077\,{\beta}^{3}+12383\,{\beta}^{2}-1008\,\beta+36
\end{eqnarray*}
then mapping \eqref{eq2.1} undergoes a Chenciner bifurcation near $E_2$.
\label{th6.1}
\end{thm}
\begin{proof}
Actually, we need verify the conditions (CH.0), (CH.1) and (CH.2) offered in Section 9.4 of \cite[pp.418-422]{Kuznetsov} to prove the theorem.
Applying the same method as in Theorem \ref{th5.1},
we first translate the fixed point $E_2$ to the origin $O$ in mapping \eqref{eq2.1} and expand it in Taylor series, and further set
\begin{eqnarray}
\label{p-1}
\left\{
\begin{array}{ll}
\epsilon=\alpha+\frac{1}{\beta\,(\beta-1)}, \\
 \varepsilon=r+\frac{{\beta}^{2}-1}{\beta}.
\end{array}
\right.
\end{eqnarray}
Then it follows that
\begin{eqnarray}
\left[
\begin{array}{ccc}
   u_6  \\
   \noalign{\medskip}
   v_6\\
   \noalign{\medskip}
   w_6
\end{array}
\right]
\mapsto
\left[\begin{array}{cc}
-{\frac { \left( {\beta}^{3}\epsilon-{\beta}^{2}\epsilon-{\beta}^{2}
\varepsilon+\varepsilon\,\beta-2\,\beta+1 \right) u_6}{ \left(
\varepsilon\,\beta+1 \right)  \left( \beta-1 \right) }}-{\frac {
 \left( \varepsilon\,\beta+1 \right) v_6}{\beta}}-{\frac { \left( {
\beta}^{2}\epsilon-\beta\,\epsilon-1 \right)\,u_6\,v_6 }{N\beta\, \left( \beta-1
 \right) }}\\
\noalign{\medskip}
{\frac {{\beta}^{2} \left( \beta\,\epsilon-\varepsilon\,\beta-\epsilon
+\varepsilon-1 \right) u_6}{ \left( \varepsilon\,\beta+1 \right)
 \left( \beta-1 \right) }}+v_6+{\frac {\left( {\beta}^{2}\epsilon-
\beta\,\epsilon-1 \right)\,u_6\,v_6 }{N\beta\, \left( \beta-1 \right) }}\\
\noalign{\medskip}
-{\frac { \left( {\beta}^{2}-\varepsilon\,\beta-1 \right) v_6}{\beta}}-
 \left( \beta-1 \right) w_6
\end{array}\right].
\label{eq5.1}
\end{eqnarray}
By employing the invertible linear transformation $(u_6,v_6,w_6)^T=H_3(u_7,v_7,w_7)^T$, where
\begin{eqnarray*}
H_3=\left[ \begin {array}{ccc}
1&0&0\\
0&1&0\\
-1&-1&1
\end {array}
 \right],
\end{eqnarray*}
mapping \eqref{eq5.1} is converted into the following form
\begin{eqnarray}
\left[
\begin{array}{ccc}
   u_7  \\
   \noalign{\medskip}
   v_7\\
   \noalign{\medskip}
   w_7
\end{array}
\right]
\mapsto
\left[\begin{array}{cc}
-{\frac { \left( {\beta}^{3}\epsilon-{\beta}^{2}\epsilon-{\beta}^{2}
\varepsilon+\varepsilon\,\beta-2\,\beta+1 \right)\,u_7}{ \left(
\varepsilon\,\beta+1 \right)  \left( \beta-1 \right) }}-{\frac {
 \left( \varepsilon\,\beta+1 \right)\,v_7}{\beta}}-{\frac {\left( {
\beta}^{2}\epsilon-\beta\,\epsilon-1 \right)\,u_7\,v_7}{N\beta\, \left( \beta-1
 \right) }}\\
\noalign{\medskip}
{\frac {{\beta}^{2} \left( \beta\,\epsilon-\varepsilon\,\beta-\epsilon
+\varepsilon-1 \right)\,u_7}{ \left( \varepsilon\,\beta+1 \right)
 \left( \beta-1 \right) }}+v_7+{\frac {\left( {\beta}^{2}\epsilon-
\beta\,\epsilon-1 \right)\,u_7\,v_7 }{N\beta\, \left( \beta-1 \right) }}\\
\noalign{\medskip}
-\left( \beta-1 \right)\,w_7
\end{array}\right].
\label{eq5.2}
\end{eqnarray}
As done in the proof of Theorem \ref{th5.1}, mapping \eqref{eq5.2} can be restricted to a two dimensional $C^{2}$ center manifold
\begin{eqnarray*}
w_7=h_{4}(u_7,v_7)=O(|(u_7,v_7)|^3),
\end{eqnarray*}
and then it is changed into the following two dimensional form
\begin{eqnarray}
\left[
\begin{array}{ccc}
   u_7  \\
   \noalign{\medskip}
   v_7
\end{array}
\right]
\mapsto
\left[\begin{array}{cc}
-{\frac { \left( {\beta}^{3}\epsilon-{\beta}^{2}\epsilon-{\beta}^{2}
\varepsilon+\varepsilon\,\beta-2\,\beta+1 \right)\,u_7}{ \left(
\varepsilon\,\beta+1 \right)  \left( \beta-1 \right) }}-{\frac {
 \left( \varepsilon\,\beta+1 \right)\,v_7}{\beta}}-{\frac {\left( {
\beta}^{2}\epsilon-\beta\,\epsilon-1 \right)\,u_7\,v_7}{N\beta\, \left( \beta-1
 \right) }}\\
\noalign{\medskip}
{\frac {{\beta}^{2} \left( \beta\,\epsilon-\varepsilon\,\beta-\epsilon
+\varepsilon-1 \right)\,u_7}{ \left( \varepsilon\,\beta+1 \right)
 \left( \beta-1 \right) }}+v_7+{\frac {\left( {\beta}^{2}\epsilon-
\beta\,\epsilon-1 \right)\,u_7\,v_7 }{N\beta\, \left( \beta-1 \right) }}
 \end{array}\right].
\label{eq5.3}
\end{eqnarray}
Moreover, applying the invertible linear transformation $(u_7,v_7)^T=H_4(u_8,v_8)^T$ for mapping \eqref{eq5.3}, where
\begin{eqnarray*}
H_4=\left[ \begin {array}{cc}
-{\frac {{\beta}^{2}\epsilon-\beta\,\epsilon-1}{2\,\beta\, \left(
\beta\,\epsilon-\varepsilon\,\beta-\epsilon+\varepsilon-1 \right) }}
 & \Upsilon_0 \\
\noalign{\medskip}
1&0
\end {array} \right],
\end{eqnarray*}
and
\begin{eqnarray*}
&&\!\!\!\!\!\!\!\Upsilon_0:={\frac {\sqrt {\Upsilon_1}}{{2\,\beta}^{2} (  ( \epsilon-\varepsilon )
\beta-\epsilon+\varepsilon-1 ) }},\\
&&\!\!\!\!\!\!\!\Upsilon_1:=-\beta\, ( {\beta}^{3} ( \beta-1
 ) ^{2}{\epsilon}^{2}- ( 4\,\beta-4 )  ( {\beta}
^{2} ( \beta-1 ) {\varepsilon}^{2}+ ( 2\,{\beta}^{2}-2
\,\beta ) \varepsilon+\frac{{\beta}^{2}}{2}+\beta-1 ) \epsilon\\
&&~~~~~+4\,{\beta}^{2} ( \beta-1 ) ^{2}{\varepsilon}^{3}+ ( 12
\,{\beta}^{3}-20\,{\beta}^{2}+8\,\beta ) {\varepsilon}^{2}+
 ( 12\,{\beta}^{2}-16\,\beta+4 ) \varepsilon+5\,\beta-4) ,
\end{eqnarray*}
we obtain
{\small\begin{eqnarray}
\left[
\begin{array}{cc}
   \mu_1\\
   \\
   \nu_1
\end{array}
\right]
\mapsto
\left[\begin{array}{cc}
a_{1,0}(\epsilon,\varepsilon)\mu_{{1}}+a_{0,1}(\epsilon,\varepsilon)\nu_{{1}}+a_{2,0}(\epsilon,\varepsilon)\mu_1^2
+a_{1,1}(\epsilon,\varepsilon)\mu_{{1}}\nu_{{1}}+a_{0,2}(\epsilon,\varepsilon)\nu_1^2
\\
   \\
b_{1,0}(\epsilon,\varepsilon)\mu_{{1}}+b_{0,1}(\epsilon,\varepsilon)v\nu_{{1}}+b_{2,0}(\epsilon,\varepsilon)\mu_1^2
+b_{1,1}(\epsilon,\varepsilon)\mu_{{1}}\nu_{{1}}+b_{0,2}(\epsilon,\varepsilon)\nu_1^2
 \end{array}\right],
\label{eq5.4}
\end{eqnarray}}
where
\begin{eqnarray*}
&&a_{1,0}(\epsilon,\varepsilon)=b_{0,1}(\epsilon,\varepsilon):=-{\frac {{\beta}^{3}\epsilon-{\beta}^{2}\epsilon-2\,{\beta}^{2}
\varepsilon+2\,\varepsilon\,\beta-3\,\beta+2}{ 2\,\left( \beta-1 \right) \left( \varepsilon\,\beta+1 \right) }},\\
&&a_{0,1}(\epsilon,\varepsilon)=-b_{1,0}:={\frac {\sqrt {\Upsilon_1}}{ 2\,\left( \beta-1 \right)  \left( \beta\,\varepsilon+1 \right)
}},\\
&&a_{2,0}(\epsilon,\varepsilon):={\frac {{\beta}^{4}{\epsilon}^{2}-2\,{\beta}^{3}{\epsilon}^{2}+{
\beta}^{2}{\epsilon}^{2}-2\,{\beta}^{2}\epsilon+2\,\beta\,\epsilon+1}{2\,
 \left( \beta-1 \right) N \left( \beta\,\epsilon-\beta\,\varepsilon-
\epsilon+\varepsilon-1 \right) {\beta}^{2}}},\\
&&a_{1,1}(\epsilon,\varepsilon):={\frac {\sqrt {\Upsilon_1} \left( {\beta}^{2}\epsilon-\beta\,\epsilon-1 \right) }{ 2\,\left(
\beta-1 \right) N{\beta}^{3} \left( \beta\,\epsilon-\varepsilon\,\beta
-\epsilon+\varepsilon-1 \right) }},\\
&&a_{0,2}(\epsilon,\varepsilon)=b_{0,2}(\epsilon,\varepsilon):=0, b_{2,0}(\epsilon,\varepsilon):={\frac {\sqrt {\Upsilon_1} \left( {\beta}^{2}\epsilon-\beta\,\epsilon-1 \right) }{ 2\,\left(
\beta-1 \right) N{\beta}^{3} \left( \beta\,\epsilon-\varepsilon\,\beta
-\epsilon+\varepsilon-1 \right) }}
, \\
&&b_{1,1}(\epsilon,\varepsilon):={\frac {\Upsilon_2}{2\, \left( \beta-1 \right)
N \left( \beta\,\epsilon-\varepsilon\,\beta-\epsilon+\varepsilon-1
 \right) {\beta}^{2}}},\\
&&\Upsilon_2:={\beta}^{4}{\epsilon}^{2}-2\,\epsilon\,{\beta}^{4}
\varepsilon-2\,{\beta}^{3}{\epsilon}^{2}+4\,{\beta}^{3}\epsilon\,
\varepsilon-2\,{\beta}^{3}\epsilon+{\beta}^{2}{\epsilon}^{2}-2\,{\beta
}^{2}\epsilon\,\varepsilon+2\,{\beta}^{2}\epsilon+2\,{\beta}^{2}
\varepsilon\\
&&~~~~~~~~-2\,\varepsilon\,\beta+2\,\beta-1.
\label{norm-1}
\end{eqnarray*}
Let $\omega_1:= \mu_1 +{\bf{i}}\,\nu_1$. Consider $\omega_1$ as the variable in Taylor expansion, mapping \eqref{eq5.4} can be transformed into the following complex form
\begin{equation}
\omega_1\mapsto \mu(\epsilon,\varepsilon)\omega_1+\sum_{i=0}^2\gamma_{2-i,i}(\epsilon,\varepsilon)\omega^{2-i}_1\bar{\omega}^i_1,
\label{eq5.5}
\end{equation}
where
\begin{eqnarray*}
&&\!\!\!\!\!\!\!\mu(\epsilon,\varepsilon):=\frac{\Upsilon_3}{2\,\sqrt {\Upsilon_1} \left(\beta\,\varepsilon+1 \right)  \left( \beta-1 \right) },\\
&&\!\!\!\!\!\!\!\Upsilon_3:=8\,{\bf{i}}\,{\beta}^{4}{\varepsilon}^{2}\epsilon-4\,{\bf{i}}\,{\beta}^{3}{\varepsilon}^{2}\epsilon
+16\,{\bf{i}}\,{\beta}^{3}\varepsilon\,\epsilon-8\,{\bf{i}}{
\beta}^{2}\epsilon\,\varepsilon-8\,{\bf{i}}{\beta}^{4}\varepsilon\,\epsilon-4\,{\bf{i}}\,{\beta}^{5}{\varepsilon}^{2}\epsilon-2\,\sqrt {\Upsilon_1}+{\beta}^{4}{\epsilon}^{2}\,{\bf{i}}\\
&&~~~~+{\epsilon}^{2}{\beta}^{6}\,{\bf{i}}+2\,\sqrt {\Upsilon_1}{\beta}^{2}\varepsilon+\sqrt {\Upsilon_1}{\beta}^{2}\epsilon-\sqrt {\Upsilon_1}{\beta}^{3}\epsilon-2\,\sqrt {\Upsilon_1}\beta\,\varepsilon+3\,\beta\,
\sqrt {\Upsilon_1}+5\,{\bf{i}}\,{\beta}^{2}\\
&&~~~~-4\,{\bf{i}}\,\beta-4\,{\bf{i}}\,\beta\,\epsilon-16\,{\bf{i}}\,{
\beta}^{2}\varepsilon-2\,{\bf{i}}\,{\beta}^{3}\epsilon+12\,{\bf{i}}\,{\beta}^{3}\varepsilon
+8\,{\bf{i}}\,{\beta}^{2}{\varepsilon}^{2}+12\,{\bf{i}}\,{\beta}^{4}{\varepsilon}^{2}+4\,{\bf{i}}\,{\beta}^{3}{\varepsilon}^{3}\\
&&~~~~-20\,{\bf{i}}\,{\beta}^{3}{\varepsilon}^{2}+4\,{\bf{i}}\,\beta\,\varepsilon+4\,{\bf{i}}\,{\beta}^{5}{\varepsilon}^{3}
-8\,{\bf{i}}\,{\beta}^{4}{\varepsilon}^{3}-2\,{\bf{i}}\,{\beta}^{5}{\epsilon}^{2}-2\,{\bf{i}}\,{\beta}^{4}\epsilon
+8\,{\bf{i}}\,{\beta}^{2}\epsilon,\\
&&\!\!\!\!\!\!\!\gamma_{2,0}(\epsilon,\varepsilon):=\frac { \Upsilon_4\,( {\beta}^{2}\epsilon-\beta\,
\epsilon-1 ) }{4\,\sqrt {\Upsilon_1}N \left( \beta-1 \right) {\beta}^{2}},\\
&&\!\!\!\!\!\!\!\Upsilon_4:=( {\bf{i}}\,{\beta}^{4}\epsilon-2\,{\bf{i}}{\beta}^{3}{\varepsilon}
^{2}-{\bf{i}}\,{\beta}^{3}\epsilon+2\,{\bf{i}}\,{\beta}^{2}{\varepsilon}^{2}-4\,{\bf{i}}\,{\beta}
^{2}\varepsilon-{\bf{i}}\,{\beta}^{2}+4\,{\bf{i}}\,\beta\,\varepsilon-2\,{\bf{i}}\,\beta-\beta\,
\sqrt {\Upsilon_1}+2\,{\bf{i}} ),\\
&&\!\!\!\!\!\!\!\gamma_{1,1}:=\frac { ( {\beta}^{3}\epsilon\,{\bf{i}}-2\,{\bf{i}}\,{\beta}^{3}\varepsilon
-{\bf{i}}\,{\beta}^{2}\epsilon+2\,{\bf{i}}\,{\beta}^{2}\varepsilon-2\,{\bf{i}}\,{\beta}^{2}+{\bf{i}}\,\beta-\sqrt {\Upsilon_1} ) ( {\beta}^{2}\epsilon-\beta\,\epsilon-1 ) ^{2}}{\sqrt {\Upsilon_1} ( \beta-1 ) N
 ( \beta\,\epsilon-\beta\,\varepsilon-\epsilon+\varepsilon-1 ) {\beta}^{2}},\\
&&\!\!\!\!\!\!\!\gamma_{0,2}:=-\frac { \Upsilon_5\,( {\beta}^{2}\epsilon-\beta\,\epsilon-1 ) ( \beta\,\varepsilon+1 ) }{4\,\sqrt {
\Upsilon_1} ( \beta\,\epsilon-\beta\,\varepsilon-\epsilon+\varepsilon-1 ) N{\beta}^{2}},\\
&&\!\!\!\!\!\!\!\Upsilon_5:={\beta}^{3}\epsilon\,{\bf{i}}-2\,{\bf{i}}\,{\beta}^{2}\epsilon\,
\varepsilon+2\,{\bf{i}}\,{\beta}^{2}{\varepsilon}^{2}-{\bf{i}}\,{\beta}^{2}\epsilon+2\,{\bf{i}}\,\beta\,\epsilon\,\varepsilon
-2\,{\bf{i}}\,\beta\,{\varepsilon}^{2}-2\,{\bf{i}}\,\beta\,
\epsilon+4\,{\bf{i}}\,\beta\,\varepsilon-{\bf{i}}\,\beta\\
&&~~~~+2\,{\bf{i}}\,\epsilon-2\,{\bf{i}}\,\varepsilon+2\,{\bf{i}}+\sqrt {\Upsilon}.
\end{eqnarray*}

Now we check the condition (CH.0), i.e. $\mu^{i}(0,0)=(3\,\beta-2-{\bf{i}}\,\sqrt{-5\,\beta^2+4\,\beta})^i/(2\,\beta-2)^i\ne 1$ for $i=1,2,3,4,5,6$. Firstly, since $\mu(0,0)$ is not a real number, then the inequality $\mu^{i}(0,0)\ne 1$ for $i=1,2$ is true. Further, we claim $\mu^{3}(0,0)\neq1$. Otherwise, from $(\mu(0,0)-1)(\mu^2(0,0)+\mu(0,0)+1)=0$ we get $\mu^2(0,0)+\mu(0,0)+1=0$, implying $$\mu(0,0)=-\frac{1}{2}+\frac{\sqrt{3}}{2}\,{\bf{i}} ~~~\mbox{or}~~~ \mu(0,0)=-\frac{1}{2}-\frac{\sqrt{3}}{2}\,{\bf{i}}, $$
and thus $\beta=3/4$, which contradicts to the conditions in Theorem \ref{th5.1}. Secondly, in the same way, suppose that $\mu^{4}(0,0)=1$, which is equivalent to $\mu^2(0,0)+1=0$, implying the real part of $\mu(0,0)$, denoted as $\mathfrak{R}(\mu_0)$, is equal to $0$. This means $\beta=2/3$, a contradiction to the conditions in Theorem \ref{th5.1}. Besides, since $\beta\neq (30-2\,\sqrt{5})/44$ (from $\mathfrak{E}_3$) and $(30+2\,\sqrt{5})/44$, $\mu^{5}(0,0)\neq1$. Finally, notice that $\beta\neq1/2$ by $\mathfrak{E}_3$, hence $\mu^{6}(0,0)\neq1$. Therefore, We have verified the condition (CH.0).

Furthermore, we verify condition (CH.1). From the proof of Theorem \ref{th5.1}, we know that there exists an approximate identity transformation such that mapping \eqref{eq5.5} can be transformed into the following Poincare normal form
\begin{eqnarray}
\begin{aligned}
\omega_1\mapsto\mu(\epsilon,\varepsilon)\omega_1+p_{2,1}(\epsilon,\varepsilon)\omega_1\bar{\omega}^2_1+O(|\omega_1|^4)~~~\\
=e^{{\bf{i}}\,\theta(\epsilon,\varepsilon)}
(r(\epsilon,\varepsilon)+d_1(\epsilon,\varepsilon)\,\omega_1\,\bar{\omega}^2_1)+O(|\omega_1|^4),\!\!\!\!\!\!\!\!\!\\
\end{aligned}
\label{norm-ch1}
\end{eqnarray}
where
\begin{eqnarray*}
&&\!\!\!\!\!\!\!r(\epsilon,\varepsilon):=\sqrt {{\frac {{\beta}^{3}\epsilon\,\varepsilon-{\beta}^{3}{
\varepsilon}^{2}-{\beta}^{3}\epsilon-{\beta}^{2}\epsilon\,\varepsilon+
{\beta}^{2}{\varepsilon}^{2}+2\,{\beta}^{2}\epsilon-{\beta}^{2}
\varepsilon-\beta\,\epsilon+\beta-1}{ \left( \beta\,\varepsilon+1
 \right)  \left( \beta-1 \right) }}},\\
&&\!\!\!\!\!\!\!d_1(\epsilon,\varepsilon):=-\frac { \Upsilon_6 }{\Upsilon_7},~{{\mbox{and $\Upsilon_6$ and $\Upsilon_7$ are shown in Appendix.}}}\\
\end{eqnarray*}
Besides, one can see $\mu(\epsilon,\varepsilon)|_{\epsilon=\varepsilon=0}=\mu(0,0)=\mu_0$ and $p_{2,1}(\epsilon,\varepsilon)|_{\epsilon=\varepsilon=0}=p_{2,1}(0,0)=p_{2,1}$, where $\mu_0$ and $p_{2,1}$ as shown in Mapping \eqref{eq4.7}.

Furthermore, let
\begin{eqnarray}
\left\{\begin{array}{ll}
\epsilon_0=r(\epsilon,\varepsilon)-1,\\
\varepsilon_0=Re(d_1(\epsilon,\varepsilon)).
\end{array}\right.
\label{p-2}
\end{eqnarray}
We calculate that
\begin{eqnarray*}
\Upsilon_8:=det
\left.
\left(
\begin{array}{ll}
\frac{\partial\epsilon_0}{\partial\epsilon}&\frac{\partial\epsilon_0}{\partial\varepsilon}\\
\frac{\partial\varepsilon_0}{\partial\epsilon}&\frac{\partial\varepsilon_0}{\partial\varepsilon}
\end{array}
\right)
\right|_{(\epsilon,\varepsilon)=(0,0)}
={\frac {31\,{\beta}^{5}+4\,{\beta}^{4}-49\,{\beta}^{3}+11\,{
\beta}^{2}+19\,\beta-8}{16\,{\beta}^{2} \left( \beta-1 \right) ^{4}{N}^{2}
}}.
\end{eqnarray*}
Since $31\,{\beta}^{5}+4\,{\beta}^{4}-49\,{\beta}^{3}+11\,{
\beta}^{2}+19\,\beta-8\neq0$ from the conditions of the theorem, then $\Upsilon_8\neq0$. Hence, we obtain that condition (CH.1) is true.

Finally, we will discuss the condition (CH.2). For the convenience of calculation, let $r=-(\beta^2-1)/\beta$ in mapping \eqref{eq5.4}, one obtain that
\begin{equation}
\omega\mapsto \tilde{\mu}_0\omega+\sum_{i=0}^2\tilde{\gamma}_{2-i,i}\omega^{2-i}\bar{\omega}^i,
\label{eq6.1}
\end{equation}
where
\begin{eqnarray*}
\!\!\!\!&&\tilde{\gamma}_{2,0}:={\frac {1}{4\,\beta\,N \left( \beta-1 \right) }}+\frac {\left( {\beta}^{2}+2\,\beta-2 \right)\,\sqrt{\beta}\,{\bf{i}}}{4\,{\beta}^{3}N
 \left( \beta-1 \right)\,\sqrt{-5\,\beta+4} },\\
\!\!\!\!&&\tilde{\gamma}_{1,1}:=\frac{1}{4\,{\beta}^{2}N \left( \beta-1 \right) }+\frac {\left( 2\,\beta-1 \right)\,\sqrt{\beta}\,{\bf{i}} }{4\,{\beta}^{2}N \left( \beta-1 \right)\,\sqrt{-5\,\beta+4} }
,\\
\!\!\!\!&&\tilde{\gamma}_{0,2}:={\frac {1}{4\,{\beta}^{2}N }}+\frac {\left( \beta-2 \right)\,\sqrt{\beta}\,{\bf{i}} }{4\,{\beta}^{3}N\,\sqrt{-5\,\beta+4}}
.\\
\end{eqnarray*}
Furthermore, using the near-identity transformation
$$
\omega=h(\omega):=\omega+\sum_{i\geq0,j\geq0,i+j=2}^{5}h_{i,j}\omega^i\bar{\omega}^j,
$$
where $h_{i,j}$ are undetermined coefficients, and applying the method used in Theorem \ref{th5.1} to compute Poincare normal form, mapping \eqref{eq6.1} can be further transformed into the following form
\begin{eqnarray}
&&\omega\rightarrow\tilde{\mu}_0\,\omega+c_1\,\omega^2\,\bar{\omega}+c_2\,\omega^2\,\bar{\omega}^3+O(|\omega|^6),
\label{smap}
\end{eqnarray}
where
{\footnotesize
\begin{eqnarray*}
&&\!\!\!\!\!\!\!c_1:=-\frac{\tilde{\mathfrak{S}}_1}{\tilde{\mathfrak{S}}_2},\\
&&\!\!\!\!\!\!\!\tilde{\mathfrak{S}}_1:=( 4\,\beta-{\frac{16}{5}} )\,( \beta-1)( \beta+1)( {\beta}^{2}-\frac{1}{2})\,(( \beta\,{\bf{i}}-\frac{3}{4}\,{\bf{i}}) \sqrt {-5\,{\beta}^{2}+4\,\beta}+\frac{5}{2}\,{\beta}^{2}-{\frac {13\,\beta}{4}}+1 ),\\
&&\!\!\!\!\!\!\!\tilde{\mathfrak{S}}_2:=125\,( \beta\,{\bf{i}}-\frac{4}{5}\,{\bf{i}}+\frac{1}{5}\,\sqrt {-5\,{\beta}^{2}+4\,\beta}
 ) ^{2}{\beta}^{5} ( \beta\,{\bf{i}}-\frac{4}{5}\,{\bf{i}}-\frac{1}{5}\,\sqrt {-5\,{\beta}^{2}+4\,\beta}) ( \beta-\frac{3}{4} )  (  ( -\frac{3}{5}\,\beta+\frac{2}{5} ) \\
&&~~~~~~\sqrt {-5\,{\beta}^{2}+4\,\beta}+{\bf{i}} ( \beta-\frac{4}{5} ) \beta ) {N}^{2},\\
&&\!\!\!\!\!\!\!c_2:=\frac{\mathfrak{S}_3}{\mathfrak{S}_4},\\
&&\!\!\!\!\!\!\!\mathfrak{S}_3:=116032\, ( -\frac { ( 32\,\beta-24 ) \sqrt {-5\,{
\beta}^{2}+4\,\beta}}{5439} ( {\frac{3}{40}}+N{\beta}^{15}-{
\frac {463\,N{\beta}^{14}}{60}}+ ( -{\frac{9731}{160}}+{\frac {
1013\,N}{40}} ) {\beta}^{13}\\
&&~~~~~~+ ( -{\frac {355\,N}{8}}+{\frac
{153799}{320}} ) {\beta}^{12}
+( {\frac {949\,N}{24}}-{
\frac{553179}{320}} ) {\beta}^{11}+ ( -{\frac {123\,N}{40}}
+{\frac{229241}{60}} ) {\beta}^{10}+ ( -{\frac {1411\,N}{40
}}\\
&&~~~~~~-{\frac{70357}{12}} ) {\beta}^{9}+ ( {\frac {5321\,N}{120
}}+{\frac{1057643}{160}} ) {\beta}^{8}+ ( -{\frac {227\,N}{
8}}-{\frac{893079}{160}} ) {\beta}^{7}+ ( {\frac {85\,N}{8}
}+{\frac{562569}{160}} ) {\beta}^{6}\\
&&~~~~~~+ ( -{\frac {133\,N}{60
}}-{\frac{519367}{320}} ) {\beta}^{5}+ ( \frac{N}{5}+{\frac{513041}
{960}} ) {\beta}^{4}-{\frac {2909\,{\beta}^{3}}{24}}+{\frac {
8687\,{\beta}^{2}}{480}}-{\frac {53\,\beta}{32}} ) +{\bf{i}}\, ( {
\beta}^{13}\\
&&~~~~~~-{\frac {68275\,{\beta}^{12}}{9324}}+{\frac {3366367\,{
\beta}^{11}}{130536}}-{\frac {1083521\,{\beta}^{10}}{18648}}+{\frac {
868241\,{\beta}^{9}}{9324}}-{\frac {1810112\,{\beta}^{8}}{16317}}+{
\frac {6506401\,{\beta}^{7}}{65268}}\\
&&~~~~~~-{\frac {4406245\,{\beta}^{6}}{
65268}}+{\frac {1112617\,{\beta}^{5}}{32634}}-{\frac {1644655\,{\beta}
^{4}}{130536}}+{\frac {431737\,{\beta}^{3}}{130536}}-{\frac {19177\,{
\beta}^{2}}{32634}}+{\frac {1051\,\beta}{16317}}-{\frac{221}{65268}}
 ) \beta\, ( \beta\\
&&~~~~~~-\frac{4}{5} )  )  ( \beta-1 ) ^{10} ( \beta-\frac{4}{5} ) ^{5},\\
&&\!\!\!\!\!\!\!\mathfrak{S}_4:=703125\,\sqrt {-5\,{\beta}^{2}+4\,\beta}{\beta}^{9} (  ( \frac{2}{5}
\,\beta-\frac{3}{10} ) \sqrt {-5\,{\beta}^{2}+4\,\beta}+{\bf{i}}\,( \beta-\frac{1}{2} )  ( \beta-\frac{4}{5} )  )  (  ( \beta\,{\bf{i}}
\\
&&~~~~~~-\frac{2}{3}\,{\bf{i}} ) \sqrt {-5\,{\beta}^{2}+4\,\beta}+\frac{5}{3}\,{\beta}^{2}-\frac{4}{3}\,\beta )  ( \beta-\frac{2}{3} )  ( \beta\,{\bf{i}}-\frac{2}{3}\,{\bf{i}}+\frac{1}{3}\,\sqrt {-5\,{\beta}^{2}+4\,\beta} )  (  ( -\frac{3}{5}\,\beta\\
&&~~~~~~+\frac{2}{5} ) \sqrt {-5\,{\beta}^{2}+4\,\beta}+{\bf{i}}\,( \beta-\frac{4}{5} ) \beta ) ^{3} (  ( \frac{3}{5}\,\beta-\frac{2}{5} )\sqrt {-5\,{\beta}^{2}+4\,\beta}+{\bf{i}}\,( \beta-\frac{4}{5} ) \beta )  ( \beta-\frac{3}{4} ) ^{3}{N}^{4} ( \beta\,{\bf{i}}\\
&&~~~~~~-\frac{4}{5}\,{\bf{i}}-\frac{1}{5}\,\sqrt {-5\,{\beta}^{2}+4\,\beta} ) ^{3} (  ( -\frac{2}{5}\,\beta+\frac{3}{10} ) \sqrt {-5\,{\beta}^{2}+4\,\beta}+{\bf{i}}\,( \beta-\frac{1}{2} )  ( \beta-\frac{4}{5} )  )  (  ( \beta\,{\bf{i}}\\
&&~~~~~~-\frac{2}{3}\,{\bf{i}} ) \sqrt {-5\,{\beta}^{2}+4\,\beta}-\frac{5}{3}\,{\beta}^{2}+\frac{4}{3}\,\beta )  ( \beta\,{\bf{i}}-\frac{4}{5}\,{\bf{i}}+\frac{1}{5}\,\sqrt {-5\,{\beta}^{2}+4
\,\beta} ) ^{4}.
\end{eqnarray*}}
Besides, from \cite{Kuznetsov}, we can get
\begin{eqnarray*}
&&\!\!\!\!\!\!\!L_2:=(Im(\frac{c_1(0)}{4\,\tilde{\mu}_0})^2+Re(\frac{c_2(0)}{\tilde{\mu}_0})=\frac {\mathfrak{S}_5}{32\, \left( 5\,\beta-4 \right) ^{2} \left( 3\,\beta-2 \right) {N}^{4}
 \left( \beta-1 \right) ^{3}{\beta}^{12} \left( 4\,\beta-3 \right) ^{2
} }.\\
\end{eqnarray*}
From the condition of Theorem \ref{th6.1}, we know that $L_2\neq0$. It implies that condition (CH.2) of \cite{Kuznetsov} is fulfilled.
Therefore, mapping (2.1) undergoes a Chenciner bifurcation near $E_2$.
\end{proof}

\begin{cor}\label{cor6.11}
Under the conditions of Theorem \ref{th6.1}, if $L_2<0$, then there exists the following bifurcation phenomena when the parameters lie in a sufficiently small neighborhood of the point $(r, \alpha)=(-(\beta^2-1)/\beta,-1/(\beta\,(\beta-1)))$,
\begin{description}
    \item[$(1)$] When the parameters cross $\mathfrak{L}$, and the conditions of Theorem \ref{th5.1} hold, mapping \eqref{eq2.1} undergoes a Neimark-Sacker bifurcation and generates an unique stable invariant circle.
    \item[$(2)$] If the parameters belong to the region $\mathfrak{G_1}$, then mapping \eqref{eq2.1} have two invariant circles. The outside one is stable, while the inside one is unstable.
    \item[$(3)$] As the parameters are close to the region $\mathfrak{G_2}$, the two invariant circles of mapping \eqref{eq2.1} become a unique invariant circle, which is stable from the outside and unstable from the inside.
\end{description}
where
{\footnotesize\begin{eqnarray*}
&&\!\!\!\!\!\!\!\!\!\!\!\mathfrak{G_1}:=\{\mathcal{C}_1-1>0, \mathcal{C}_2>0, \sqrt {{\frac {-{\beta}^{3}+ \left( -2\,r+\alpha \right) {\beta}^{2}-
 \left( r-1 \right)  \left( r-\alpha+1 \right) \beta+r}{\beta+r}}}-1<0\},\\
&&\!\!\!\!\!\!\!\!\!\!\!\mathfrak{G_2}:=\{\mathcal{C}_1-1=0, \mathcal{C}_2>0\},\\
&&\!\!\!\!\!\!\!\!\!\!\!\mathcal{C}_1:={\frac {\mathcal{C}_3}{\mathcal{C}_4}  },\\
&&\!\!\!\!\!\!\!\!\!\!\!\mathcal{C}_3:=( 122880\, ( \beta+r ) ^{4} ( {\frac{3}{80}}+N{\beta}^{15}-{\frac {433\,N{\beta}^{14}}{60
}}+ ( {\frac {659\,N}{30}}-{\frac{40973}{320}} ) {\beta}^{13}+ ( -{\frac {8417\,N}{240}}+{\frac{281191}{320}} ) {
\beta}^{12}+ \\
&&~~~( {\frac {3199\,N}{120}}-{\frac{454973}{160}} ) {\beta}^{11}+ ( {\frac {61\,N}{16}}+{\frac{86953}{15}}
 ) {\beta}^{10}+ ( -{\frac {1197\,N}{40}}-{\frac{10638827}{1280}} ) {\beta}^{9}+ ( {\frac {7693\,N}{240}}\\
&&~~~+{\frac{
8418421}{960}} ) {\beta}^{8}+ ( -{\frac {2219\,N}{120}}-{\frac{1763795}{256}} ) {\beta}^{7}+ ( {\frac {1517\,N}{240}
}+{\frac{2563957}{640}} ) {\beta}^{6}+ ( -{\frac {29\,N}{24}}-{\frac{542409}{320}} ) {\beta}^{5}\\
&&~~~+ ( \frac{N}{10}+{\frac{40697}
{80}} ) {\beta}^{4}-{\frac {5015\,{\beta}^{3}}{48}}+{\frac {4453\,{\beta}^{2}}{320}}-{\frac {87\,\beta}{80}} )  ( \beta-
\alpha+r ) ^{4} ( {\beta}^{3}+ ( 3\,r-\alpha ) {\beta}^{2}+3\, ( r-\frac{\alpha}{2} )  ( r\\
&&~~~-\frac{\alpha}{6} )\beta+{r}^{2} ( r-\alpha )  ) ^{2}\sqrt {{\frac {-{\beta}^{3}+ ( -2\,r+\alpha ) {\beta}^{2}- ( r-1
 )  ( r-\alpha+1 ) \beta+r}{\beta+r}}}+4800\,{\alpha}^{4} ( 4\,{\beta}^{9}\alpha\\
&&~~~+ ( -16\,{\alpha}^{2}+28\,\alpha
\,r+10\,\alpha ) {\beta}^{8}+ ( 84\,\alpha\,{r}^{2}+ ( -93\,{\alpha}^{2}+65\,\alpha+1 ) r+25\,{\alpha}^{3}-30\,
{\alpha}^{2}+8\,\alpha ) {\beta}^{7}\\
&&~~~+ ( 140\,{r}^{3}\alpha+ ( -225\,{\alpha}^{2}+180\,\alpha+7 ) {r}^{2}+ ( 114\,
{\alpha}^{3}-160\,{\alpha}^{2}+46\,\alpha ) r-16\,{\alpha}^{2}+2\,\alpha-19\,{\alpha}^{4}\\
&&~~~+{\frac {65\,{\alpha}^{3}}{2}} ) {\beta
}^{6}+ ( 140\,{r}^{4}\alpha+ ( -290\,{\alpha}^{2}+275\,\alpha+21 ) {r}^{3}+ ( 206\,{\alpha}^{3}-350\,{\alpha}^{2}+
108\,\alpha ) {r}^{2}+ ( -{\frac {155\,{\alpha}^{2}}{2}}\\
&&~~~+11\,\alpha-61\,{\alpha}^{4}+{\frac {265\,{\alpha}^{3}}{2}} ) r+7\,
{\alpha}^{5}-15\,{\alpha}^{4}+10\,{\alpha}^{3}-2\,{\alpha}^{2} ) {\beta}^{5}+ ( 84\,{r}^{5}\alpha+ ( -210\,{\alpha}^
{2}+250\,\alpha\\
&&~~~+35 ) {r}^{4}+ ( 184\,{\alpha}^{3}-400\,{\alpha}^{2}+130\,\alpha ) {r}^{3}+ ( -149\,{\alpha}^{2}+25\,\alpha-70\,{\alpha}^{4}+{\frac {415\,{\alpha}^{3}}{2}} ) {r}^{2}+ ( {\frac {73\,{\alpha}^{3}}{2}}-9\,{\alpha}^{2}\\
&&~~~+12\,{\alpha}^{5}-40\,{\alpha}^{4} ) r-{\alpha}^{3} ( \alpha-\frac{1}{2} ) ( \alpha-1 ) ^{2} ) {\beta}^{4}+28\,r ( {r}^{5}\alpha+ ( -{\frac {81\,{\alpha}^{2}}{28}}+{\frac {135\,\alpha}{28}}+\frac{5}{4} ) {r}^{4}+ ( {\frac {81\,{\alpha}^{3}}{28}}\\
&&~~~-{\frac {
125\,{\alpha}^{2}}{14}}+{\frac {20\,\alpha}{7}} ) {r}^{3}+ ( -{\frac {33\,{\alpha}^{4}}{28}}+{\frac {305\,{\alpha}^{3}}{56}
}-{\frac {141\,{\alpha}^{2}}{28}}+{\frac {15\,\alpha}{14}} ) {r}^{2}+\frac {5\,{\alpha}^{2}r}{28} ( {\alpha}^{3}-7\,{\alpha}^{2}
+{\frac {99\,\alpha}{10}}\\
&&~~~-{\frac{16}{5}} ) +{\frac {5\,{\alpha}^{3} ( \alpha-1 )  ( \alpha-\frac{3}{5} ) }{56}}
 ) {\beta}^{3}+4\,{r}^{2} ( {r}^{5}\alpha+ ( -{\frac {13\,{\alpha}^{2}}{4}}+10\,\alpha+{\frac{21}{4}} ) {r}^{4}+
 ( \frac{7}{2}\,{\alpha}^{3}-20\,{\alpha}^{2}\\
&&~~~+\frac{9}{2}\,\alpha ) {r}^{3}+ ( -\frac{5}{4}\,{\alpha}^{4}+{\frac{25}{2}}\,{\alpha}^{3}-16\,{\alpha}^
{2}+5\,\alpha ) {r}^{2}+ ( -\frac{5}{2}\,{\alpha}^{4}+{\frac {59\,{\alpha}^{3}}{8}}-\frac{7}{2}\,{\alpha}^{2} ) r-\frac{1}{2}\,{\alpha}^{4}+\frac{3}{8}\,{
\alpha}^{3} ) {\beta}^{2}\\
&&~~~+5\,{r}^{3} (  ( \alpha+\frac{7}{5} ) {r}^{4}+ ( -2\,{\alpha}^{2}-\frac{4}{5}\,\alpha ) {r}^{3}+
 ( {\alpha}^{3}-{\frac {19\,{\alpha}^{2}}{10}}+\frac{7}{5}\,\alpha ) {r}^{2}+ ( {\frac {13\,{\alpha}^{3}}{10}}-\frac{6}{5}\,{\alpha}^
{2} ) r+\frac{1}{10}\,{\alpha}^{3} ) \beta\\
&&~~~+{r}^{5} ( r-\alpha )  ( -\alpha\,r+{r}^{2}+\alpha )  ) ^{2}
 ( -\frac{2}{3}+\beta )  ( \beta-\frac{4}{5} ) ^{2}{\beta}^{10} ( \beta-1 ) ^{4} ( \beta-\frac{3}{4} ) ^{2} ),\\
&&\!\!\!\!\!\!\!\!\!\!\!\mathcal{C}_4:=2\,( \beta-\alpha+r ) ^{4} ( {\alpha}^{2
}\beta-4\,\alpha\,{\beta}^{2}-8\,\alpha\,\beta\,r-4\,\alpha\,{r}^{2}+4
\,{\beta}^{3}+12\,r{\beta}^{2}+12\,\beta\,{r}^{2}+4\,{r}^{3} ) ^
{2} ( 3840\,N{\beta}^{15}\\
&&~~~-27712\,N{\beta}^{14}+84352\,{\beta}^{13
}N-134672\,{\beta}^{12}N-491676\,{\beta}^{13}+102368\,{\beta}^{11}N+
3374292\,{\beta}^{12}\\
&&~~~+14640\,{\beta}^{10}N-10919352\,{\beta}^{11}-
114912\,{\beta}^{9}N+22259968\,{\beta}^{10}+123088\,{\beta}^{8}N-
31916481\,{\beta}^{9}\\
&&~~~-71008\,{\beta}^{7}N+33673684\,{\beta}^{8}+24272
\,{\beta}^{6}N-26456925\,{\beta}^{7}-4640\,{\beta}^{5}N+15383742\,{
\beta}^{6}+384\,{\beta}^{4}N\\
&&~~~-6508908\,{\beta}^{5}+1953456\,{\beta}^{4}
-401200\,{\beta}^{3}+53436\,{\beta}^{2}-4176\,\beta+144 )
 ( \beta+r ) ^{4},\\
&&\!\!\!\!\!\!\!\!\!\!\!\mathcal{C}_2={\frac {{\alpha}^{2}\,\mathcal{C}_5}{ 16\,\left( \beta+r \right) ^{2} \left( \beta-
\alpha+r \right) ^{2}\beta\,{N}^{2} \left( {\beta}^{3}+ \left( 3\,r-\alpha \right) {\beta}^{2}+3\, \left( r-\frac{\alpha}{2} \right)  \left( r-
\frac{\alpha}{6} \right) \beta+{r}^{2} \left( r-\alpha \right)  \right) } },\\
&&\!\!\!\!\!\!\!\!\!\!\!\mathcal{C}_5=( 4\,{\beta}^{9}\alpha+ ( -16\,{\alpha}^{2}+ ( 28\,r+10 ) \alpha ) {\beta}^{8}+ ( 25\,{\alpha}^{3}+
 ( -93\,r-30 ) {\alpha}^{2}+ ( 84\,{r}^{2}+65\,r+8 ) \alpha+r ) {\beta}^{7}\\
&&~~+ ( -19\,{\alpha}^{4}+
 ( 114\,r+{\frac{65}{2}} ) {\alpha}^{3}+ ( -225\,{r}^{2}-160\,r-16 ) {\alpha}^{2}+ ( 140\,{r}^{3}+180\,{r}^{2}+46
\,r+2 ) \alpha\\
&&~~+7\,{r}^{2} ) {\beta}^{6}+ ( 7\,{\alpha}^{5}+ ( -61\,r-15 ) {\alpha}^{4}+ ( 206\,{r}^{2}+{
\frac {265\,r}{2}}+10 ) {\alpha}^{3}+ ( -290\,{r}^{3}-350\,{r}^{2}-{\frac {155\,r}{2}}\\
&&~~-2 ) {\alpha}^{2}+ ( 140\,{r}^{4
}+275\,{r}^{3}+108\,{r}^{2}+11\,r ) \alpha+21\,{r}^{3} ) {\beta}^{5}+ ( -{\alpha}^{6}+ ( 12\,r+\frac{5}{2} ) {\alpha}^{5
}+ ( -70\,{r}^{2}-40\,r\\
&&~~-2 ) {\alpha}^{4}+ ( 184\,{r}^{3}+{\frac {415\,{r}^{2}}{2}}+{\frac {73\,r}{2}}+\frac{1}{2} ) {\alpha}^{
3}+ ( -210\,{r}^{4}-400\,{r}^{3}-149\,{r}^{2}-9\,r ) {\alpha}^{2}+ ( 84\,{r}^{5}\\
&&~~+250\,{r}^{4}+130\,{r}^{3}+25\,{r}^{2}
 ) \alpha+35\,{r}^{4} ) {\beta}^{4}+28\,r (  ( {\frac {5\,r}{28}}+{\frac{5}{56}} ) {\alpha}^{5}+ ( -{\frac
{33\,{r}^{2}}{28}}-\frac{5}{4}\,r-\frac{1}{7} ) {\alpha}^{4}+ ( {\frac {81\,{r}^{3}}{28}}\\
&&~~+{\frac {305\,{r}^{2}}{56}}+{\frac {99\,r}{56}}+{\frac{3}{56}} ) {\alpha}^{3}+ ( -{\frac {81\,{r}^{4}}{28}}-{\frac {125\,{r}^{3}}{14}}-{\frac {141\,{r}^{2}}{28}}-\frac{4}{7}\,r ) {\alpha}^{2}+{r}^{2} ( {r}^{3}+{\frac {135\,{r}^{2}}{28}}+{\frac {20\,r}{7}}\\
&&~~+{\frac{15}{14}} ) \alpha+\frac{5}{4}\,{r}^{4} ) {\beta}^{3}+4\,{r}^{2} (  ( -\frac{5}{4}\,{r}^{2}-\frac{5}{2}\,r-\frac{1}{2} ) {\alpha}^{4}+ (\frac{7}{2}\,{r}^{3}+{\frac{25}{2}}\,{r}^{2}+{\frac {59\,r}{8}}+\frac{3}{8} ) {\alpha}^{3}+ ( -{\frac {13\,{r}^{4}}{4}}\\
&&~~-20\,{r}^{3}-16\,{r}^{2}-\frac{7}{2}\,r ) {\alpha}^{2}+{r}^{2} ( {r}^{3}+10\,{r}^{2}+\frac{9}{2}\,r+5 ) \alpha+{\frac {21\,{r}^{4}}{4}} ) {\beta}^{2}+5\, (  ( {r}^{2}+{\frac {13\,r}{10}}+\frac{1}{10} ) {\alpha}^{3}\\
&&~~+ ( -2\,{r}^{3}-{\frac {19\,{r}^{2}}{10}}-\frac{6}{5}\,r ) {\alpha}^{2}+{r}^{2} ( {r}^{2}-\frac{4}{5}\,r+\frac{7}{5} )\alpha+\frac{7}{5}\,{r}^{4} ) {r}^{3}\beta+{r}^{5} ( r-\alpha )  (  ( 1-r ) \alpha+{r}^{2} )  ).
\end{eqnarray*}}
\end{cor}
\begin{proof}
According to Poincare normal form theory, we know that mapping \eqref{eq5.5} can be conjugated with the following mapping
\begin{eqnarray}
&&\omega\rightarrow\tilde{\mu}(\epsilon,\varepsilon)\,\omega+c_1(\epsilon,\varepsilon)\,\omega^2\,\bar{\omega}
+c_2(\epsilon,\varepsilon)\,\omega^2\,\bar{\omega}^3+O(|\omega|^6),
\label{psmap}
\end{eqnarray}
by a near-identity transformation, where $\tilde{\mu}(\epsilon,\varepsilon)|_{\epsilon=\varepsilon=0}=\tilde{\mu}(0,0)=\tilde{\mu}_0$, $c_1(\epsilon,\varepsilon)|_{\epsilon=\varepsilon=0}=c_1(0,0)=c_1$ and $c_2(\epsilon,\varepsilon)|_{\epsilon=\varepsilon=0}=c_2(0,0)=c_2$, $\tilde{\mu}_0$, $c_1$ and $c_2$ are shown in mapping \eqref{smap}.

Since $\Upsilon_8\neq0$, by applying the parameter variation \eqref{p-2}, mapping \eqref{psmap} is transformed into the following form
\begin{eqnarray}
\begin{aligned}
\omega\rightarrow e^{{\bf{i}}\,\tilde{\theta} (\epsilon_0,\varepsilon_0)}
(1+\epsilon_0+(\varepsilon_0+{\bf{i}}\,D_1(\epsilon_0,\varepsilon_0))\,\omega_1\,\bar{\omega}^2_1~~~~~\\+(D_2(\epsilon_0,\varepsilon_0)
+{\bf{i}}\,E_2(\epsilon_0,\varepsilon_0))\,\omega^2_1\,\bar{\omega}^3_1)+O(|\omega_1|^6),
\end{aligned}
\label{psmap2}
\end{eqnarray}
where $\tilde{\theta} (\epsilon_0,\varepsilon_0)$, $D_1(\epsilon_0,\varepsilon_0)=Im(c_1(\epsilon,\varepsilon))$, $D_2(\epsilon_0,\varepsilon_0)
=Re(c_2(\epsilon,\varepsilon))$ and $E_2(\epsilon_0,\varepsilon_0)=Im(c_2(\epsilon,\varepsilon))$ are smooth real-valued functions of $(\epsilon_0,\varepsilon_0)$. For mapping \eqref{psmap2}, truncating $O(|\omega_1|^6)$ terms, one can get the following mapping,
\begin{eqnarray}
\begin{aligned}
\omega\rightarrow e^{{\bf{i}}\,\tilde{\theta} (\epsilon_0,\varepsilon_0)}
(1+\epsilon_0+(\varepsilon_0+{\bf{i}}\,D_1(\epsilon_0,\varepsilon_0))\,\omega_1\,\bar{\omega}^2_1\\+(D_2(\epsilon_0,\varepsilon_0)
+{\bf{i}}\,E_2(\epsilon_0,\varepsilon_0))\,\omega^2_1\,\bar{\omega}^3_1).
\end{aligned}
\label{psmap3}
\end{eqnarray}

Due to Takens' Theorem (see \cite{Takens}), mappings \eqref{psmap2} and \eqref{psmap3} are topologically equivalent, thus these two mappings have the same dynamical properties. Next, we only need to study the dynamic properties of truncated mappings \eqref{psmap3}.
By employing polar coordinates $(\rho, \varphi)$, i.e., setting $\omega_1=\rho e^{{\bf{i}}\,\varphi}$, we obtain the following representation of \eqref{psmap3}
\begin{eqnarray}
\begin{cases}
\begin{array}{l}
\rho\rightarrow \rho+\epsilon_0\,\rho+\varepsilon_0\,\rho^3+L_2(\epsilon_0,\varepsilon_0)\,\rho^5+R(\rho,\epsilon_0,\varepsilon_0)\rho^6,\\
\varphi\rightarrow \varphi+\tilde{\theta}(\epsilon_0,\varepsilon_0)+Q(\rho,\epsilon_0,\varepsilon_0)\rho^2,
\end{array}
\end{cases}
\label{polarmap1}
\end{eqnarray}
where
$R$ and $Q$ are smooth functions, and
$$L_2(\epsilon_0,\varepsilon_0)=\frac{(D_1(\epsilon_0,\varepsilon_0))^2}{2(1+\epsilon_0)}+D_2(\epsilon_0,\varepsilon_0).$$

Furthermore, we see that the first mapping in \eqref{polarmap1} is independent of $\varphi$. Hence, the first mapping in \eqref{polarmap1} can be studied
separately. The $\varphi-$map describes the rotation through a $\rho-$dependent angle that is approximately equal to $\tilde{\theta}(\epsilon_0,\varepsilon_0)$. By $L_2(0,0)=L_2<0$, the bifurcation sketch diagram of mapping \eqref{polarmap1} can be referred to \cite[pp.419-422]{Kuznetsov}. Accordingly, \eqref{polarmap1} has the following dynamic properties:
\begin{description}
    \item[$\bullet$] when the parameter $(\epsilon,\varepsilon)$ crosses the half line $N_{-}$, i.e.,
    $$\{(\epsilon_0,\varepsilon_0)|\epsilon_0=0, \varepsilon_0<0\},$$
    from the region \ding{172} to the region \ding{173}, mapping \eqref{polarmap1} undergoes a Neimark-Sacker bifurcation and generates a stable invariant circle.
    \item[$\bullet$] when the parameter $(\epsilon,\varepsilon)$ crosses the half line $N_{+}$, i.e.,
    $$\{(\epsilon_0,\varepsilon_0)|\epsilon_0=0, \varepsilon_0>0\},$$
    from the region \ding{173} to the region \ding{174}, mapping \eqref{polarmap1} still undergoes a Neimark-Sacker bifurcation and produces an unstable invariant circle. Hence, there are two invariant circles, stable ``outer'' one and unstable ``inner'' one, in the region \ding{174}.
    \item[$\bullet$] when the parameter $(\epsilon,\varepsilon)$ lies on the curve $T_c$, the two invariant circles collide and coincide, mapping \eqref{polarmap1} has a unique invariant circle, stable from the outside and unstable from the inside,
\end{description}
where the regions \ding{172}, \ding{173} and \ding{174} are
\begin{eqnarray*}
    &&\{(\epsilon_0,\varepsilon_0)|\epsilon_0<\frac{\varepsilon^2_0}{4L_2}+o(\varepsilon^2_0),\varepsilon_0>0\}\cup \{(\epsilon_0,\varepsilon_0)|\epsilon_0<0,\varepsilon_0<0\},~~  \{(\epsilon_0,\varepsilon_0)|\epsilon_0>0\}
\end{eqnarray*}
and
\begin{eqnarray*}
    &&\{(\epsilon_0,\varepsilon_0)|\epsilon_0>\frac{\varepsilon^2_0}{4L_2}+o(\varepsilon^2_0),\varepsilon_0>0\}\cup \{(\epsilon_0,\varepsilon_0)|\epsilon_0<0,\varepsilon_0>0\},
\end{eqnarray*}
respectively.
Since the fixed point of the $\rho-$map corresponds to an invariant circle of the truncated normal form \eqref{psmap3}, and mappings \eqref{psmap3}, \eqref{eq2.1} and \eqref{polarmap1} are topologically equivalent, then they have the same dynamical properties as described above.
Therefore, by using parameter transformations \eqref{p-1} and \eqref{p-2}, we obtain the bifurcation phenomena of mapping \eqref{eq2.1} regarding to parameter $(\epsilon,\varepsilon)$. This completes the whole proof.
\end{proof}

Moreover, for $L_2>0$, we obtain the following conclusions,
\begin{cor}\label{cor6.12}
Under the conditions of Theorem \ref{th6.1}, if $L_2>0$, then there exists the following bifurcation phenomena when the parameters lie in a sufficiently small neighborhood of point $(r, \alpha)=(-(\beta^2-1)/\beta,-1/(\beta\,(\beta-1)))$,
\begin{description}
    \item[$(1)$] If the conditions of Theorem \ref{th5.1} hold, and as the parameters cross $\mathfrak{L}$, mapping \eqref{eq2.1} undergoes a Neimark-Sacker bifurcation and generates an unique unstable invariant circle.
    \item[$(2)$] When the parameters belong to the region $\mathfrak{G_1}$, mapping \eqref{eq2.1} have two invariant circles. The outside is unstable, while the inside is stable.
    \item[$(3)$] As the paremeters near the region $\mathfrak{G_2}$, two invariant circles of mapping \eqref{eq2.1} become a unique invariant circle which is unstable from the outside and stable from the inside.
\end{description}
\end{cor}

From the biological point of view, if mapping \eqref{eq2.1} undergoes a Chenciner bifurcation, then the numbers of susceptible, infective and recovered individuals will exhibit periodic or quasi-periodic fluctuations. Besides, similar to the Neimark-Sacker bifurcation, whether the disease is under control depends on the stability of the invariant cycle.
\subsection{$1:3$ resonance at $E_2$}
\allowdisplaybreaks[4]
In this subsection, we focus on the 1:3 resonance phenomena of mapping \eqref{eq2.1} at $E_2$,
which first proposed by Horozov (see \cite{Horozov}). As we know, if the eigenvalue of Jacobian matrix of a two-dimensional mapping at corresponding fixed point is
$$-\frac{1}{2}+\frac{\sqrt{3}}{2}\,{\bf{i}}~~\mbox{and}~~-\frac{1}{2}-\frac{\sqrt{3}}{2}\,{\bf{i}},$$
then it may undergo a $1:3$ resonance near the fixed point.

Furthermore, regarding mapping \eqref{eq2.1}, when the following conditions are satisfied
{\small\begin{eqnarray*}
&&\!\!\!\!\!\!\!{P_{E_2}}(-\frac{1}{2}+\frac{\sqrt{3}}{2}\,{\bf{i}})={\frac {2\,{\beta}^{3}+ \left( 4\,r-2\,\alpha \right) {\beta}^{2}+
 \left( -2\,r\alpha+2\,{r}^{2}+3\,\alpha-3 \right) \beta-3\,r}{2\,
\beta+2\,r}}\\
&&~~~~~~~~~~~~~~~~~~~~~~~-{\frac {{\bf{i}}\sqrt {3} \left( \beta\,\alpha-3\,\beta-3\,r \right) }{2\,
\beta+2\,r}}=0,\\
&&\!\!\!\!\!\!\!{P_{E_2}}(-\frac{1}{2}-\frac{\sqrt{3}}{2}\,{\bf{i}})={\frac {2\,{\beta}^{3}+ \left( 4\,r-2\,\alpha \right) {\beta}^{2}+
 \left( -2\,r\alpha+2\,{r}^{2}+3\,\alpha-3 \right) \beta-3\,r}{2\,
\beta+2\,r}}\\
&&~~~~~~~~~~~~~~~~~~~~~~~+{\frac {{\bf{i}}\sqrt {3} \left( \beta\,\alpha-3\,\beta-3\,r \right) }{2\,
\beta+2\,r}}=0,\\
&&\!\!\!\!\!\!\!\Delta({P_{E_2}}(t))<0, N>0, 0<\beta<1, 0<r<1,
\end{eqnarray*}}the $1:3$ resonance may occur near $E_2$. By solving the above semi-algebraic system, we get
$$r=-\frac{\beta^2-3\,\beta+3}{\beta-3}~~\mbox{and}~~\alpha=-\frac{9}{\beta\,(\beta-3)}.$$

As a result, the following conclusions will be obtained.
\begin{thm}
If the parameter $(r,\alpha)$ varies near the point $(-(\beta^2-3\,\beta+3)/(\beta-3),-9/\beta\,(\beta-3))$, then mapping \eqref{eq2.1} undergoes $1:3$ resonance.
\label{th6.2}
\end{thm}

\begin{proof}
Applying the transformation
$$(x,y,z)=(u_9-{\frac {N \left( \beta+r \right) }{\alpha}},v_9-{\frac {\beta\,N \left( \alpha-\beta-r \right) }{\alpha\, \left( \beta+r \right) }},w_9-{\frac {rN \left( \alpha-\beta-r \right) }{\alpha\, \left( \beta+r\right) }}),$$
we translate $E_2$ to the origin O.
Then let $\breve{\varepsilon}=(\varepsilon_1, \varepsilon_2):=(r+(\beta^2-3\,\beta+3)/(\beta-3),\alpha+9/\beta\,(\beta-3))$ as a new parameter, mapping \eqref{eq2.1} can be changed into
\begin{eqnarray}
\left(
\begin{array}{cc}
u_9 \\
\\
v_9 \\
\\
w_9
\end{array}
\right)\!\!\!
\rightarrow\!\!\!
\left(
\begin{array}{llll}
-{\frac { \left( \varepsilon_2\,{\beta}^{2}-3\,\varepsilon_2\,\beta-\varepsilon_1\,\beta+3\,\varepsilon_1-6 \right) \,u_9 }{\varepsilon_1\,\beta-3\,\varepsilon_1-3}}-{\frac { \left( \varepsilon_1\,\beta-3\,\varepsilon_1-3 \right)\,v_9 }{
\beta-3}}-{\frac {\left( \varepsilon_2\,{\beta}^{2}-3\,\varepsilon_2\,\beta-9 \right)\,u_9\,v_9 }{N\beta\, \left( \beta-3 \right) }}\\
\noalign{\medskip}
{\frac {\left( \varepsilon_2\,{\beta}^{2}-\varepsilon_1\,{\beta}^{2}-3\,\varepsilon_2\,\beta+3\,\varepsilon_1\,\beta+3\,\beta-9 \right)\,u_9 }
{\varepsilon_1\,\beta-3\,\varepsilon_1-3}}+v_9+{\frac {\left( \varepsilon_2\,{\beta}^{2}-3\,\varepsilon_2\,\beta-9 \right)\,u_9\,v_9 }
{N\beta\, \left( \beta-3 \right) }}\\
\noalign{\medskip}
-{\frac { \left( {\beta}^{2}-\varepsilon_1\,\beta-3\,\beta+3\,
\varepsilon_1+3 \right)\,v_9}{\beta-3}}-\left( -1+\beta \right)\,w_9
\end{array}
\right)\!\!.
\label{eq6.21}
\end{eqnarray}
Next, by employing the invertible linear transformation $(u_9,v_9,w_9)^T=H(u_{10},v_{10},w_{10})^T$, where
\begin{eqnarray*}
H=\left[ \begin {array}{lll}
1&0&0\\
0&1&0\\
-1&-1&1
\end {array}
 \right],
\end{eqnarray*}
mapping \eqref{eq6.21} is converted into the following form
\begin{eqnarray}
\!\!\!\!\left[
\begin{array}{ccc}
   u_{10}  \\
   \noalign{\medskip}
   v_{10} \\
   \noalign{\medskip}
   w_{10}
\end{array}
\right]\!\!\!
\mapsto\!\!\!
\left[\begin{array}{lll}
-{\frac { \left( \varepsilon_2\,{\beta}^{2}-3\,\varepsilon_2\,\beta-\varepsilon_1\,\beta+3\,\varepsilon_1-6 \right)\,u_{10} }{\varepsilon_1\,\beta-3\,\varepsilon_1-3}}-{\frac { \left( \varepsilon_1\,\beta-3\,\varepsilon_1-3 \right)\,v_{10} }{
\beta-3}}-{\frac { \left( \varepsilon_2\,{\beta}^{2}-3\,\varepsilon_2\,\beta-9 \right)\,u_{10}\,v_{10} }{N\beta\, \left( \beta-3 \right) }}
\\
\noalign{\medskip}
  {\frac { \left( \varepsilon_2\,{\beta}^{2}-\varepsilon_1\,{\beta}^{2}-3\,\varepsilon_2\,\beta+3\,\varepsilon_1\,\beta+3\,\beta-9 \right)\,u_{10} }{\varepsilon_1\,\beta-3\,\varepsilon_1-3}}+v_{10}+{\frac {\left( \varepsilon_2\,{\beta}^{2}-3
\,\varepsilon_2\,\beta-9 \right)\,u_{10}\,v_{10} }{N\beta\, \left( \beta-3 \right) }}
 \\
\noalign{\medskip}
-\left( -1+\beta \right)\,w_{10}
 \end{array}\right].
\label{eq6.22}
\end{eqnarray}
As done in proof of Theorem \ref{th3.1}, mapping \eqref{eq6.22} has a two dimensional $C^{2}$ center manifold
\begin{eqnarray*}
w_{10}=h_{5}(u_{10},v_{10})=O(|(u_{10},v_{10})|^3),
\end{eqnarray*}
and then \eqref{eq6.22} can be converted into the following two dimensional form
\begin{eqnarray}
\!\!\!\!\left[
\begin{array}{ccc}
   u_{10}  \\
   \noalign{\medskip}
   v_{10}
\end{array}
\right]\!\!\!
\mapsto\!\!\!
\left[\begin{array}{lll}
-{\frac { \left( \varepsilon_2\,{\beta}^{2}-3\,\varepsilon_2\,\beta-\varepsilon_1\,\beta+3\,\varepsilon_1-6 \right)\,u_{10} }{\varepsilon_1\,\beta-3\,\varepsilon_1-3}}-{\frac { \left( \varepsilon_1\,\beta-3\,\varepsilon_1-3 \right)\,v_{10} }{
\beta-3}}-{\frac { \left( \varepsilon_2\,{\beta}^{2}-3\,\varepsilon_2\,\beta-9 \right)\,u_{10}\,v_{10} }{N\beta\, \left( \beta-3 \right) }}
\\
\noalign{\medskip}
  {\frac { \left( \varepsilon_2\,{\beta}^{2}-\varepsilon_1\,{\beta}^{2}-3\,\varepsilon_2\,\beta+3\,\varepsilon_1\,\beta+3\,\beta-9 \right)\,u_{10} }{\varepsilon_1\,\beta-3\,\varepsilon_1-3}}+v_{10}+{\frac {\left( \varepsilon_2\,{\beta}^{2}-3
\,\varepsilon_2\,\beta-9 \right)\,u_{10}\,v_{10} }{N\beta\, \left( \beta-3 \right) }}
 \end{array}\right].
\label{eq6.23}
\end{eqnarray}

Moreover, using the transformation
\begin{eqnarray*}
\left(
\begin{array}{cc}
u_{10} \\
v_{10}
\end{array}
\right)
\rightarrow
\left(
\begin{array}{cc}
{\frac {-{\beta}^{2}\,\varepsilon_2+3\,\beta\,\varepsilon_2+9}{ 2\,\left( \beta-3 \right)  \left( 3+ \left( \varepsilon_2-\varepsilon_1 \right) \beta
 \right) }} &{\frac {\sqrt {\chi_{1}}}{2\,
 \left( \beta-3 \right)  \left( 3+ \left( \varepsilon_2-\varepsilon_1
 \right) \beta \right) }}
\\
1&0
\end{array}
\right)
\left(
\begin{array}{l}
u_{11}
\\
v_{11}
\end{array}
\right),
\end{eqnarray*}
where
\begin{eqnarray*}
&&\!\!\!\!\!\!\!\!\chi_{1}= -4\,\beta\, ( \beta-3 ) ^{2}{\varepsilon_1}^{3}+4\, ( \beta-3 ) ( ( {\beta}^{2
}-3\,\beta ) \varepsilon_2+9\,\beta-9 ) {\varepsilon}^{2}_1+( ( -24\,{\beta}^{2}+72\,\beta )\,\varepsilon_2\\
&&~~~-108\,\beta
+216 ) \varepsilon_1+27-{\beta}^{2} ( \beta-3 ) ^{2}{\varepsilon_2}^{2}+( 18\,{\beta}^{2}-18\,\beta ) \,\varepsilon_2,
\end{eqnarray*}
we change mapping \eqref{eq6.23} into the mapping
\begin{eqnarray}
\begin{aligned}
\left(
\begin{array}{cc}
u_{11} \\
v_{11}
\end{array}
\right)
\rightarrow
\left(
\begin{array}{l}
k_{11}\,u_{11}+k_{12}\,v_{11}+p_{20}\,u^2_{11}+p_{11}\,u_{11}\,v_{11}\\
k_{21}\,u_{11}+k_{22}\,v_{11}+q_{20}\,u^2_{11}+q_{11}\,u_{11}\,v_{11}\\
\end{array}
\right),
\end{aligned}
\label{eq6.24}
\end{eqnarray}
where
{\footnotesize\begin{eqnarray*}
&&\!\!\!\!\!\!\!\!k_{11}=k_{22}:=-{\frac {{\beta}^{2}\,\varepsilon_2-3\,\beta\,\varepsilon_2-2\,\beta\,\varepsilon_1+
6\,\varepsilon_1-3}{2\,\beta\,\varepsilon_1-6\,\varepsilon_1-6}},~k_{12}=-k_{21}:={\frac {q\sqrt {\chi_1}}{2\,\beta\,\varepsilon_1-6\,\varepsilon_1-6}},\\
&&\!\!\!\!\!\!\!\!p_{20}:=-{\frac {{\beta}^{4}{\varepsilon}^{2}_2-6\,{\beta}^{3}{\varepsilon}^{2}_2+9
\,{\beta}^{2}{\varepsilon}^{2}_2-18\,{\beta}^{2}\varepsilon_2+54\,\beta\,\varepsilon_2+81}{2\,N \left( \beta-3 \right) ^{2}\beta\, \left( \beta\,
\varepsilon_2-\beta\,\varepsilon_1+3 \right) }},\\
&&\!\!\!\!\!\!\!\!p_{11}:={\frac {\sqrt {\chi_1} \left( {\beta}^{2}\varepsilon_2-3\,\beta\,\varepsilon_2-9 \right) }{2\,N \left( \beta-3 \right) ^{2}\beta\, \left( \beta\,\varepsilon_2-\beta\,\varepsilon_1+3 \right) }},~q_{20}={\frac {\chi_2 }{2\,N ( \beta-3 ) ^{2}\beta\, ( \beta\,\varepsilon_2-\beta\,\varepsilon_1+3 ) \sqrt {\chi_1}}},\\
&&\!\!\!\!\!\!\!\!\chi_2:= {\beta}^{6}{\varepsilon}^{3}_2-2\,{\beta}^{6}{\varepsilon}^{2}_2\varepsilon_1-9\,{\beta}^{5}{\varepsilon}^{3}_2
+18\,{\beta}^{5}{\varepsilon}^{2}_2\varepsilon_1+6\,{\beta}^{5}{\varepsilon}^{2}_2+27\,{\beta}^{4}{\epsilon}^{3}
-54\,{\beta}^{4}{\varepsilon}^{2}_2\varepsilon_1-63\,{\beta}^{4}{\varepsilon}^{2}_2+36\,{\beta}^{4}\varepsilon_2\,\varepsilon_1\\
&&~~~-27\,{\beta}^{3}{
\varepsilon}^{3}_2+54\,{\beta}^{3}{\varepsilon}^{2}_2\varepsilon_1
+216\,{\beta}^{3}{\varepsilon}^{2}_2-216\,{\beta}^{3}\varepsilon_2\,\varepsilon_1-108\,{\beta}^{3}
\varepsilon_2-243\,{\beta}^{2}{\varepsilon}^{2}_2+324\,{\beta}^{2}\varepsilon_2\,
\varepsilon_1+567\,{\beta}^{2}\varepsilon_2\\
&&~~~-162\,{\beta}^{2}\varepsilon_1-729\,\beta\,\varepsilon_2+486\,\beta\,\varepsilon_1+486\,\beta-729,\\
&&\!\!\!\!\!\!\!\!q_{11}:=-{\frac{\chi_3}{2\,
N \left( \beta-3 \right) ^{2}\beta\, \left( \beta\,\varepsilon_2-\beta\,\varepsilon_1+3 \right) }},\\
&&\!\!\!\!\!\!\!\!\chi_3:={\beta}^{4}{\varepsilon}^{2}_2-2\,{\beta}^{4}\varepsilon_2\,\varepsilon_1-6\,{\beta}^{3}{\varepsilon}^{2}_2
+12\,{\beta}^{3}\varepsilon_2\,\varepsilon_1+6\,{\beta}^{3}\varepsilon_2+9\,{\beta}^{2}{\varepsilon}^{2}_2-18\,{
\beta}^{2}\varepsilon_2\,\varepsilon_1-36\,{\beta}^{2}\varepsilon_2+18\,{\beta}^{2
}\varepsilon_1\\
&&~~~+54\,\beta\,\varepsilon_2-54\,\beta\,\varepsilon_1-54\,\beta+81.
\end{eqnarray*}}
We rewrite \eqref{eq6.24} in the complex form
\begin{eqnarray}
\begin{array}{l}
z_1\mapsto\zeta(\breve{\varepsilon})z_1+t_{20}z_1^2+t_{11}z_1\bar{z}_1+t_{02}\bar{z}^2_1,
\end{array}
\label{eq6.25}
\end{eqnarray}
with $z_1=u_{11}+v_{11}\,{\bf{i}}$,
where
{\footnotesize\begin{eqnarray*}
&&\!\!\!\!\!\!\!\!\!\!\!\!\zeta(\breve{\varepsilon}):={\frac { \left( 2\,\beta-6 \right) \varepsilon_1-{\beta}^{2}\varepsilon_2+3\,
\beta\,\varepsilon_2+3}{2\,\beta\,\varepsilon_1-6\,\varepsilon_1-6}}-{\frac {{\bf i}\,\sqrt {\chi_1}}{2\,\beta\,\varepsilon_1-6\,\varepsilon_1-6}},\\
&&\!\!\!\!\!\!\!\!\!\!\!\!t_{20}:={\frac {-{\beta}^{2}\varepsilon_2+3\,\beta\,\varepsilon_2+9}{2\,\beta\, \left( \beta-3 \right) N}}\\
&&\!\!\!+{\frac {{\bf i}\, \left( {\beta}^{2}\varepsilon_2-3\,\beta\,\varepsilon_2-9 \right) \left( {\beta}^{3}\varepsilon_2+ \left( -2\,{\varepsilon}^{2}_1-6\,\varepsilon_2
 \right) {\beta}^{2}+ \left( 12\,{\varepsilon}^{2}_1+9\,\varepsilon_2+12\,\varepsilon_1-9 \right) \beta-18\,{\varepsilon}^{2}_1-36\,\varepsilon_1+9
 \right) }{2\,\sqrt {\chi_1}N \left( \beta-3 \right) ^{2}\beta}},\\
&&\!\!\!\!\!\!\!\!\!\!\!\!t_{11}:=-{\frac { \left( {\beta}^{2}\varepsilon_2-3\,\beta\,\varepsilon_2-9
 \right) ^{2}}{4\,N \left( 3+ \left( \varepsilon_2-\varepsilon_1 \right) \beta
 \right)  \left( \beta-3 \right) ^{2}\beta}}\\
 &&~~+{\frac {{\bf i}\left( -9+ \left( \varepsilon_2-2\,\varepsilon_1 \right) {\beta}^{2}+ \left( -3\,\varepsilon_2+6\,\varepsilon_1+6 \right) \beta \right) \left( {\beta}^{2}\epsilon-3\,\beta\,\varepsilon_2-9 \right) ^{2}}{\sqrt {
\chi_1} \left( \beta-3 \right) ^{2}\beta\,N \left( 3+ \left( \varepsilon_2-\varepsilon_1 \right) \beta \right) }},\\
&&\!\!\!\!\!\!\!\!\!\!\!\!t_{02}:=-\frac { \left( \beta\,\varepsilon_1-3\,\varepsilon_1-3 \right)
 \left( {\beta}^{2}\varepsilon_2-3\,\beta\,\varepsilon_2-9 \right) }{2\,N \left( 3+ \left( \varepsilon_2-\varepsilon_1 \right) \beta \right)  \left( \beta-3 \right) ^{2}}-\frac {{\bf i} \,\chi_4}{2\,\sqrt {\chi_1}N \left( \beta-3
 \right) ^{2} \left( 3+ \left( \varepsilon_2-\varepsilon_1 \right) \beta \right) \beta},\\
&&\!\!\!\!\!\!\!\!\!\!\!\!\chi_4:=\left( {\beta}^{3}\varepsilon_2+ \left( -2\,\varepsilon_2\,
\varepsilon_1+2\,{\varepsilon}^{2}_1-3\,\varepsilon_2 \right) {\beta}^{2}+
 \left( -6\,{\varepsilon}^{2}_1+ \left( 6\,\varepsilon_2-12 \right)
\varepsilon_1+6\,\varepsilon_2-9 \right) \beta+18\,\varepsilon_1+18 \right)
 ( \beta\,\varepsilon_1\\
&&~-3\,\varepsilon_1-3 )  \left( {\beta}^{2
}\varepsilon_2-3\,\beta\,\varepsilon_2-9 \right),
\end{eqnarray*}}
\!\!\!By Lemma $9.11$ in \cite[p.429]{Kuznetsov}, for small $|\breve{\varepsilon}|$, mapping \eqref{eq6.25} can be changed by an near-identity transformation into the mapping
\begin{eqnarray}
w\mapsto\Gamma_{\breve{\varepsilon}}(w):=\zeta(\breve{\varepsilon})w+B(\breve{\varepsilon})\bar{w}^2+A(\breve{\varepsilon})w^2\bar{w}+O(|w|^4),
\label{eq6.26}
\end{eqnarray}
where
\begin{eqnarray*}
&&B(\breve{\varepsilon}):=t_{02}\\
&&A(\breve{\varepsilon}):=\frac{t_{20}t_{11}(2\zeta(\breve{\varepsilon})+\bar{\zeta}(\breve{\varepsilon})-3)}{(\bar{\zeta}(\breve{\varepsilon})-1)
(\zeta^2(\breve{\varepsilon})-\zeta(\breve{\varepsilon}))}
+\frac{|t_{11}|^2}{1-\bar{\zeta}(\breve{\varepsilon})}+t_{21}.
\end{eqnarray*}

Besides, from Lemma $9.12$ in \cite[Lemma 9.12, p.448]{Kuznetsov}, it follows that for sufficiently small $|\breve{\varepsilon}|$ the third iterate of the mapping \eqref{eq6.26} can be represented in the form
\begin{eqnarray*}
\Gamma_{\breve{\varepsilon}}^3(w)=\phi(1,w,\breve{\varepsilon})+O(|w|^4),
\end{eqnarray*}
where $\phi(1,w,\breve{\varepsilon})$ is the flow of a differential equation
\begin{eqnarray}
\begin{aligned}
\frac{dw}{dt}=\varrho(\breve{\varepsilon})w+B_1(\breve{\varepsilon})\bar{w}^2+A_1(\breve{\varepsilon})w^2\bar{w},
\end{aligned}
\label{eq3.04}
\end{eqnarray}
with
\begin{eqnarray*}
&&\varrho(\breve{\varepsilon}):=(-3\,\beta+\frac{9}{2}+\frac{(4\,\beta-9)\,\sqrt{3}\,{\bf i}}{2})\,\varepsilon_1+(\frac{\beta^2}{2}-\frac{\beta\,(\beta-2)\,\sqrt{3}\,{\bf i}}{2})\,\varepsilon_2+O(|\varepsilon|^2),\\
&&B_1(\breve{\varepsilon}):=3\bar{\zeta}(\breve{\varepsilon})B(\breve{\varepsilon}),\\
&&A_1(\breve{\varepsilon}):=-3|B(\varepsilon)|^2+3\zeta^2(\breve{\varepsilon})A(\breve{\varepsilon}).
\end{eqnarray*}

Furthermore, let $\varrho(\breve{\varepsilon})=\sigma_1+\sigma_2{\bf{i}}$ in the system \eqref{eq3.04}, where
\begin{eqnarray}
\begin{cases}
\begin{array}{l}
\sigma_1=(-3\,\beta+\frac{9}{2})\,\varepsilon_1+\frac{\beta^2\,\varepsilon_2}{2}+O(|\varepsilon|^2),\\
\sigma_2=\frac{(4\,\beta-9)\,\sqrt{3}\,\varepsilon_1}{2}-\frac{\beta\,(\beta-2)\,\sqrt{3}\,\varepsilon_2}{2}+O(|\varepsilon|^2).
\end{array}
\end{cases}
\label{eq3.05}
\end{eqnarray}
One can compute that
\begin{eqnarray*}
det
\left.
\left(
\begin{array}{ll}
\frac{\partial\sigma_1}{\partial\varepsilon_1}&\frac{\partial\sigma_1}{\partial\varepsilon_2}\\
\frac{\partial\sigma_2}{\partial\varepsilon_1}&\frac{\partial\sigma_2}{\partial\varepsilon_2}
\end{array}
\right)
\right|_{(\varepsilon_1,\varepsilon_2)=(0,0)}
=\frac{\sqrt {3}{\beta}^{3}}{2}-3\,\sqrt {3}\,{\beta}^{2}+\frac{9\,\sqrt {3}\,\beta}{2}.
\end{eqnarray*}
Then the condition $(R3.0)$ of resonance shown in \cite[p.451]{Kuznetsov} is satisfied. Regarding $\sigma=(\sigma_1,\sigma_2)\in\mathbb{R}^2$ as a new parameter, system \eqref{eq3.04} is rewritten as the form
\begin{eqnarray}
\frac{dw}{dt}=(\sigma_1+\sigma_2\,{\bf{i}})w+b_1(\sigma)\bar{w}^2+c_1(\sigma)w^2\bar{w}+O(|\sigma|^2|w|,|w|^4).
\label{eq3.06}
\end{eqnarray}
It is clear that
$$b_1(0)=-{\frac { \left( 27\,{\bf i}\,\sqrt {3}+27 \right)  \left( \beta\,{\bf i}-\frac{\beta
\,\sqrt {3}}{3}-2\,{\bf i} \right) \sqrt {3}}{8\,N \left( \beta-3 \right) ^{2}
\beta}}
~~{\mbox{and}}~~Re(c_1(0))=-{\frac {972\,\beta-729}{8\,{N}^{2} \left( \beta-3 \right) ^{4}\beta}},$$
implying that the conditions $(R3.1)$ and $(R3.2)$ in \cite[p.451]{Kuznetsov} are satisfied, respectively. Hence, mapping \eqref{eq2.1} undergoes the $1:3$ resonance.
\end{proof}
\begin{cor}
Under the condition of Theorem \ref{th6.2}, if
$${\frac {-4\,{\beta}^{2}+3\,\beta}{6\,{\beta}^{2}-18\,\beta+18}}<0,$$
then mapping \eqref{eq2.1} possesses the following results when the parameter $(r, \alpha)$ in a sufficiently small neighborhood of point $(-(\beta^2-3\,\beta+3)/(\beta-3),-9/\beta\,(\beta-3))$:
\begin{description}
    \item[$(1)$] Mapping $\eqref{eq2.1}$ always has a fixed point $E_2$. As $(r,\alpha)$ crosses the line
    $$
    (r,\alpha)\in \mathfrak{L}:=\{N>0, 0<\beta<1, r=-\frac{\beta^2-3\,\beta+3}{\beta-3}, \alpha=\Psi_3\},
    $$
    \eqref{eq2.1} undergoes a non-degenerate Neimark-Sacker bifurcation.
    \item[$(2)$] Mapping $\eqref{eq2.1}$ always has a saddle cycle of period-3 $\{T_{11}, T_{12},T_{13}\}$.
    \item[$(3)$] When $(r,\alpha)$ lies in the curve
    {\footnotesize\begin{eqnarray*}
    &&\!\!\!\!\!\!\!\!\!\!\!\!\!\!\!\!\!\!H_{3r}:=\{(r,\alpha)|-{\frac {(16\,\sqrt {3}{\beta}^{3}-48\,\sqrt {3}{\beta}^{2}+72\,{
        \beta}^{3}+27\,\sqrt {3}\beta-324\,{\beta}^{2}+540\,\beta-324)\,( r+{\frac {{\beta}^{2}-3\,\beta+3}{\beta-3}}
         )}{24\,({\beta}^{2}-3\,\beta+3} )}\\
    &&~~~~~~~~~~+{\frac {{\beta}^{2} \left( 4\,\sqrt {3}{\beta}^{2}-11
          \,\sqrt {3}\beta+12\,{\beta}^{2}+6\,\sqrt {3}-36\,\beta+36 \right) \left( \alpha+{\frac {9}{\beta\, \left(
          \beta-3 \right) }} \right)}{24\,({
          \beta}^{2}-3\,\beta+3} ) }\\
    &&~~~~~~~~~~+O(|(r+{\frac {{\beta}^{2}-3\,\beta+3}{\beta-3}}
           ,\alpha+{\frac {9}{\beta\, \left( \beta-3 \right) }})|^2)=0    \},
    \end{eqnarray*}}the saddle cycle of period-3 cycle lies in the invariant circle, which forms a homoclinic structure.
\end{description}
\label{cor6.21}
\end{cor}
\begin{proof}
In order to clarify the specific bifurcation phenomenon, we use the transformation $w=\gamma(\sigma)\eta$ with
\begin{eqnarray*}
\gamma(\sigma)=\frac{e^{(\frac{arg(b_1(\sigma))}{3}\,{\bf{i}})}}{|b_1(\sigma)|}
\end{eqnarray*}
to change system \eqref{eq3.06} into the form
\begin{eqnarray}
\frac{d\eta}{dt}=((\sigma_1+\sigma_2\,{\bf{i}})\eta+\bar{\eta}^2+c(\sigma)\eta^2\bar{\eta}+O(|\sigma|^2|\eta|,|\eta|^4),
\label{eq3.07}
\end{eqnarray}
where
$$
c(\sigma)=\frac{c_1(\sigma)}{|b_1(\sigma)|^2}
$$
In the polar coordinates $\eta=\rho e^{{\bf{i}}\iota}$ system \eqref{eq3.07} has the following approximating system:
\begin{eqnarray}
\begin{cases}
\begin{array}{l}
\frac{d\rho}{dt}=\sigma_1\rho+\rho^2cos(3\iota)+R_c(\sigma)\rho^3\\
\frac{d\iota}{dt}=\sigma_2-\rho sin(3\iota)+I_c(\sigma)\rho^2
\end{array}.
\end{cases}
\label{eq3.08}
\end{eqnarray}
where $R_c(\sigma)=Re(c(\sigma))$ and $I_c(\sigma)=Im(c(\sigma))$ are the real part and imaginary part of $c(\sigma)$, respectively.
\begin{figure}[htbp]
\centering
{
\includegraphics[width=10cm]{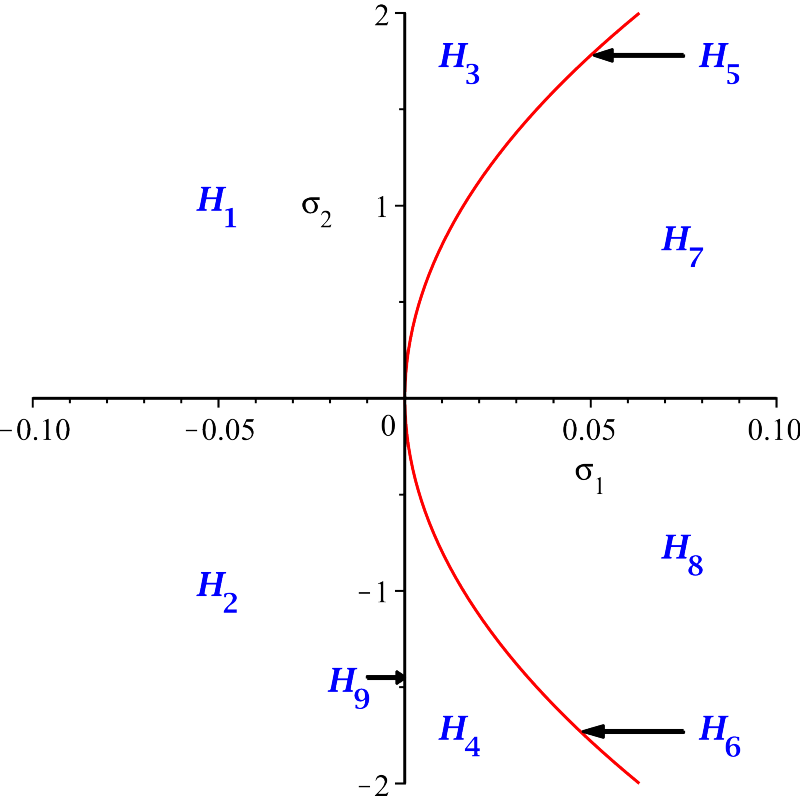}
}
\caption{Bifurcation diagram of system \eqref{eq3.08}.}
\label{resonance1}
\end{figure}

Since
$$R_c(0)={\frac {-4\,{\beta}^{2}+3\,\beta}{6\,{\beta}^{2}-18\,\beta+18}}<0,$$
we obtain the bifurcation diagram of the approximating system \eqref{eq3.08}, which is composed of Figures \ref{resonance1} and \ref{1bi3YJQ1}.
To be precise, when the parameters are located in regions $H_i$ $(i=1,2,\cdots,8)$ in Figure \ref{resonance1}, their corresponding phase diagrams are Figures \ref{Y3-1}-\ref{Y3-8} respectively.
The more general explanations for the corresponding bifurcation diagrams are as follows.
\begin{figure}[ht!]
\T\T\T\T\T\T\T\T\T\T\T\T\T\T\T\T\T\T\T\T\T\T\T\T\T\T\T\T\T\T\T\T\T\T\T\T
\subfigure[The phase diagrams of the region $H_1$ in Figure \ref{resonance1}.]{
\includegraphics[width=8cm]{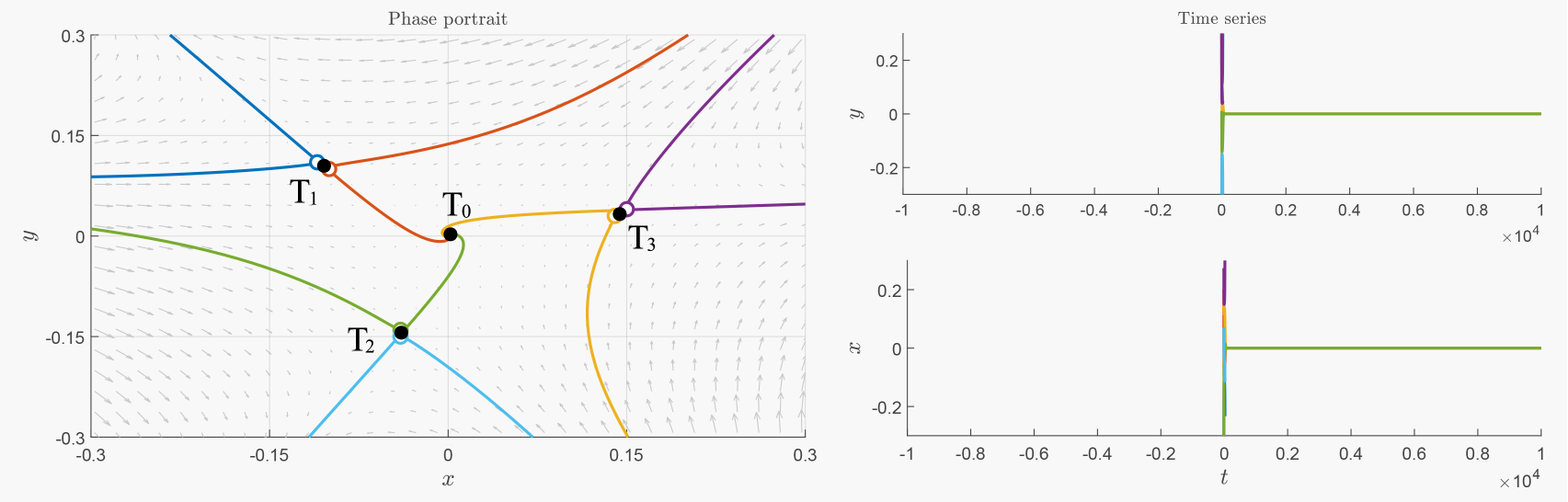}
\label{Y3-1}
}
\quad
\quad
\T\T\T\T\T\T\T\T\T\T\T\T\T\T\T\T\T\T\T\T\T\T\T\T\T\T\T\T\T\T\T\T\T\T\T\T
\subfigure[The phase diagrams of the region $H_2$ in Figure \ref{resonance1}.]{
\includegraphics[width=8cm]{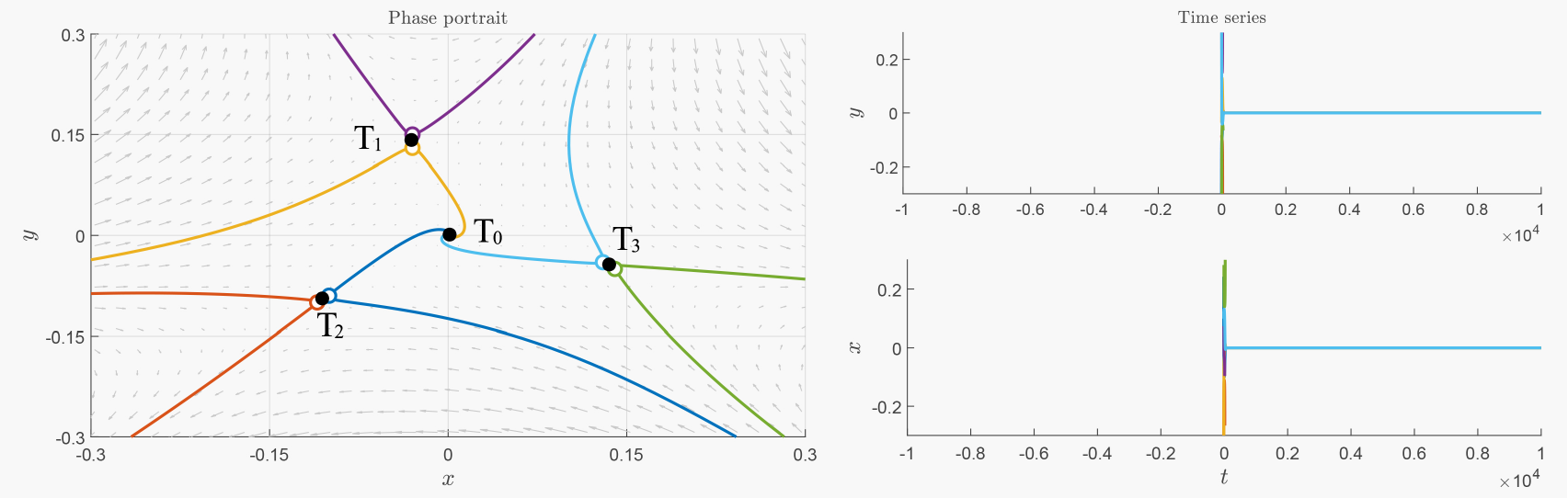}
\label{Y3-2}
}
\quad
\T\T\T\T\T\T\T\T\T\T\T\T\T\T\T\T\T\T\T\T\T\T\T\T\T\T\T\T\T\T\T\T\T\T\T\T
\subfigure[The phase diagrams of the region $H_3$ in Figure \ref{resonance1}.]{
\includegraphics[width=8cm]{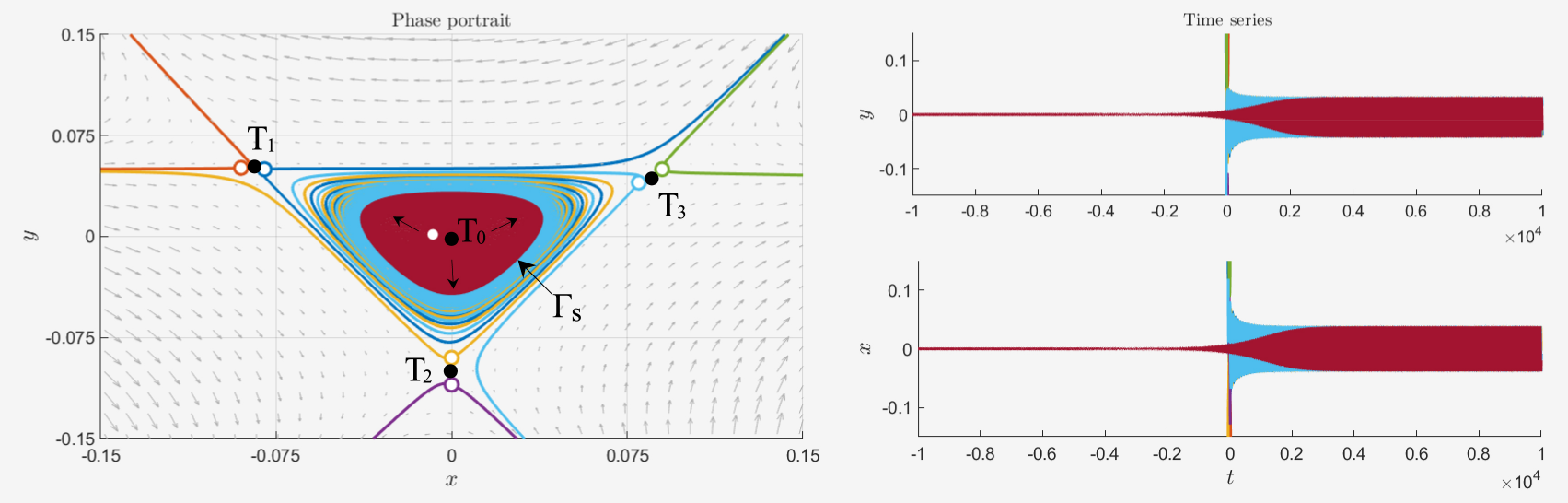}
\label{Y3-3}
}
\quad
\quad
\T\T\T\T\T\T\T\T\T\T\T\T\T\T\T\T\T\T\T\T\T\T\T\T\T\T\T\T\T\T\T\T\T\T\T\T
\subfigure[The phase diagrams of the region $H_4$ in Figure \ref{resonance1}.]{
\includegraphics[width=8cm]{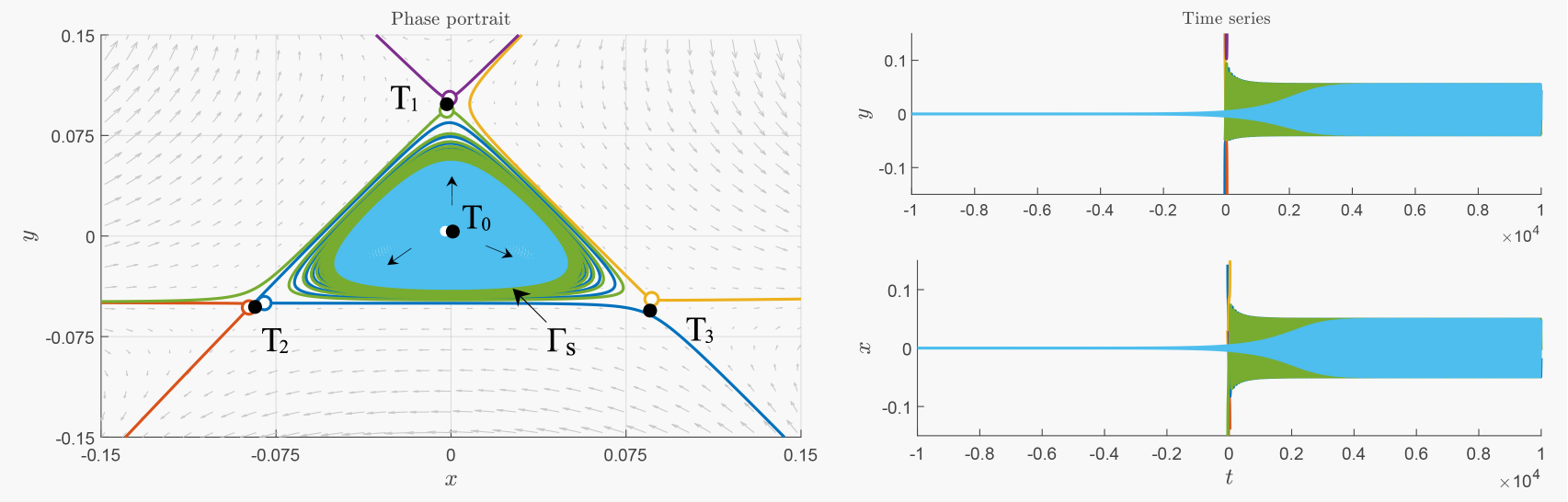}
\label{Y3-4}
}
\quad
\T\T\T\T\T\T\T\T\T\T\T\T\T\T\T\T\T\T\T\T\T\T\T\T\T\T\T\T\T\T\T\T\T\T\T\T
\subfigure[The phase diagrams of the region $H_5$ in Figure \ref{resonance1}.]{
\includegraphics[width=8cm]{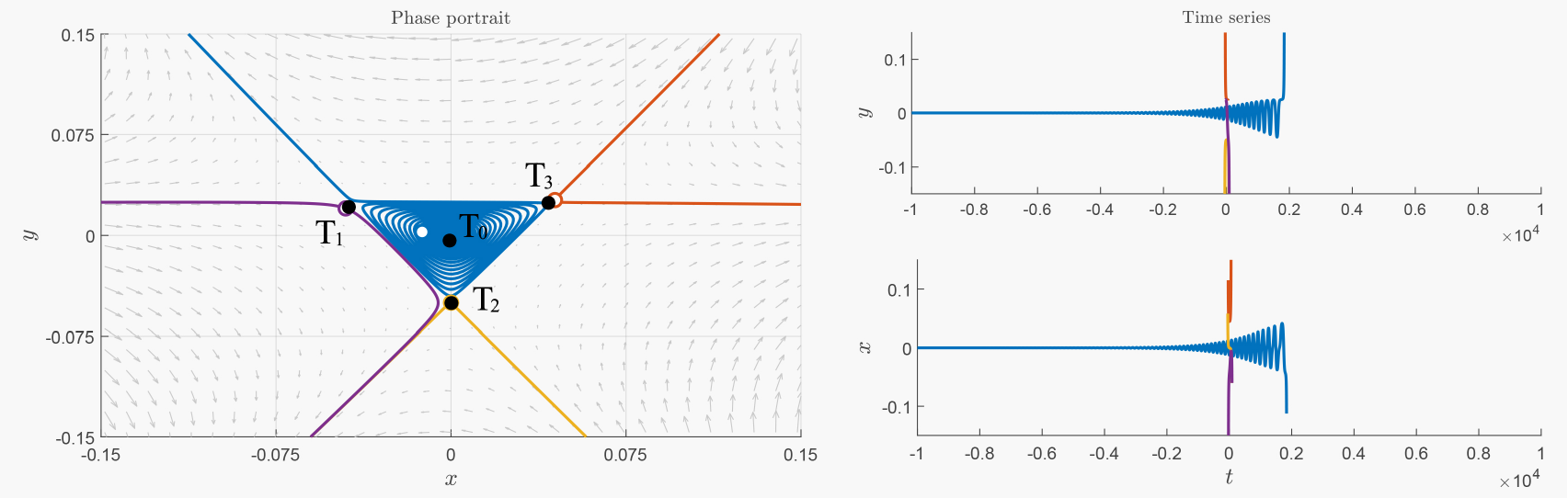}
\label{Y3-5}
}
\quad
\quad
\T\T\T\T\T\T\T\T\T\T\T\T\T\T\T\T\T\T\T\T\T\T\T\T\T\T\T\T\T\T\T\T\T\T\T\T
\subfigure[The phase diagrams of the region $H_6$ in Figure \ref{resonance1}.]{
\includegraphics[width=8cm]{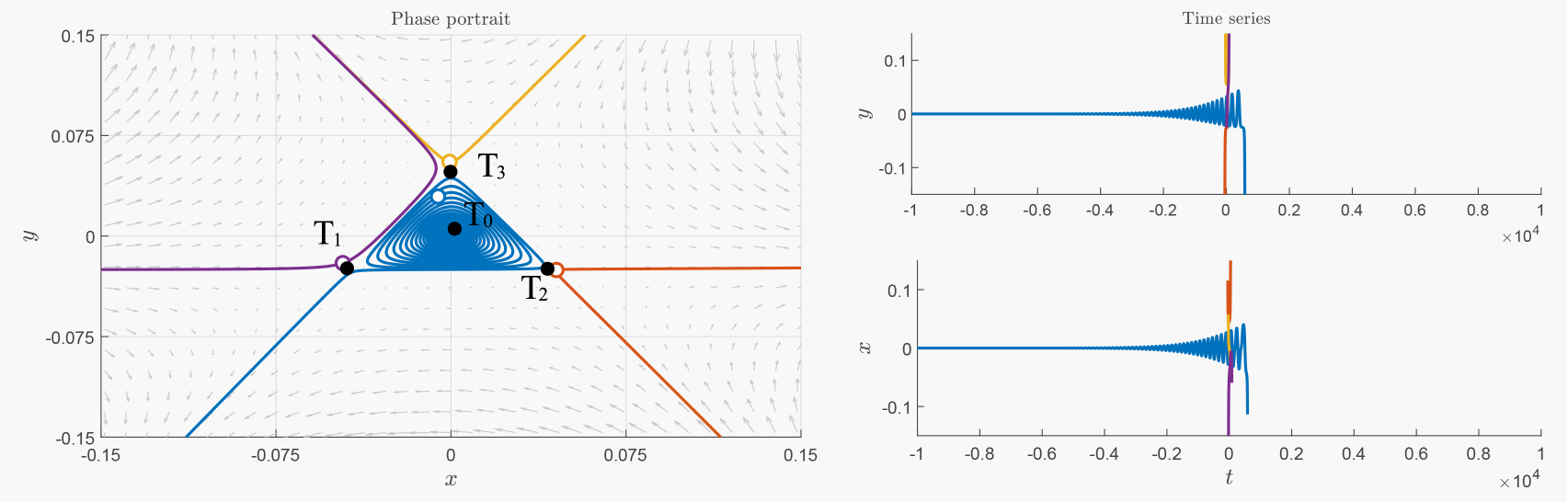}
\label{Y3-6}
}
\quad
\T\T\T\T\T\T\T\T\T\T\T\T\T\T\T\T\T\T\T\T\T\T\T\T\T\T\T\T\T\T\T\T\T\T\T\T
\subfigure[The phase diagrams of the region $H_7$ in Figure \ref{resonance1}.]{
\includegraphics[width=8cm]{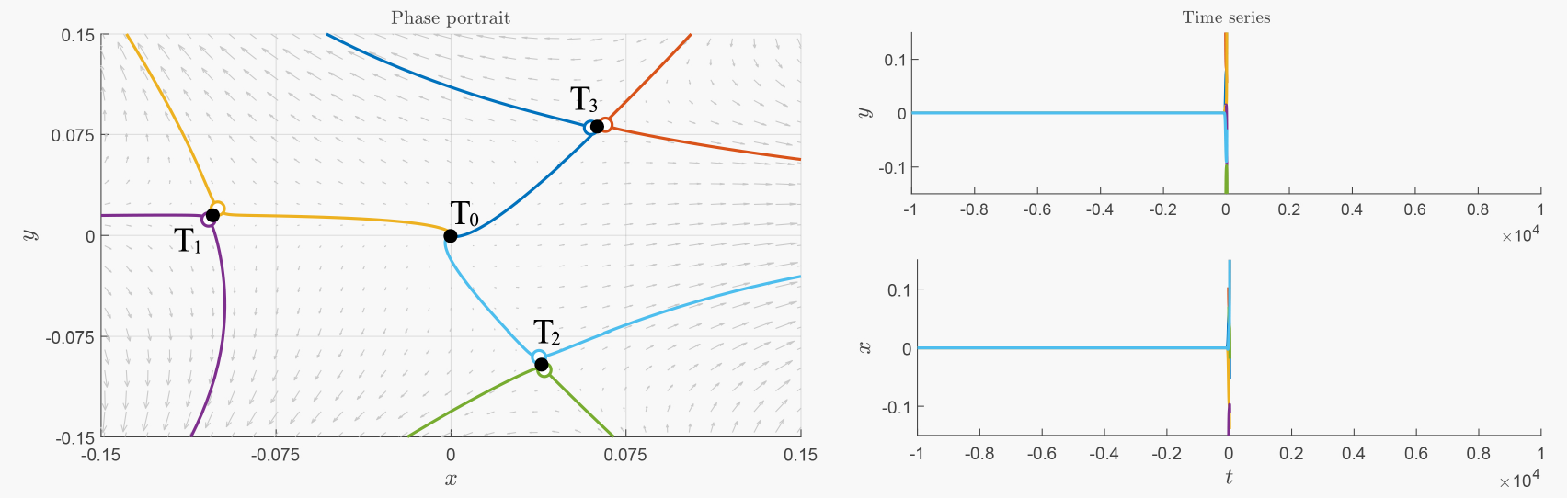}
\label{Y3-7}
}
\quad
\quad
\T\T\T\T\T\T\T\T\T\T\T\T\T\T\T\T\T\T\T\T\T\T\T\T\T\T\T\T\T\T\T\T\T\T\T\T
\subfigure[The phase diagrams of the region $H_8$ in Figure \ref{resonance1}.]{
\includegraphics[width=8cm]{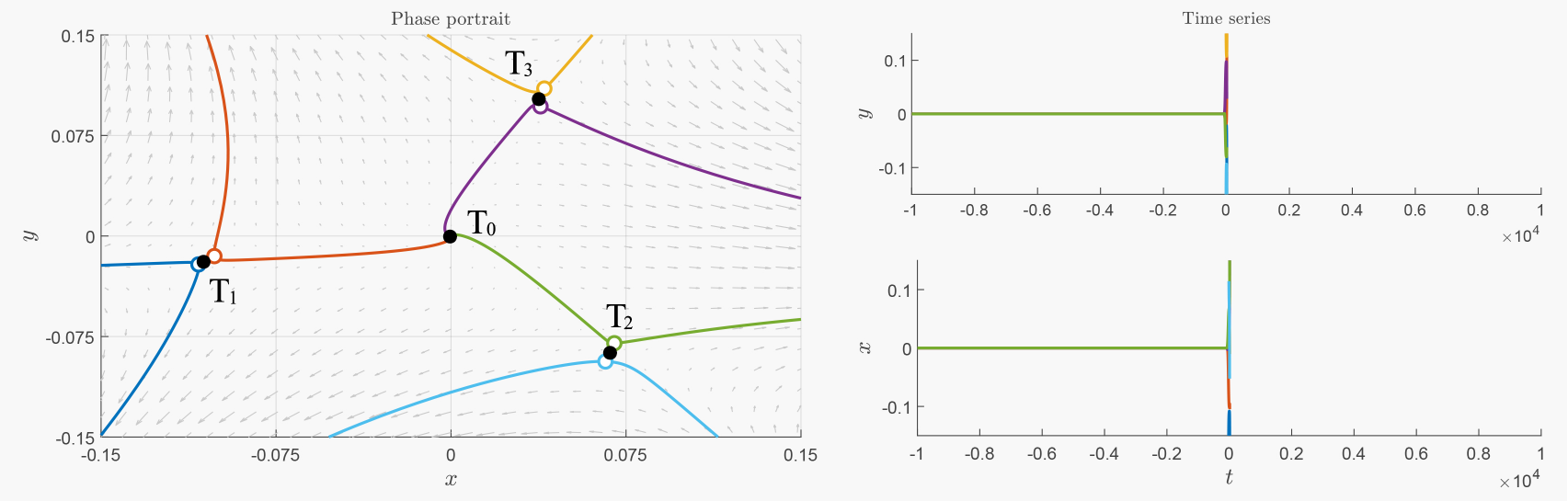}
\label{Y3-8}
}
\caption{phase portraits of system \eqref{eq3.08}.}
\label{1bi3YJQ1}
\end{figure}
\begin{description}
\item[(i)]When $(\sigma_1, \sigma_2)$ is near the origin, system \eqref{eq3.08} always has a trivial equilibrium $T_0$ with $\rho=0$, which is a focus in the case $(\sigma_1, \sigma_2)\neq0$ and a degenerate saddle in the $(\sigma_1, \sigma_2)=0$ (see \cite[Figure 3.1, p.279]{Chow}), and three nontrivial equilibria $T_k:(\rho_k,\iota_{s,k}),k = 1, 2, 3$, which are all saddles, located on a circle of radius $r_s$ with $r_s=\sqrt{\sigma_1^2+\sigma_2^2}+O(|\sigma|^{\frac{3}{2}})$, and separated by the angle $2\pi/3$ in $\iota$-coordinate.

\item[(ii)]For $\sigma_1<0$, the trivial equilibrium $T_0$ is a stable focus. As $\sigma$ crosses the $\sigma_2$-axis from left to right, $T_0$ becomes an unstable focus, and system \eqref{eq3.08} undergoes the Hopf bifurcation and produces an unique stable limit cycle $\Gamma_s$.

\item[(iii)]There is a bifurcation curve
    \begin{eqnarray*}
    H:=\{(\sigma_1,\sigma_2)|\sigma_1=-\frac{R_c(0)\,\sigma_2^2}{2}+o(\sigma_2^2)\}
    \end{eqnarray*}
    such that system \eqref{eq3.08} has a heteroclinic cycle formed by coinciding stable and unstable manifold of the nontrivial saddles for $\sigma\in H$, Where $H$ consist of $H_5$ and $H_6$. The heteroclinic cycle resembles a triangle and is stable from the inside.

\item[(iv)]As $\sigma$ crosses the bifurcation curve $H_1$ from left to right, the limit cycle disappears via a heteroclinic bifurcation.
\end{description}
Therefore, mapping $\Gamma_{\breve{\varepsilon}}$ has the following relationship with the differential system \eqref{eq3.08}: Mapping $\Gamma_{\breve{\varepsilon}}$ always has a trivial fixed point $O:(0,0)$ corresponding to the trivial equilibrium $T_0$ of \eqref{eq3.08}. Mapping $\Gamma_{\breve{\varepsilon}}$ undergoes a Neimark-Sacker bifurcation and generates an invariant circle surrounding the fixed point $T_0$ on a bifurcation curve corresponding to the Hopf bifurcation curve $\sigma_2$-axis of \eqref{eq3.08}. Mapping $\Gamma_{\breve{\varepsilon}}$ has a saddle cycle of period-3, which corresponds to the three nontrivial equilibria $T_k, k=1,2,3$ of \eqref{eq3.08}. The saddle cycle of period-3 lies in the invariant circle if $(\sigma_1, \sigma_2)$ lies in the curve corresponding to the bifurcation curve $H$ of \eqref{eq3.08}. Due to the topological equivalence of the mappings \eqref{eq2.1} and $\Gamma_{\breve{\varepsilon}}$, they exist the same bifurcation phenomena.

In addition, in order to obtain the bifurcation curve of mapping \eqref{eq2.1}, from $(\varepsilon_1,\varepsilon_2):=(r+(\beta^2-3\,\beta+3)/(\beta-3),\alpha+9/\beta\,(\beta-3))$ and \eqref{eq3.05}, we see that the $\sigma_2$-axis in Figure \ref{resonance1} corresponds to the line
\begin{eqnarray*}
&&(r,\alpha)\in \mathfrak{L}_2:=\{
N>0, 0<\beta<1, r=-\frac{\beta^2-3\,\beta+3}{\beta-3}, \alpha=\Psi_3
\}
\end{eqnarray*}
as $(r,\alpha)$ is close to the point $((\beta^2-3\,\beta+3)/(\beta-3), 9/\beta\,(\beta-3))$. Further, by the Implicit Function Theorem, one can check that the bifurcation curve $H$ of \eqref{eq3.08} corresponds to the curve $H_{3r}$ regarding the original parameter. The proof is completed.
\end{proof}

Besides, for
$$R_c(0):={\frac {-4\,{\beta}^{2}+3\,\beta}{6\,{\beta}^{2}-18\,\beta+18}}>0,$$
based on the same idea as in the proof of Corollary \ref{cor6.21} and in \cite{Chow,Kuznetsov,Kuznetsov1}, the similar conclusions will be obtained.

Based on the above analysis, we see that mapping \eqref{eq2.1} undergoes a $1:3$ resonance and generates complex dynamic properties near $1:3$ resonance point $((\beta^2-3\,\beta+3)/(\beta-3), 9/\beta\,(\beta-3))$ in Theorem \ref{th6.2} and Corollary \ref{cor6.21}. Hence, the complex dynamic behaviors of mapping \eqref{eq2.1} are not only dependent on bifurcation parameters, but also highly sensitive to parameter perturbations. From a biological perspective, as the parameter $(r,\alpha)$ changes near $((\beta^2-3\,\beta+3)/(\beta-3), 9/\beta\,(\beta-3))$, the numbers of susceptible, infective and recovered individuals can experience periodic or quasi periodic fluctuations (caused by non-degenerate Neimark-Sacker bifurcation), period-3 fluctuations (caused by the saddle cycle of period-3), long-period fluctuations, large-scale population outbreaks, and even chaos (caused by homoclinic structures).

\subsection{$1:4$ resonance at $E_2$}
\allowdisplaybreaks[4]
In this subsection, we discuss another type of codimension 2 bifurcation in mapping \eqref{eq2.1}, i.e., $1:4$ resonance, most of which is theoretically obtained by Arnold (see \cite{Arnold1,Arnold2}). For mapping \eqref{eq2.1}, when the following conditions are satisfied
{\small\begin{eqnarray*}
&&N>0, 0<\beta<1, 0<r<1, \Delta({P_{E_2}}(t))<0,\\
&&{P_{E_2}}({\bf{i}})={\frac {\beta\, \left( {\beta}^{2}+ \left( -\alpha+2\,r \right) \beta+
{r}^{2}-r\alpha+\alpha \right) }{\beta+r}}-{\frac {\left( \beta\,\alpha-2\,\beta-2\,r \right)\,{\bf i} }{\beta+r}}=0,\\
&&{P_{E_2}}(-{\bf{i}})={\frac {\beta\, \left( {\beta}^{2}+ \left( -\alpha+2\,r \right) \beta+
{r}^{2}-r\alpha+\alpha \right) }{\beta+r}}+{\frac {\left( \beta\,\alpha-2\,\beta-2\,r \right)\,{\bf i} }{\beta+r}}=0,\\
\end{eqnarray*}}$JF(E_2)$ has a pair of complex eigenvalues, namely $\pm{\bf{i}}$. Therefore, the $1:4$ resonance may occur near $E_2$. By solving the above semi-algebraic system, we get
$$r=-{\frac {{\beta}^{2}-2\,\beta+2}{\beta-2}}~~\mbox{and}~~\alpha=-{\frac {4}{\beta\, \left( \beta-2 \right) }}.$$

Under the aforementioned conditions, the subsequent analysis will focus on the existence of a 1:4 resonance in the vicinity of $E_2$.

\begin{thm}
Mapping \eqref{eq2.1} undergoes a $1:4$ resonance, if the parameter $(r,\alpha)$ varies near the point $(-({\beta}^{2}-2\,\beta+2)/(\beta-2),-4/\beta\,(\beta-2))$, and meets the following conditions
\begin{eqnarray*}
&&\frac{\beta\,(\beta-2)^2}{8}\neq0,-{\frac {3\,{\beta}^{3}{N}^{2} \left( \beta-2 \right) ^{4} \left(
\beta-\frac{2}{3} \right) }{40\,{\beta}^{4}-192\,{\beta}^{3}+384\,{\beta}^{2}-
384\,\beta+160}}\neq0,\\
&&{\frac {{\beta}^{2} \left( {\beta}^{2}-3\,\beta+1 \right) {N}^{2}
 \left( \beta-2 \right) ^{4}}{20\,{\beta}^{4}-96\,{\beta}^{3}+192\,{
\beta}^{2}-192\,\beta+80}}\neq0.
\end{eqnarray*}
\label{th6.3}
\end{thm}

\begin{proof}
We first translate $E_2$ to the origin O, and then use
\begin{eqnarray}
\breve{\epsilon}=(\epsilon_1, \epsilon_2)=(r+{\frac {{\beta}^{2}-2\,\beta+2}{\beta-2}}, \alpha+{\frac {4}{\beta\, \left( \beta-2 \right) }})
\label{eq7.1}
\end{eqnarray}
as a new parameter to replace $(\lambda,\mu)$, mapping \eqref{eq2.1} can be converted to
\begin{eqnarray}
\!\!\left(
\begin{array}{l}
u_{12} \\
\\
v_{12} \\
\\
w_{12}
\end{array}
\right)
\!\rightarrow\!
\left(
\begin{array}{c}
-{\frac {u_{{12}} \left( {\beta}^{2}\epsilon_{{2}}-\beta\,\epsilon_{{1
}}-2\,\beta\,\epsilon_{{2}}+2\,\epsilon_{{1}}-2 \right) }{\beta\,
\epsilon_{{1}}-2\,\epsilon_{{1}}-2}}-{\frac {v_{{12}} \left( \beta\,
\epsilon_{{1}}-2\,\epsilon_{{1}}-2 \right) }{\beta-2}}-{\frac {v_{{12}
}u_{{12}} \left( {\beta}^{2}\epsilon_{{2}}-2\,\beta\,\epsilon_{{2}}-4
 \right) }{N\beta\, \left( \beta-2 \right) }}\\
 \noalign{\medskip}
-{\frac {u_{{12}} \left( {\beta}^{2}\epsilon_{{1}}-{\beta}^{2}\epsilon
_{{2}}-2\,\beta\,\epsilon_{{1}}+2\,\beta\,\epsilon_{{2}}-2\,\beta+4
 \right) }{\beta\,\epsilon_{{1}}-2\,\epsilon_{{1}}-2}}+v_{{12}}+{
\frac {v_{{12}}u_{{12}} \left( {\beta}^{2}\epsilon_{{2}}-2\,\beta\,
\epsilon_{{2}}-4 \right) }{N\beta\, \left( \beta-2 \right) }}\\
\noalign{\medskip}
-{\frac { \left( {\beta}^{2}-\beta\,\epsilon_{{1}}-2\,\beta+2\,
\epsilon_{{1}}+2 \right) v_{{12}}}{\beta-2}}-w_{{12}} \left( -1+\beta
 \right)
\end{array}
\right).~~~
\label{eq7.2}
\end{eqnarray}
Continuing to introduce a reversible transformation
\begin{eqnarray*}
\left(
\begin{array}{cc}
u_{12} \\
v_{12} \\
w_{12}
\end{array}
\right)
\rightarrow
\left(
\begin{array}{ccc}
1 & 1 &0
\\
0 & 1 &0
\\
-1&-1&1
\end{array}
\right)
\left(
\begin{array}{l}
u_{13}
\\
v_{13}
\\
w_{13}
\end{array}
\right),
\end{eqnarray*}
mapping \eqref{eq7.2} can be changed into
\begin{eqnarray}\!\!\!\!\!
\begin{aligned}
\left(
\begin{array}{cc}
u_{13} \\
\\
v_{13} \\
\\
w_{13}
\end{array}
\right)\!\!\!
\rightarrow\!\!\!
\left(
\begin{array}{c}
-{\frac { \left( {\beta}^{2}\epsilon_{{2}}-\beta\,\epsilon_{{1
}}-2\,\beta\,\epsilon_{{2}}+2\,\epsilon_{{1}}-2 \right)\,u_{{13}} }{\beta\,
\epsilon_{{1}}-2\,\epsilon_{{1}}-2}}-{\frac { \left( \beta\,
\epsilon_{{1}}-2\,\epsilon_{{1}}-2 \right)\,v_{{13}} }{\beta-2}}-{\frac { \left( {\beta}^{2}\epsilon_{{2}}-2\,\beta\,\epsilon_{{2}}-4
 \right)\,u_{{13}}\,v_{{13}} }{\beta\,N \left( \beta-2 \right) }}
\\
\noalign{\medskip}
-{\frac { \left( {\beta}^{2}\epsilon_{{1}}-{\beta}^{2}\epsilon
_{{2}}-2\,\beta\,\epsilon_{{1}}+2\,\beta\,\epsilon_{{2}}-2\,\beta+4
 \right)\,u_{{13}} }{\beta\,\epsilon_{{1}}-2\,\epsilon_{{1}}-2}}+v_{{13}}+{
\frac {\left( {\beta}^{2}\epsilon_{{2}}-2\,\beta\,
\epsilon_{{2}}-4 \right)\,u_{{13}}\,v_{{13}}  }{\beta\,N \left( \beta-2 \right) }}
\\
\noalign{\medskip}
\left( 1-\beta \right)\,w_{{13}}
\end{array}
\right).
\end{aligned}
\label{eq7.3}
\end{eqnarray}

In addition, mapping \eqref{eq7.3} exists a two dimensional $C^{2}$ center manifold
\begin{eqnarray*}
w_{10}=h_{6}(u_{10},v_{10})=O(|(u_{10},v_{10})|^3),
\end{eqnarray*}
and then mapping \eqref{eq7.2} is restricted into the following two dimensional form
\begin{eqnarray}
\left[
\begin{array}{ccc}
   u_{13}  \\
   \noalign{\medskip}
   v_{13}
\end{array}
\right]\!\!\!
\mapsto\!\!\!
\left[\begin{array}{lll}
-{\frac { \left( {\beta}^{2}\epsilon_{{2}}-\beta\,\epsilon_{{1
}}-2\,\beta\,\epsilon_{{2}}+2\,\epsilon_{{1}}-2 \right)\,u_{{13}} }{\beta\,
\epsilon_{{1}}-2\,\epsilon_{{1}}-2}}-{\frac { \left( \beta\,
\epsilon_{{1}}-2\,\epsilon_{{1}}-2 \right)\,v_{{13}} }{\beta-2}}-{\frac { \left( {\beta}^{2}\epsilon_{{2}}-2\,\beta\,\epsilon_{{2}}-4
 \right)\,u_{{13}}\,v_{{13}} }{\beta\,N \left( \beta-2 \right) }}
\\
\noalign{\medskip}
-{\frac { \left( {\beta}^{2}\epsilon_{{1}}-{\beta}^{2}\epsilon
_{{2}}-2\,\beta\,\epsilon_{{1}}+2\,\beta\,\epsilon_{{2}}-2\,\beta+4
 \right)\,u_{{13}} }{\beta\,\epsilon_{{1}}-2\,\epsilon_{{1}}-2}}+v_{{13}}+{
\frac {\left( {\beta}^{2}\epsilon_{{2}}-2\,\beta\,
\epsilon_{{2}}-4 \right)\,u_{{13}}\,v_{{13}}  }{\beta\,N \left( \beta-2 \right) }}
 \end{array}\right].
\label{eq7.4}
\end{eqnarray}

Moreover, employing the transformation
\begin{eqnarray*}
\left(
\begin{array}{cc}
u_{13} \\
v_{13}
\end{array}
\right)
\rightarrow
\left(
\begin{array}{cc}
{\frac {{\beta}^{2}\,\epsilon_2-2\,\beta\,\epsilon_2-4}{2\,{\beta}^{2}\epsilon_1-2\,{\beta}^{
2}\,\epsilon_2-4\,\beta\,\epsilon_1+4\,\beta\,\epsilon_2-4\,\beta+8}}
 & -{\frac {\sqrt {\chi_5}}{ \left( 2\,\beta\,\epsilon_1-2\,\beta\,\epsilon_2-4 \right)
 \left( \beta-2 \right) }}
\\
1&0
\end{array}
\right)
\left(
\begin{array}{l}
u_{14}
\\
v_{14}
\end{array}
\right),
\end{eqnarray*}
where
{\footnotesize\begin{eqnarray*}
&&\!\!\!\!\!\!\!\!\chi_{5}:=-{\beta}^{4}\epsilon_2^{2}-4\,{\beta}^{3}\epsilon_1^{3}+4\,{\beta}^{3}\epsilon_1^{2}
\epsilon_2+4\,{\beta}^{3}\epsilon_2^{2}+16\,{\beta}^{2}\epsilon_1^{3}-16\,{\beta}^{2}\epsilon_1^{2}\epsilon_2+24\,{\beta}^{2}\epsilon_1^{2}
-16\,{\beta}^{2}\,\epsilon_1\,\epsilon_2-4\,{\beta
}^{2}\epsilon_2^{2}\\
&&~~~-16\,\beta\,\epsilon_1^{3}+16\,\beta\,\epsilon_1^{2}\,\epsilon_2+8\,{\beta}
^{2}\,\epsilon_2-64\,\beta\,\epsilon_1^{2}+32\,\beta\,\epsilon_1\,\epsilon_2-48\,\beta\,\epsilon_1+32\,\epsilon_1^{2}+64\,\epsilon_1+16,
\end{eqnarray*}}
we change \eqref{eq7.4} into the mapping
\begin{eqnarray}
\begin{aligned}
\left(
\begin{array}{cc}
u_{14} \\
v_{14}
\end{array}
\right)
\rightarrow
\left(
\begin{array}{l}
\tilde{k}_{11}\,u_{11}+\tilde{k}_{12}\,v_{11}+\tilde{p}_{20}\,u^2_{11}+\tilde{p}_{11}\,u_{11}\,v_{11}\\
\tilde{k}_{21}\,u_{11}+\tilde{k}_{22}\,v_{11}+\tilde{q}_{20}\,u^2_{11}+\tilde{q}_{11}\,u_{11}\,v_{11}\\
\end{array}
\right),
\end{aligned}
\label{eq7.4-1}
\end{eqnarray}
where
{\small\begin{eqnarray*}
&&\!\!\!\!\!\!\!\!\!\!\tilde{k}_{11}=\tilde{k}_{22}=:-{\frac {\nu\,{\beta}^{2}-2\,\beta\,\mu-2\,\beta\,\nu+4\,\mu}{2\,\beta
\,\mu-4\,\mu-4}},~\tilde{k}_{12}=-\tilde{k}_{21}:={\frac {\sqrt {\chi_5}}{2\,\beta\,\mu-4\,\mu-4}}\\
&&\!\!\!\!\!\!\!\!\!\!\tilde{p}_{20}:={\frac {{\beta}^{4}{\nu}^{2}-4\,{\beta}^{3}{\nu}^{2}+4\,{\beta}^{
2}{\nu}^{2}-8\,\nu\,{\beta}^{2}+16\,\beta\,\nu+16}{2\,\beta\, \left(
\beta\,\mu-\beta\,\nu-2 \right)  \left( \beta-2 \right) ^{2}N}},~\tilde{p}_{11}:=-{\frac {\sqrt {\chi_5} \left( \nu\,{\beta}^{2}-2\,\beta\,\nu-4
 \right) }{2\,\beta\, \left( \beta\,\mu-\beta\,\nu-2 \right)  \left(
\beta-2 \right) ^{2}N}}\\
&&\!\!\!\!\!\!\!\!\!\!\tilde{q}_{20}:={\frac {\chi_6}{2\,\beta\, \left( \beta\,\mu-\beta\,
\nu-2 \right)  \left( \beta-2 \right) ^{2}N\sqrt {\chi_5}}},\\
&&\!\!\!\!\!\!\!\!\!\!\tilde{q}_{11}:=-{\frac {2\,{\beta}^{4}\mu\,\nu-{\beta}^{4}{\nu}^{2}-8\,\mu\,\nu
\,{\beta}^{3}+4\,{\beta}^{3}{\nu}^{2}-4\,{\beta}^{3}\nu+8\,{\beta}^{2}
\mu\,\nu-4\,{\beta}^{2}{\nu}^{2}-8\,{\beta}^{2}\mu}{2\,\beta\, \left( \beta\,
\mu-\beta\,\nu-2 \right)  \left( \beta-2 \right) ^{2}N}}\\
&&~~~-{\frac {16\,\nu\,{\beta}^{2
}+16\,\beta\,\mu-16\,\beta\,\nu+16\,\beta-16}{2\,\beta\, \left( \beta\,
\mu-\beta\,\nu-2 \right)  \left( \beta-2 \right) ^{2}N}},
\end{eqnarray*}}and
{\small\begin{eqnarray*}
&&\!\!\!\!\!\!\!\!\!\!\chi_6:= 2\,{\beta}^{6}\epsilon_1\,\epsilon_2^{2}-{\beta}^{6}\epsilon_2^{3}-12\,\epsilon_1\,\epsilon_2^{2}{\beta}^{5}+6\,{\beta}^{5}\epsilon_2^{3}
-4\,\epsilon_2^{2}{\beta}^{5}+24\,{\beta}^{4}\epsilon_2^{2}\epsilon_1-12\,{\beta}^{4}\epsilon_2^{3}-
16\,{\beta}^{4}\epsilon_1\,\epsilon_2\\
&&~~+28\,{\beta}^{4}\epsilon_2^{2}-16\,{\beta}^{3}\epsilon_1\,\epsilon_2^{2}+8\,{\beta}^{3}\epsilon_2^{3}+64\,\epsilon_1\,\epsilon_2\,{\beta}^{3}-64\,{\beta}^{3}\epsilon_2^{2}+32\,{\beta}^{3}\epsilon_2
-64\,{\beta}^{2}\epsilon_1\,\epsilon_2+48\,{\beta}^{
2}\epsilon_2^{2}\\
&&~~+32\,{\beta}^{2}\epsilon_1-112\,\epsilon_2\,{\beta}^{2}-64\,\beta\,\epsilon_1+96
\,\beta\,\epsilon_2-64\,\beta+64.
\end{eqnarray*}}
\!\!\eqref{eq7.4-1} can be further rewritten by $z_2=u_{13}+v_{13}\,{\bf{i}}$ in the complex form
\begin{eqnarray}
\begin{array}{l}
z_2\mapsto\zeta_1(\breve{\epsilon})z_2+l_{20}z_2^2+l_{11}z_2\bar{z}_2+l_{02}\bar{z}^2_2,
\end{array}
\label{eq7.5}
\end{eqnarray}
where
\begin{eqnarray*}
&&\!\!\!\!\!\!\!\!\!\!l_{20}:=-{\frac { \left( \left( \frac{\beta}{2}-1 \right) \sqrt {\chi_5}-\frac{{\bf i}\,{\beta}^{3}\epsilon_2}{2}
+ \left( {\epsilon_1}^{2}\,{\bf i}+2\,{\bf i}\,\epsilon_2 \right) {\beta}^{2}
 \right)  \left( \epsilon_2\,{\beta}^{2}-2\,\beta\,\epsilon_2-4 \right) }{\sqrt {\chi_5}
\beta\, \left( \beta-2 \right) ^{2}N}}\\
&&~~~-{\frac { \left(  \left( -4\,{\bf i}\,{\epsilon_1}^{2}
-4\,{\bf i}\,\epsilon_1-2\,{\bf i}\,\epsilon_2+2\,{\bf i}\, \right) \beta+4\,{\bf i}\, \left( \epsilon_1+2 \right) \epsilon_1
 \right)  \left( \epsilon_2\,{\beta}^{2}-2\,\beta\,\epsilon_2-4 \right) }{\sqrt {\chi_5}
\beta\, \left( \beta-2 \right) ^{2}N}},\\
&&\!\!\!\!\!\!\!\!\!\!l_{11}:={\frac { \left( \epsilon_2\,{\beta}^{2}-2\,\beta\,\epsilon_2-4 \right) ^{2}
 \left( \frac{\sqrt {\chi_5}}{2}+2\,{\bf i}+ \left( \epsilon_1\,{\bf i}-\frac{{\bf i}\,\epsilon_2}{2} \right) {\beta}^{2}+
 \left( -2\,{\bf i}\,\epsilon_1+\epsilon_2\,{\bf i}-2\,{\bf i} \right) \beta \right) }{2\,\sqrt {\chi_5} \left(
\beta-2 \right) ^{2}N\beta\, \left( -2+ \left( \epsilon_1-\epsilon_2 \right) \beta
 \right) }},\\
&&\!\!\!\!\!\!\!\!\!\!l_{02}:={\frac { \left( \beta\,\epsilon_1-2\,\epsilon_1-2 \right)  \left( \epsilon_2\,{\beta}^{2}-2
\,\beta\,\epsilon_2-4 \right)  \left( \frac{\sqrt {\chi_5}\beta}{2}+\frac{{\bf i}\,{\beta}^{3}\epsilon_2}{2}+
 \left( {\epsilon_1}^{2}\,{\bf i}-{\bf i}\,\epsilon_1\,\epsilon_2-\epsilon_2\,{\bf i} \right) {\beta}^{2}\right) }{
\sqrt {\chi_5}\beta\, \left( \beta-2 \right) ^{2}N \left( -2+ \left( \epsilon_1-
\epsilon_2 \right) \beta \right) }}\\
&&~~~+{\frac { \left( \beta\,\epsilon_1-2\,\epsilon_1-2 \right)  \left( \epsilon_2\,{\beta}^{2}-2
\,\beta\,\epsilon_2-4 \right)  \left( -2\,{\bf i} \left(
\epsilon_1+1 \right)  \left( \epsilon_1-\epsilon_2+1 \right) \beta+4\,{\bf i}\,\epsilon_1+4\,{\bf i} \right) }{
\sqrt {\chi_5}\beta\, \left( \beta-2 \right) ^{2}N \left( -2+ \left( \epsilon_1-
\epsilon_2 \right) \beta \right) }},\\
\end{eqnarray*}
By Lemma $9.14$ in \cite[p.455]{Kuznetsov}, mapping \eqref{eq7.5} can be transformed into the form
\begin{eqnarray}
\eta_1\mapsto\Gamma_{\breve{\epsilon}}(\eta_1):=\zeta_1(\breve{\epsilon})\eta_1+C(\breve{\epsilon})\eta_1\bar{\eta}^2_1+D(\breve{\epsilon})\bar{\eta}^3_1+O(|\eta_1|^4),
\label{eq7.6}
\end{eqnarray}
where
\begin{eqnarray*}
&&C(\breve{\epsilon}):=\frac{(1+3\,{\bf{i}})\,l_{20}\,l_{11}}{4}+\frac{(1-{\bf{i}})\,l_{11}\,\bar{l}_{11}}{2}-\frac{(1+{\bf{i}})\,l_{02}\,\bar{l}_{02}}{4}+\frac{l_{21}}{2},\\
&&D(\breve{\epsilon}):=\frac{({\bf{i}}-1)\,l_{11}\,l_{02}}{4}-\frac{(1+{\bf{i}})\,l_{02}\,\bar{l}_{20}}{4}+\frac{l_{03}}{6},
\end{eqnarray*}
by an invertible smooth change of variable, smoothly depending on the parameters, for sufficiently small $\breve{\epsilon}$.
Since $\zeta_1(0)=e^{(\frac{\pi\,{\bf{i}}}{2})}={\bf{i}}$, we introduce a new parameters $\breve{\omega}=(\omega_2,\omega_3)$, which satisfies
\begin{eqnarray}
\zeta_1(0)^3\,\zeta_1(\breve{\epsilon})=e^{(\omega_2+{\bf{i}}\,\omega_3)}.
\label{eq7.7}
\end{eqnarray}
As a result, from \eqref{eq7.7} we get a mapping
\begin{eqnarray}
\breve{\epsilon}\mapsto\breve{\omega}(\breve{\epsilon})=(\omega_2(\breve{\epsilon}),\omega_3(\breve{\epsilon})).
\label{eq7.8}
\end{eqnarray}
Due to
\begin{eqnarray*}
det
\left.
\left(
\begin{array}{ll}
\frac{\partial\omega_2}{\partial\epsilon_1}&\frac{\partial\omega_2}{\partial\epsilon_2}\\
\frac{\partial\omega_3}{\partial\epsilon_1}&\frac{\partial\omega_3}{\partial\epsilon_2}
\end{array}
\right)
\right|_{(\epsilon_1,\epsilon_2)=(0,0)}
=det\left(
\begin{array}{cc}
1-\beta    &  \frac{\beta^2}{4}\\
\frac{\beta}{2}-1 &  -\frac{(\beta-2)\,\beta}{4}
\end{array}
\right)=\frac{\beta\,(\beta-2)^2}{8}\neq0,
\end{eqnarray*}
mapping \eqref{eq7.8} is regular at $0$. Thereupon, the new parameter $\breve{\omega}$ can be used as a new unfolding parameter and we can introduce the truncated normal form
\begin{eqnarray}
N_{\omega}: \eta_1\mapsto e^{(\frac{\pi\,{\bf{i}}}{2}+\omega_2+{\bf{i}}\,\omega_3)}\,\eta_1+c(\breve{\omega})\eta_1\bar{\eta}^2_1+d(\breve{\omega})\bar{\eta}^3_1,
\label{eq7.9}
\end{eqnarray}
where $c$ and $d$ are smooth functions of $\breve{\omega}$, such that $c(0)=C(0)$ and $d(0)=D(0)$.
In addition, define a linear transformation $R: \mathbb{C}\rightarrow\mathbb{C}$ by
\begin{eqnarray}
\eta_1\mapsto R\,\eta_1:=\bar{\zeta}_1(0)\,\eta_1=e^{(\frac{\pi\,{\bf{i}}}{2})}\,\eta_1=-{\bf{i}}\,\eta_1
\label{eq7.10}
\end{eqnarray}
which is the 'clockwise' rotation through $\pi/2$. The truncated normal form \eqref{eq7.9} is invariant with respect to $R$. Note that $R^4=id$. The phase portraits that follow will therefore possess $\mathbb{Z}_4$-symmetry.

From Theorem $3.28$ in \cite[p.80]{Kuznetsov1}, for sufficiently small $\breve{\omega}$, mapping \eqref{eq7.9} satisfies
$$R N_{\breve{\omega}}(\eta_1)=\psi_{\breve{\omega}}^1(\eta_1)+O(|\eta_1|^4)$$
where $\psi_\omega^t$ is the flow of the approximating planar differential system
\begin{eqnarray}
\dot{\eta}_1=(\omega_2+{\bf{i}}\,\omega_3)\,\eta_1+c_1(\breve{\omega})\eta_1\bar{\eta}^2_1+d_1(\breve{\omega})\bar{\eta}^3_1, \eta_1\in \mathbb{C}
\label{eq7.11}
\end{eqnarray}
where $c_1$ and $d_1$ are smooth complex-valued functions of $\breve{\omega}$, such that
$$c_1(0)=\bar{\zeta}_1(0)\,c(0)=-{\bf{i}}\,c(0),~~d_1(0)=\bar{\zeta}_1(0)\,d(0)=-{\bf{i}}\,d(0).$$

Using the scaling
$$\eta_1=\frac{e^{(\frac{{\bf{i}}\,arg(d_1(\breve{\omega}))}{4})}\,\xi_1}{\sqrt{|d_1(\breve{\omega})|}}$$
and employing a transformation
$$A(\breve{\omega}):=\frac{c_1(\breve{\omega})}{|d_1(\breve{\omega})|},$$
\eqref{eq7.11} can be converted into the complex ODE
\begin{eqnarray}
\dot{\xi_1}=(\omega_2+{\bf{i}}\,\omega_3)\,\xi_1+A(\breve{\omega})\xi_1^2\bar{\xi}_1+\bar{\xi}_1^3, ~\xi_1 \in \mathbb{C},
\label{eq7.12}
\end{eqnarray}
Writing \eqref{eq7.12} in polar coordinates with $\xi_1=\rho_1 exp({\bf{i}}\nu_1)$, we obtain
\begin{eqnarray}
\begin{cases}
\begin{array}{l}
\frac{d\rho_1}{dt}=\omega_2\rho+Re(A(\breve{\omega}))\rho_1^3+\rho_1^3cos(4\,\nu_1)\\
\frac{d\nu_1}{dt}=\omega_3+Im(A(\breve{\omega}))\rho_1^2-\rho_1^2\,sin(4\,\nu_1)
\end{array}
\end{cases}.
\label{eq7.14}
\end{eqnarray}
It can be checked that
\begin{eqnarray*}
&&a_0=Re(A(0)):=-{\frac {3\,{\beta}^{3}{N}^{2} \left( \beta-2 \right) ^{4} \left(
\beta-\frac{2}{3} \right) }{2\,\sqrt{20\,{\beta}^{4}-96\,{\beta}^{3}+192\,{
\beta}^{2}-192\,\beta+80}}}\neq0,\\
&&b_0=Im(A(0)):={\frac {{\beta}^{2} \left( {\beta}^{2}-3\,\beta+1 \right) {N}^{2}
 \left( \beta-2 \right) ^{4}}{\sqrt{20\,{\beta}^{4}-96\,{\beta}^{3}+192\,{
\beta}^{2}-192\,\beta+80}}}\neq0.
\end{eqnarray*}
Consequently, the system experiences a $1:4$ resonance. This completes the proof.
\end{proof}

Furthermore, from \cite[pp.80-81]{Kuznetsov1}, we know that the bifurcation diagram of the approximating system \eqref{eq7.14} in the $(\omega_2,\omega_3)$-plane depends on $A_0=a_0+{\bf{i}}\,b_0$, and the division of this quadrant of the $A_0$-plane into regions corresponding to different bifurcation diagrams are shown in Figure \ref{partion4-1}. To be more precise, we have the following conclusions.
\begin{cor}
Under the conditions of Theorem \ref{th6.3}, if $(a_0,b_0)$ belongs to region II in Figure \ref{partion4-1},
then mapping \eqref{eq2.1} possesses the following dynamic properties when the parameter $(r, \alpha)$ in a sufficiently small neighborhood of point $(-(\beta^2-2\,\beta+2)/(\beta-2),-4/\beta\,(\beta-2))$:
\begin{description}
    \item[$(1)$]System $\eqref{eq2.1}$ always has a fixed point $E_2$. As $(r,\alpha)$ crosses the line
    $$
    (r,\alpha)\in \mathfrak{L}_2:=\{N>0, 0<\beta<1, \Psi_1<r<1, \alpha=\Psi_3\},
    $$
    \eqref{eq2.1} undergoes a non-degenerate supercritical Neimark-Sacker bifurcation and generates a stable invariant circle $\Gamma$.
    \item[$(2)$]  When $(r,\alpha)$ near the curves
    {\footnotesize\begin{eqnarray*}
    &&\!\!\!\!\!\!\!H_{41}:=\{(r,\alpha)|-\frac{\chi_7}{\chi_8}\,(r+\frac {{\beta}^{2}-2\,\beta+2}{\beta-2})-\frac{\chi_9}{\chi_{10}}\,(\alpha+{\frac {4}{\beta\, \left( \beta-2 \right) }})\\
    &&~~~~~~~+O(|(r+{\frac {{\beta}^{2}-2\,\beta+2}{\beta-2}},\alpha+{\frac {4}{\beta\, \left( \beta-2 \right) }})|^2)=0\}
    \end{eqnarray*}}
    and $H_{42}$ (replacing $\sqrt{H}$ with $-\sqrt{H}$ in the expression of $H_{41}$, one can obtain the expression of $H_{42}$), where
    {\footnotesize\begin{eqnarray*}
    &&\!\!\!\!\!\!\!\!\!\!\chi_7:=-51200+51200\,\sqrt {H}+491520\,\sqrt {H}{\beta}^{6}-971776\,\sqrt {H}{\beta}^{5}+1328128\,\sqrt {H}{\beta}^{4}\\
    &&~~~~-1253376\,\sqrt {H}{\beta}^{3}+786432\,\sqrt {H}{\beta}^{2}-296960\,\sqrt {H}\beta-3200\,\sqrt {H}{\beta}^{9}-165888\,\sqrt {H}{\beta}^{7}\\
    &&~~~~+33920\,\sqrt{H}{\beta}^{8}-17408\,{\beta}^{6}{N}^{4}-74\,{N}^{4}{\beta}^{16}+804\,{N}^{4}{\beta}^{15}-5124\,{N}^{4}{\beta}^{14}
    +21416\,{N}^{4}{\beta}^{13}\\
    &&~~~~-61952\,{N}^{4}{\beta}^{12}+127232\,{N}^{4}{\beta}^{11}-186752\,{N}^{4}{\beta}^{10}+194048\,{N}^{4}{\beta}^{9}-138752\,{N}^{4}{\beta}^{8}\\
    &&~~~~+64512\,{N}^{4}{\beta}^{7}+2048\,{\beta}^{5}{N}^{4}+3\,{N}^{4}{\beta}^{17}+1600\,{\beta}^{9}+664064\,{\beta}^{5}+98304\,{\beta}^{7}
    -18560\,{\beta}^{8}\\
    &&~~~~-971776\,{\beta}^{4}+983040\,{\beta}^{3}-313344\,{\beta}^{6}-663552\,{\beta}^{2}+271360\,\beta,\\
    &&\!\!\!\!\!\!\!\!\!\!\chi_8:=2\, ( 3\,{N}^{2}{\beta}^{8}-26\,{N}^{2}{\beta}^{7}+88\,{N}^{2}{\beta}^{6}-144\,{N}^{2}{\beta}^{5}+112\,{\beta}^{4}{N}^{2}
    -32\,{\beta}^{3}{N}^{2}-40\,{\beta}^{4}+192\,{\beta}^{3}\\
    &&~~~-384\,{\beta}^{2}+384\,\beta-160)(3\,{N}^{2}{\beta}^{8}-26\,{N}^{2}{\beta}^{7}
    +88\,{N}^{2}{\beta}^{6}-144\,{N}^{2}{\beta}^{5}+112\,{\beta}^{4}{N}^{2}\\
    &&~~~-32\,{\beta}^{3}{N}^{2}+40\,{\beta}^{4}-192\,{\beta}^{3}+384\,{\beta}^{2}-384\,\beta+160),\\
    &&\!\!\!\!\!\!\!\!\!\!\chi_9:=\beta\, ( 51200-178176\,\sqrt {H}{\beta}^{6}+307712\,\sqrt {H}{\beta}^{5}-356352\,\sqrt {H}{\beta}^{4}+270336\,\sqrt {H}{\beta}^{3}\\
    &&~~~-122880\,\sqrt {H}{\beta}^{2}+25600\,\sqrt {H}\beta+1600\,\sqrt {H}{\beta}^{9}+67584\,\sqrt {H}{\beta}^{7}-15360\,\sqrt {H}{\beta}^{8}\\
    &&~~~-1024\,{\beta}^{6}{N}^{4}-56\,{N}^{4}{\beta}^{16}+474\,{N}^{4}{\beta}^{15}-2404\,{N}^{4}{\beta}^{14}+8128\,{N}^{4}{\beta}^{13}
    -19264\,{N}^{4}{\beta}^{12}\\
    &&~~~+32704\,{N}^{4}{\beta}^{11}-39808\,{N}^{4}{\beta}^{10}+34048\,{N}^{4}{\beta}^{9}-19456\,{N}^{4}{\beta}^{8}+6656\,{N}^{4}{
    \beta}^{7}+3\,{N}^{4}{\beta}^{17}\\
    &&~~~-1600\,{\beta}^{9}-664064\,{\beta}^{5}-98304\,{\beta}^{7}+18560\,{\beta}^{8}+971776\,{\beta}^{4}-983040\,{
    \beta}^{3}+313344\,{\beta}^{6}\\
    &&~~~+663552\,{\beta}^{2}-271360\,\beta),\\
    &&\!\!\!\!\!\!\!\!\!\!\chi_{10}:=4\,( 3\,{N}^{2}{\beta}^{8}-26\,{N}^{2}{\beta}^{7}+88\,{N}^{2}{
    \beta}^{6}-144\,{N}^{2}{\beta}^{5}+112\,{\beta}^{4}{N}^{2}-32\,{\beta}^{3}{N}^{2}-40\,{\beta}^{4}+192\,{\beta}^{3}\\
    &&~~~~-384\,{\beta}^{2}+384\,\beta-160)(3\,{N}^{2}{\beta}^{8}-26\,{N}^{2}{\beta}^{7}+88\,{N}^{2}{\beta}^{6}-144\,{N}^{2}{\beta}^{5}
    +112\,{\beta}^{4}{N}^{2}\\
    &&~~~~-32\,{\beta}^{3}{N}^{2}+40\,{\beta}^{4}-192\,{\beta}^{3}+384\,{\beta}^{2}-384\,\beta+160),\\
    &&\!\!\!\!\!\!\!\!\!\!H:=\frac{\chi_{11}}{1600\,( {\beta}^{2}-{\frac {14\,\beta}{5}}+2) ^{2}( {\beta}^{2}-2\,\beta+2 )},\\
    &&\!\!\!\!\!\!\!\!\!\!\chi_{11}:=-12800+13\,{N}^{4}{\beta}^{14}-218\,{N}^{4}{\beta}^{13}+1618\,{N}^{4}{\beta}^{12}-6976\,{N}^{4}{\beta}^{11}
    +19264\,{N}^{4}{\beta}^{10}\\
    &&~~~~-35392\,{N}^{4}{\beta}^{9}+43456\,{N}^{4}{\beta}^{8}-34816\,{N}^{4}{\beta}^{7}+( 17152\,{N}^{4}-1600) {\beta}^{6}+( -4608\,{N}^{4}\\
    &&~~~~+12160 ) {\beta}^{5}+( 512\,{N}^{4}-40064){\beta}^{4}+73728\,{\beta}^{3}-80128\,{\beta}^{2}+48640\,\beta,
    \end{eqnarray*}}there exists a saddle-node of period-4 located on the invariant circle $\Gamma$, which is a homoclinic structure.
 \item[$(3)$] As $(r,\alpha)$ crosses the curves $H_{41}$ and $H_{42}$, the mapping \eqref{eq2.1} undergoes a fold bifurcation and generates a period-4 saddle and a period-4 node (or focus).
\end{description}
\label{cor6.31}
\end{cor}

\begin{figure}[ht!]
\centering
{
\includegraphics[width=10cm]{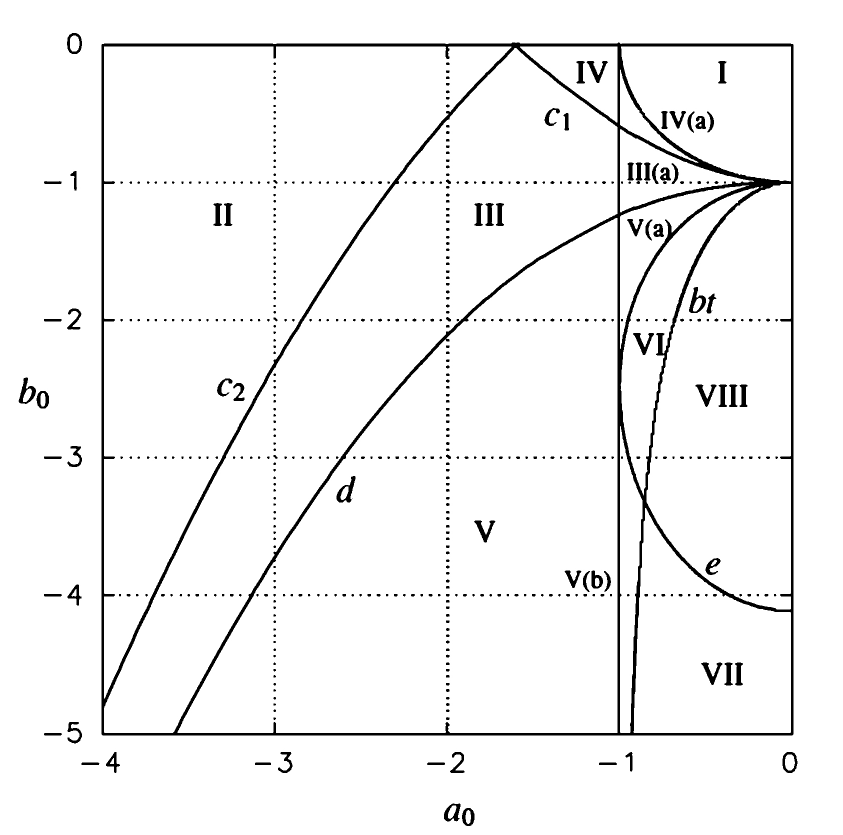}
}
\caption{Partitioning of the $(a_0, b_0)$-plane of system \eqref{eq7.14}}
\label{partion4-1}
\end{figure}
\begin{figure}[ht!]
\centering
{
\includegraphics[width=10cm]{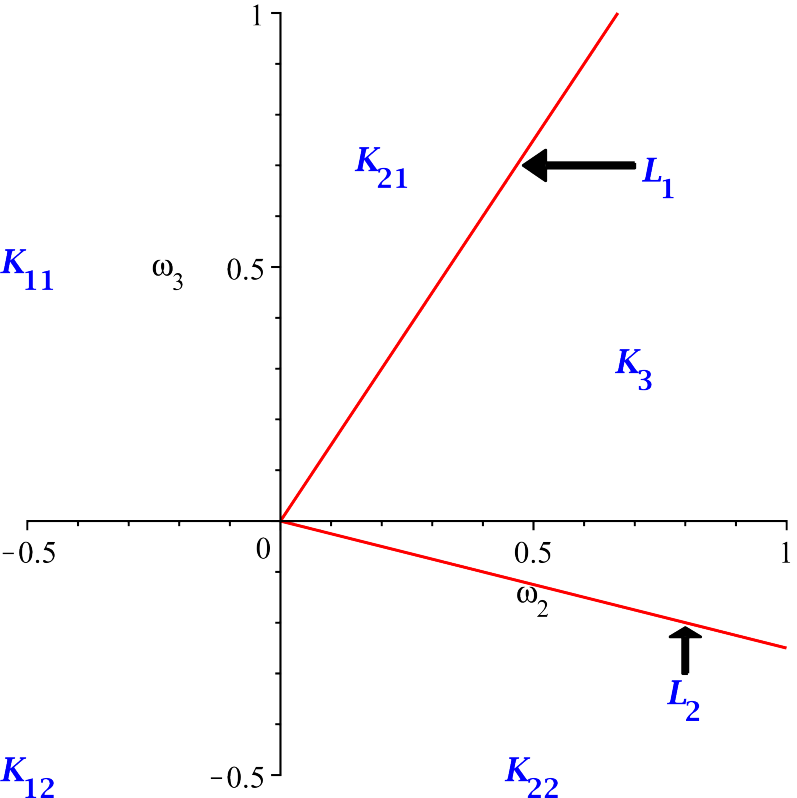}
}
\caption{Bifurcation diagram of system \eqref{eq7.14}.}
\label{rea4-2}
\end{figure}
\begin{proof}
From \cite[pp.82-86]{Kuznetsov1}, we know that system \eqref{eq7.14} has the following bifurcation in a neighbourhood of $(\omega_1,\omega_2)=(0,0)$ for sufficiently small $|\omega|$. Furthermore, we can obtain the bifurcation diagrams of system \eqref{eq7.14}, which consist of parameter diagram and phase diagram as shown in Figures \ref{rea4-2} and \ref{1bi4YJQ1} respectively.
Especially, we give the phase diagrams corresponding to regions $K_3$, $K_{1i}$, $K_{2i}$ and $L_{i}$ ($i=1,2$) of Figure \ref{rea4-2} which are shown in Figure \ref{Y4-1}-\ref{Y4-7} respectively.
\begin{figure}[ht!]
\T\T\T\T\T\T\T\T\T\T\T\T\T\T\T\T\T\T\T\T\T\T\T\T\T\T\T\T\T\T\T\T\T\T\T\T
\subfigure[The phase diagrams of the region $K_{11}$ in Figure \ref{rea4-2}.]{
\includegraphics[width=7.5cm]{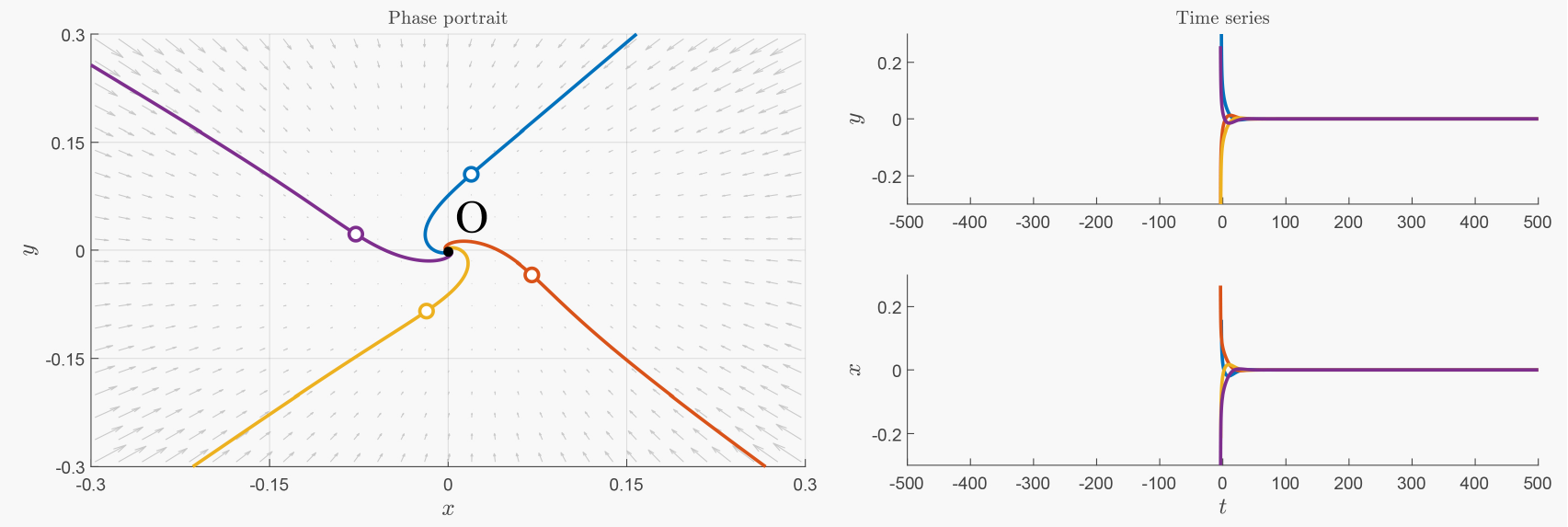}
\label{Y4-1}
}
\quad
\quad
\T\T\T\T\T\T\T\T\T\T\T\T\T\T\T\T\T\T\T\T\T\T\T\T\T\T\T\T\T\T\T\T\T\T\T\T
\subfigure[The phase diagrams of the region $K_{12}$ in Figure \ref{rea4-2}.]{
\includegraphics[width=7.5cm]{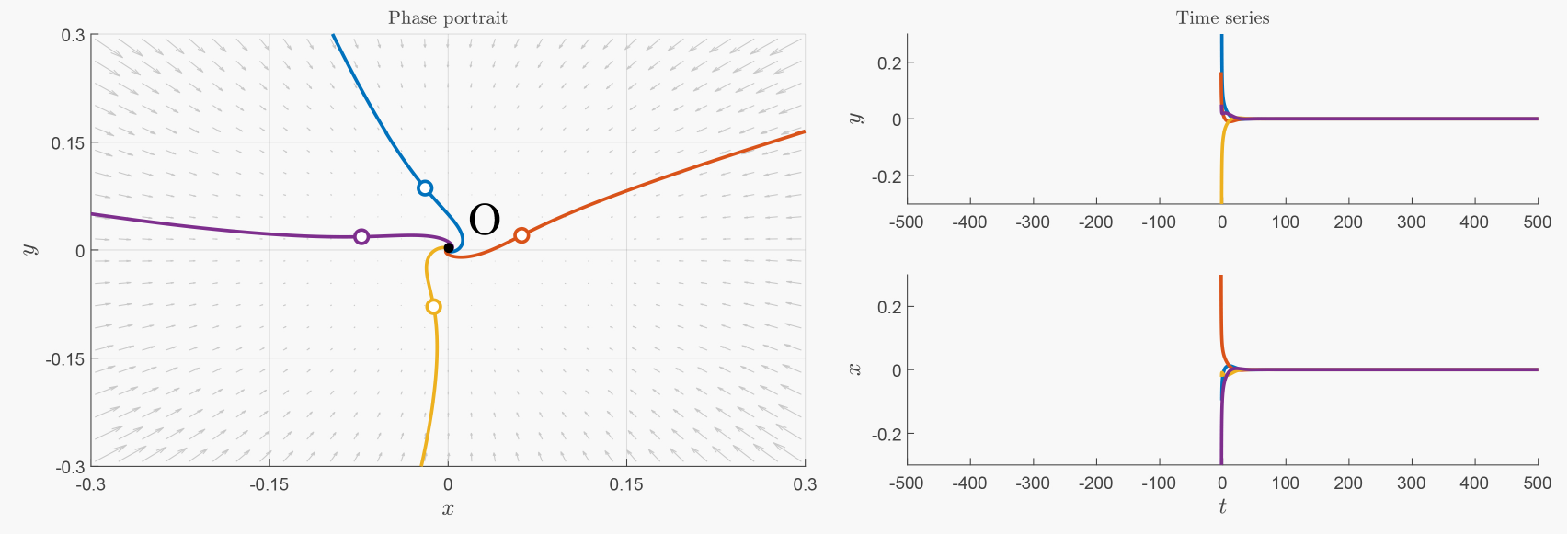}
\label{Y4-2}
}
\quad
\T\T\T\T\T\T\T\T\T\T\T\T\T\T\T\T\T\T\T\T\T\T\T\T\T\T\T\T\T\T\T\T\T\T\T\T
\subfigure[The phase diagrams of the region $K_{21}$ in Figure \ref{rea4-2}.]{
\includegraphics[width=7.5cm]{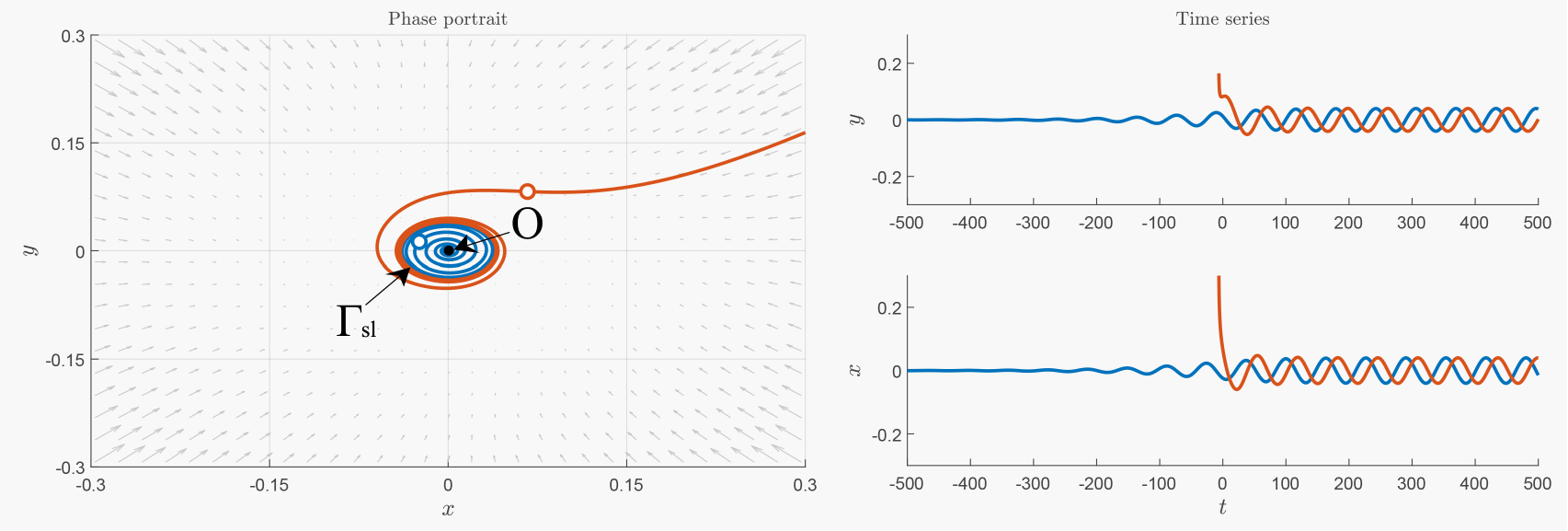}
\label{Y4-3}
}
\quad
\quad
\T\T\T\T\T\T\T\T\T\T\T\T\T\T\T\T\T\T\T\T\T\T\T\T\T\T\T\T\T\T\T\T\T\T\T\T
\subfigure[The phase diagrams of the region $K_{22}$ in Figure \ref{rea4-2}.]{
\includegraphics[width=7.5cm]{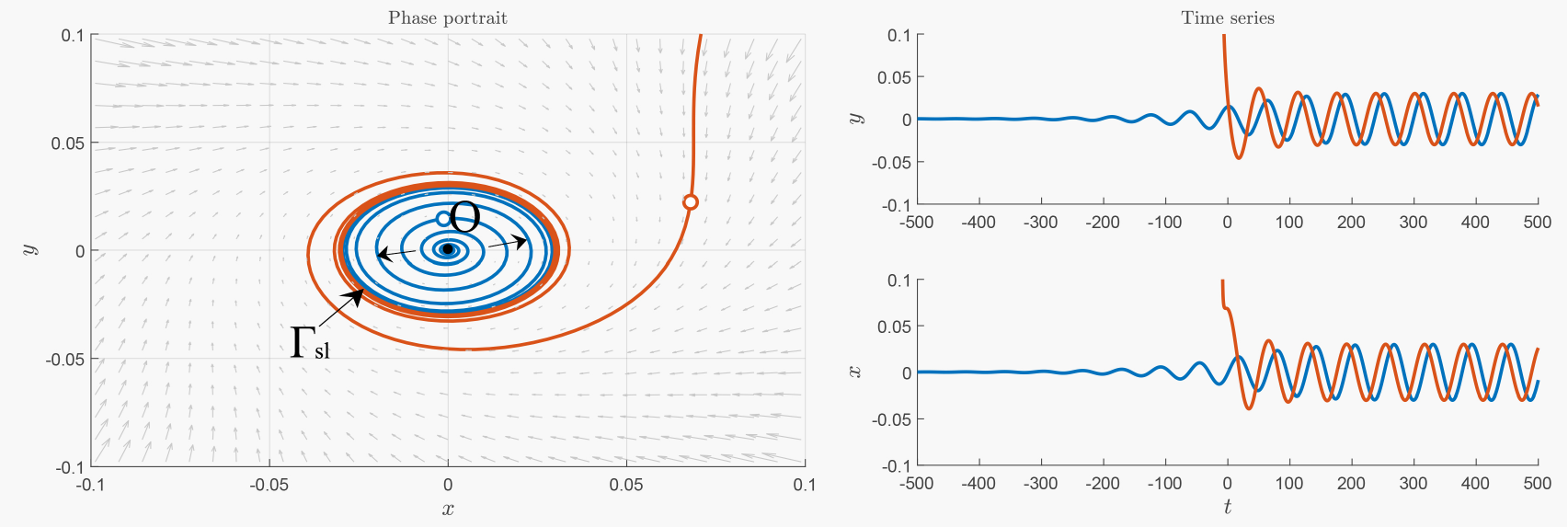}
\label{Y4-4}
}
\T\T\T\T\T\T\T\T\T\T\T\T\T\T\T\T\T\T\T\T\T\T\T\T\T\T\T\T\T\T\T\T\T\T\T\T
\subfigure[The phase diagrams of the region $L_1$ in Figure \ref{rea4-2}.]{
\includegraphics[width=7.5cm]{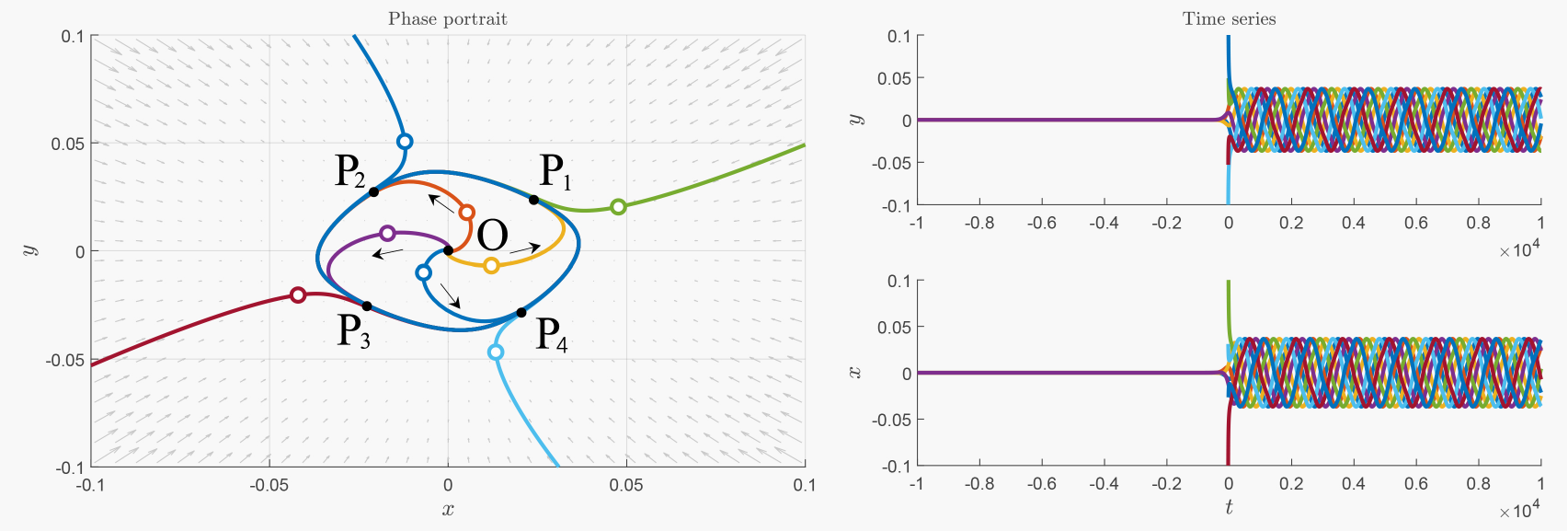}
\label{Y4-5}
}
\quad
\quad
\T\T\T\T\T\T\T\T\T\T\T\T\T\T\T\T\T\T\T\T\T\T\T\T\T\T\T\T\T\T\T\T\T\T\T\T
\subfigure[The phase diagrams of the region $L_2$ in Figure \ref{rea4-2}.]{
\includegraphics[width=7.5cm]{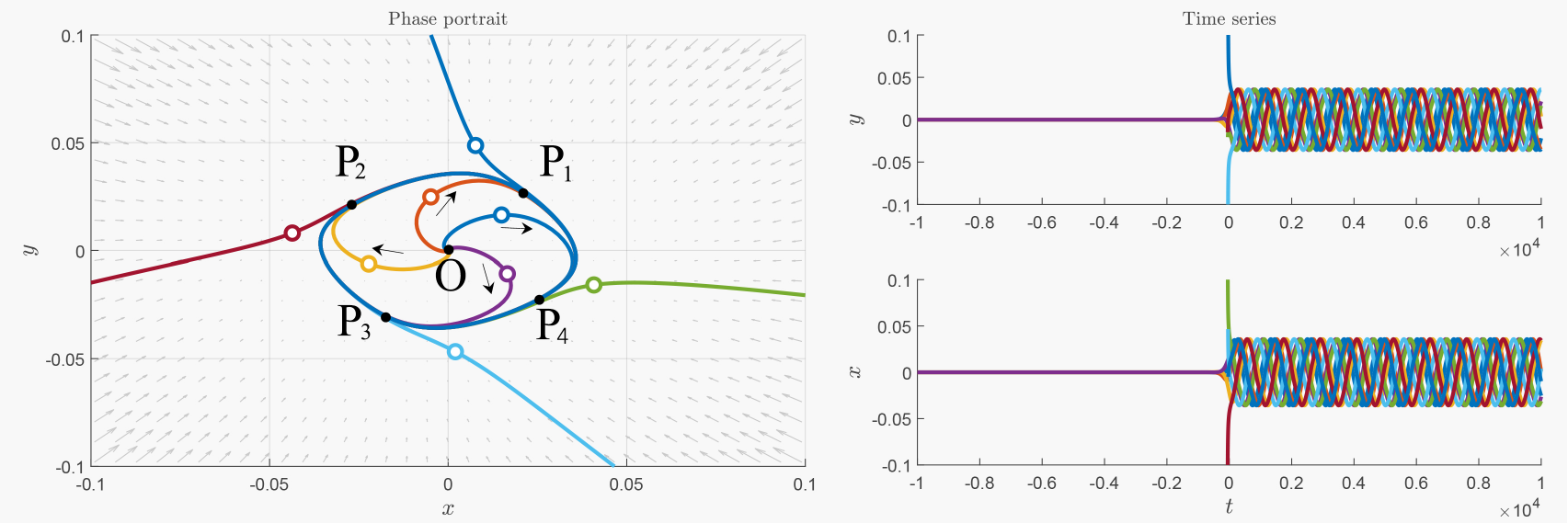}
\label{Y4-6}
}
\centering
\T\T\T\T\T\T\T\T\T\T\T\T\T\T\T\T\T\T\T\T\T\T\T\T\T\T\T\T\T\T\T\T\T\T\T\T
\subfigure[The phase diagrams of the region $K_3$ in Figure \ref{rea4-2}.]{
\includegraphics[width=12cm]{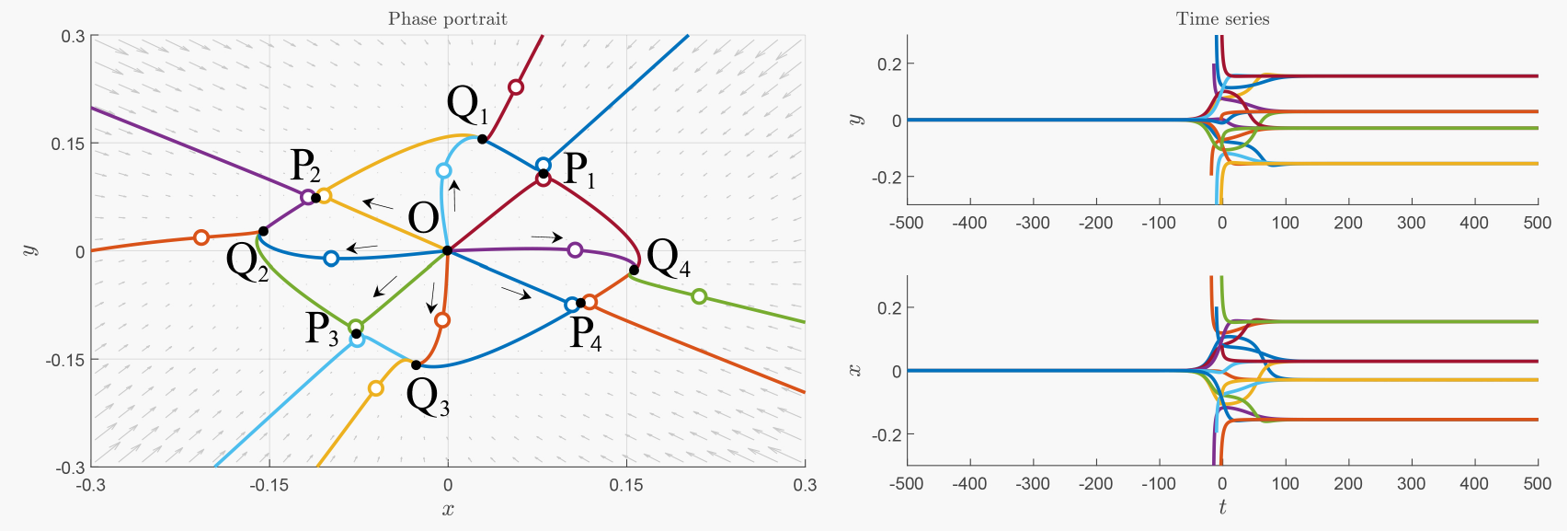}
\label{Y4-7}
}
\caption{phase portraits of system \eqref{eq7.14}.}
\label{1bi4YJQ1}
\end{figure}
\begin{description}
\item[(i)]When the parameter $\breve{\omega}=(\omega_2,\omega_3)$ locates in regions $K_{11}$ and $K_{12}$ in Figure \ref{rea4-2}, system \eqref{eq7.14} has a trivial equilibria $O$ which is a stable focus. And as $\breve{\omega}$ crosses the curve
    $$N:=\{(\omega_2,\omega_3)\in\mathbb{R}^2:\omega_2=0\}$$
    from $K_{11}$ and $K_{12}$ to $K_{21}$ and $K_{22}$, respectively, \eqref{eq7.14} undergoes a Andronnov-Hopf bifurcation and produces a unique stable limit cycle $\Gamma_{sl}$ in $K_{21}$ and $K_{22}$.
\item[(ii)]If $\breve{\omega}$ lies on the curves
    $$
    L_1:\{(\omega_2,\omega_3)\in\mathbb{R}^2:\omega_3=-\frac{\chi_{12}\,\omega_2}{\chi_{13}}+O(\omega^2_2)\}
    $$
    and
    $$L_2:\{(\omega_2,\omega_3)\in\mathbb{R}^2:\omega_3=-\frac{\chi_{14}\,\omega_2}{\chi_{13}}+O(\omega^2_2)\},$$
    where
    {\footnotesize\begin{eqnarray*}
    &&\!\!\!\!\!\!\!\!\!\!\chi_{12}:=6\, ( -{\frac {800\, ( {\beta}^{2}-2\,\beta+2 ) ^{2}\sqrt {H}}{3} ( {\beta}^{2}-{\frac {14\,\beta}{5}}+2 ) ^{2}}+ ( {\beta}^{2}-3\,\beta+1 ) {N}^{4} ( \beta-2 ) ^{8}{\beta}^{5} ( \beta-\frac{2}{3} )) w_{{2}},\\
    &&\!\!\!\!\!\!\!\!\!\!\chi_{13}:=-25600+9\,{N}^{4}{\beta}^{16}-156\,{N}^{4}{\beta}^{15}+1204\,{N}^{4}{\beta}^{14}-5440\,{N}^{4}{\beta}^{13}+15904\,{N}^{4}
    {\beta}^{12}\\
    &&~~~-31360\,{N}^{4}{\beta}^{11}+42112\,{N}^{4}{\beta}^{10}-37888\,{N}^{4}{\beta}^{9}+( 21760\,{N}^{4}-1600) {\beta}^{8}+(-7168\,{N}^{4}\\
    &&~~~+15360) {\beta}^{7}+ ( 1024\,{N}^{4}-67584 ) {\beta}^{6}+178176\,{\beta}^{5}-307712\,{\beta}^{4}+356352\,{\beta}^{3}
    -270336\,{\beta}^{2}\\
    &&~~~+122880\,\beta,\\
    &&\!\!\!\!\!\!\!\!\!\!\chi_{14}:=6\,( {\frac {800\, ( {\beta}^{2}-2\,\beta+2 ) ^{2}\sqrt {H}}{3} ( {\beta}^{2}-{\frac {14\,\beta}{5}}+2 ) ^{2}}+ ( {\beta}^{2}-3\,\beta+1 ) {N}^{4} ( \beta-2 ) ^{8}{\beta}^{5} ( \beta-\frac{2}{3} )) w_{{2}},
    \end{eqnarray*}}system \eqref{eq7.14} exists four non-trivial equilibria $P_i$ ($i=1,\cdots,4$), which are all saddle-nodes. Meanwhile, the four saddle-nodes locates on the limit cycle which is generated by Andronnov-Hopf bifurcation. It forms a heteroclinic cycle.
\item[(iii)] As $\breve{\omega}$ crosses the curves $L_1$ and $L_2$ from $K_{21}$ and $K_{22}$ to $K_3$ respectively,
    system \eqref{eq2.1} undergoes a fold bifurcation and generates eight nontrivial equilibria, which are four saddles $P_i$ and four stable foci $Q_i$ (or stable nodes), where $i=1,\cdots,4$. In addition, four saddles and four foci (or nodes) form a heteroclinic cycle.
\end{description}

Since the time-$1$ flow of the ODE approximates the map $N_{\breve{\omega}}$. Hence, the trivial equilibrium of the approximating system \eqref{eq7.14} corresponds to the trivial fixed point of $N_{\breve{\omega}}$, while four symmetric nontrivial equilibria actually correspond to one period-$4$ cycle. The Hopf bifurcation line $N$ approximates the Neimark-Sacker bifurcation line $\mathfrak{L}_2$ of the trivial fixed point of $N_{\breve{\omega}}$. Moreover, the leading terms of the asymptotic expressions for the respective curves coincide. As usual, heteroclinic connections in the approximating system \eqref{eq7.14} become homoclinic structures for the map $N_{\breve{\omega}}$. They are formed by intersections of the stable and unstable invariant manifolds of the saddle period-$4$ cycle. These structures imply the existence of an infinite number of periodic orbits. Closed invariant curves corresponding to limit cycles lose their smoothness and are destroyed, almost   colliding   with the saddle period-$4$ cycle. All these phenomena are also present in the full (nontruncated) normal form \eqref{eq7.6} under the same non-degeneracy conditions.

Finally, we give the bifurcation curves in the original parameter $(r,\alpha)$-plane. From the parameter transformations \eqref{eq7.1} and \eqref{eq7.8}, one can check that the bifurcation curves $N$, $L_1$ and $L_2$ in the $(\omega_2,\omega_3)$-plane correspond to the curves $\mathfrak{L}_4$,
$H_{41}$ and $H_{42}$ in the $(r,\alpha)$-plane, respectively. The proof is completed.
\end{proof}

From Theorem \ref{th6.3}, we know that mapping \eqref{eq2.1} experiences a $1:4$ resonance and has many complex dynamic characteristics nearby the $1:4$ resonance point $(1,1)$. Thus, when the parameter $(\lambda,\mu)$ is close to $(1,1)$, the dynamic characteristics of the mapping exhibit high sensitivity to variations in the parameter.
Furthermore, in the field of biology, non-degenerate Neimark-Sacker bifurcation can give rise to the periodic or quasi-periodic fluctuations of  the numbers of susceptible, infective and recovered individuals. Additionally, an invariant periodic orbit with period-4 implies that, after 4 time intervals, the numbers of susceptible, infective and recovered individuals will shift from a stable state to a repeatable (approaching) state.

\section{Arnold tongues}

Based on the results proved in Theorem \ref{th5.1}, in this section we investigate the parameter conditions for the Arnold tongue on the stable invariant circle. Under these conditions, system \eqref{eq2.1} exhibits periodic orbits with period $m$, where $m\geq 5$ and $m\in \mathbb{Z}^+$. As shown in \cite{Arrowsmith,Whitley}, for this phenomenon to occur, system \eqref{eq2.1} must have a pair of complex roots that cross the unit circle at the roots of unity $\varrho_0:= e^{(2\pi\,{\bf i}\,n/m)}$. In other words, the eigenvalues $t_{1}$ and $t_{2}$, given in $\eqref{solu-1}$, cross the unit circle at the roots of unity $\varrho:= e^{(2\pi\,{\bf i}\,n/m)}$. Here, $n$ is a positive integer such that $n/m\in(0,1)$ is irreducible.

\begin{thm}
Suppose that $\beta\neq-2\,cos(2\,\pi\,n/m)+2$, where $m\geq5$ is a positive integer and $n/m\in(0,1)$ with $n\in \mathbb{Z}^+$ is an irreducible fraction. Assume further that the parameter $(r,\alpha)$ is in the vicinity of $(r_*,\alpha_*)$ and lies within an Arnold tongue $\mathcal{A}_{n/m}$ defined as follows:
$$\mathcal{A}_{n/m}=\left\{(r,\alpha)\bigg|T_{-}+\frac{2\,\pi\,n}{m}-\Xi_1<\arctan\left(\frac{\sqrt{\Xi_0}}{\left( 2-\alpha
 \right) \beta+2\,r}\right)<T_{+}+\frac{2\,\pi\,n}{m}-\Xi_1\right\},$$
where
\begin{eqnarray*}
&&\Xi_0:=-\beta\, \left( 4\,{r}^{3}+ \left( 12\,\beta-4\,\alpha
 \right) {r}^{2}+ \left( -8\,\beta\,\alpha+12\,{\beta}^{2} \right) r+
 \left( 2\,\beta-\alpha \right) ^{2}\beta \right),\\
&&r_*:=-{\frac {{\beta}^{2}+2\,\beta\,cos(\frac{2\pi n}{m})-2\,\beta-2\,cos(\frac{2\pi n}{m})+2}{\beta+2\,cos(\frac{2\pi n}{m})-2}}, \alpha_*:={\frac {-4\, \left( cos(\frac{2\pi n}{m})-1 \right) ^{2}}{\beta\, \left( \beta+2\,cos(\frac{2\pi n}{m})-2 \right) }},\\
&&T_{\pm}\approx\frac{\tilde{\varrho}_{2}(0)}{\check{\varrho}_{3}(0)}\,\left(\sqrt{{\frac {-{\beta}^{3}+ \left( \alpha-2\,r \right) {\beta}^{2}- \left( r
-1 \right)  \left( r-\alpha+1 \right) \beta+r}{\beta+r}}
}-1\right)\\
&&~~~~~~~\pm\frac{|\varsigma(0)|}{|\check{\varrho}_{3}(0)|^{(m-2)/2}}\,\left(\sqrt{{\frac {-{\beta}^{3}+ \left( \alpha-2\,r \right) {\beta}^{2}- \left( r
-1 \right)  \left( r-\alpha+1 \right) \beta+r}{\beta+r}}
}-1\right)^{\frac{m-2}{2}},
\end{eqnarray*}
and
\begin{eqnarray*}
\Xi_1:=
\left\{
\begin{array}{llll}
0,~~~\mbox{when}~~{\frac {-\alpha\,\beta+2\,\beta+2\,r}{2\,\beta+2\,r}}>0,{\frac {\sqrt {\Xi_0}}{2\,\beta+2\,r}}>0,\\
\pi,~~~\mbox{when}~~{\frac {-\alpha\,\beta+2\,\beta+2\,r}{2\,\beta+2\,r}}<0,{\frac {\sqrt {\Xi_0}}{2\,\beta+2\,r}}>0,\\
-\pi,~~~\mbox{when}~~{\frac {-\alpha\,\beta+2\,\beta+2\,r}{2\,\beta+2\,r}}<0,{\frac {\sqrt {\Xi_0}}{2\,\beta+2\,r}}<0,\\
0,~~~\mbox{when}~~{\frac {-\alpha\,\beta+2\,\beta+2\,r}{2\,\beta+2\,r}}>0,{\frac {\sqrt {\Xi_0}}{2\,\beta+2\,r}}<0,\\
\end{array}
\right.
\end{eqnarray*}
system \eqref{eq2.1} possesses two $m$-periodic orbits, one attracting and one repelling, on the invariant circle that arises from the Neimark-Sacker bifurcation.
\label{th7.1}
\end{thm}
\begin{proof}
To facilitate the discussion, we express $t_{1}$ as presented in $\eqref{solu-1}$ in the following exponential form:
\begin{eqnarray}
&&\tilde{t}=(1+\varpi_1(r,\alpha))\,exp({\bf i}\,(2\,\pi\,n/m+\varpi_2(r,\alpha))).
\label{a1}
\end{eqnarray}
Obviously, $\varpi_1(r_*,\alpha_*)=\varpi_2(r_*,\alpha_*)=0$.
Drawing on the Poincar$\acute{e}$ normal form theory (cf. Lemma 2 in \cite[p.44]{Iooss} and \cite[p.259]{Arrowsmith}) and the center manifold theory (the dimensionality-reduction method employed in the proof of Theorem \ref{th5.1}), when $(r,\alpha)$ is in the vicinity of $(r_*,\alpha_*)$, system \eqref{eq2.1} can be expressed in the following complex-valued normal form:
\begin{eqnarray}
&&z\mapsto\tilde{t}\,z+\sum_{k=1}^{[(m-2)/2]}\varrho_{k+1,k}(\varpi)\,z^{k+1}\,\bar{z}^{k}+\varsigma(\varpi)\,\bar{z}^{m-1}+O(|z|^m),
\label{a2}
\end{eqnarray}
where $\varpi:=(\varpi_1(r,\alpha),\varpi_2(r,\alpha))$, $[\cdot]$ denotes the integer part of $\cdot$, and $\varrho_{k+1,k}(\varpi)$ and $\varsigma(\varpi)$ are all analytic functions of $\varpi$. Let $z:=\check{\rho}\,e^{(2\pi{\bf i}\theta)}$. Then system \eqref{a2} is transformed into the following polar coordinate form:
\begin{eqnarray}
\begin{aligned}
\!\!\!\!\!\!\!\left(
\begin{array}{cc}
\check{\rho} \\
\theta
\end{array}
\right)\!\!
\rightarrow\!\!
\left(
\begin{array}{l}
(1+\varpi_1)\,\check{\rho}+\sum_{k=1}^{[(m-2)/2]}\check{\varrho}_{2\,k+1}(\varpi)\,\check{\rho}^{2\,k+1}
+\check{\varsigma}(\varpi,\theta)\,\check{\rho}^{m-1}+O(\check{\rho}^{m})\\
\theta+\frac{2\,\pi\,n}{m}+\varpi_2+\sum_{k=1}^{[(m-2)/2]}\tilde{\varrho}_{2\,k}(\varpi)\,\check{\rho}^{2\,k+1}
+\tilde{\varsigma}(\varpi,\theta)\,\check{\rho}^{m-2}+O(\check{\rho}^{m-1})
\end{array}
\right)\!\!,
\end{aligned}
\label{a3}
\end{eqnarray}
where $\check{\varrho}_{2\,k+1}$ and $\tilde{\varrho}_{2\,k}$ are all analytic functions independent of $\theta$,
\begin{eqnarray*}
&&\check{\varsigma}(\varpi,\theta)=\mathrm{Re}(\varsigma(\varpi)\,e^{(-{\bf i}\,(m\,\theta+\frac{2\,\pi\,n}{m}+\varpi_2))}),\\
&&\tilde{\varsigma}(\varpi,\theta)=\mathrm{Im}\left(\frac{\varsigma(\varpi)\,e^{(-{\bf i}\,(m\,\theta+\frac{2\,\pi\,n}{m}+\varpi_2))}}{1+\varpi_1}\right).
\end{eqnarray*}
Specifically, the following computations can be carried out:
\begin{eqnarray*}
&&\check{\varrho}_{3}(\varpi)=\mathrm{Re}(\varrho_{2,1}(\varpi)\,e^{(-{\bf i}\,(\frac{2\,\pi\,n}{m}+\varpi_2))}),\\
&&\tilde{\varrho}_{2}(\varpi)=\frac{\mathrm{Im}(\varrho_{2,1}(\varpi)\,e^{(-{\bf i}\,(\frac{2\,\pi\,n}{m}+\varpi_2))})}{1+\varpi_1}.
\end{eqnarray*}
According to Theorem 2.4 in \cite{Whitley}, for system \eqref{a3}, within the parameter $(\varpi_1,\varpi_2)$-plane, there exist two $m$-periodic orbits, one attracting and one repelling, on the invariant circle resulting from the Neimark-Sacker bifurcation.
This occurs when $\varpi$ lies within a tongue $\mathcal{A}_{n/m}$ whose boundaries are given by:
\begin{eqnarray}
&&\varpi_2\approx\frac{\tilde{\varrho}_{2}(0)\,\varpi_1}{\check{\varrho}_{3}(0)}\pm\frac{|\varsigma(0)|\,\varpi_1^{(m-2)/2}}{|\check{\varrho}_{3}(0)|^{(m-2)/2}}.
\label{a4}
\end{eqnarray}
Moreover, we can compute that
\begin{eqnarray*}
&&\!\!\!\!\check{\varrho}_{3}(0)=-{\frac { \left( {\beta}^{2}+\beta\,r_{{*}}-1 \right)  \left( r_{
{*}}+\beta \right) ^{4}}{ 8\,\left( r_{{*}}+\beta-1 \right) {N}^{2}}},\\
&&\!\!\!\!\tilde{\varrho}_{2}(0)=-\frac{\chi_{15}}{\chi_{16}},
\end{eqnarray*}
where
\begin{eqnarray*}
&&\!\!\!\!\chi_{15}:=\sqrt {-\beta\,( {\beta}^{2}+ ( r_{{*}}-4) \beta-4\,r_{{*}}+4)( r_{{*}}+\beta ) }( r_{{*}}+\beta) ^{3}( {\beta}^{7}+( 3\,r_{{*
}}-5 ) {\beta}^{6}+( 3\,r_{*}^{2}\\
&&~~~~~~-15\,r_{{*}}+8) {\beta}^{5}+ ( r_{*}^{3}-*5\,r_{*}^{2}+20\,r_{{*}}-* ) {\beta}^{4}+(-5\,r_{*}^{3}+16\,r_{*}^{2}+2\,r_{{*}}-9 ) {\beta}^{3}+ \\
&&~~~~~~( 4\,r_{*}^{3}+7\,r_{*}^{2}-17\,r_{{*}}+6 ) {\beta}^{2}+ ( 4\,r_{*}^{3}-6\,r_{*}^{2}+2\,r_{{*}} ) \beta+2\,r_{{*}}( r_{{*}}-1) ^{2} ) ,\\
&&\!\!\!\!\chi_{16}:=8\, \left( {\beta}^{2}+ \left( r_{{*}}-3 \right)
\beta-3\,r_{{*}}+3 \right) {\beta}^{2} \left( {\beta}^{2}+ \left( r_{{
*}}-4 \right) \beta-4\,r_{{*}}+4 \right) {N}^{2} \left( r_{{*}}+\beta-
1 \right).
\end{eqnarray*}
Furthermore, as observed from \eqref{a2}, we see that $|\varsigma(0)|$ varies with $m$. Additionally, from \eqref{a1} we know that $|\tilde{t}|=1+\varpi_1$ and the argument of $\tilde{t}$, denoted as $\arg(\tilde{t})$, is given by $\arg(\tilde{t})=2\pi n/m+\varpi_2$. By incorporating  $\eqref{solu-1}$, we obtain the following expressions:
\begin{eqnarray*}
&&\!\!\!\!\!\!\!\!|\tilde{t}|=\sqrt{{\frac {-{\beta}^{3}+ \left( \alpha-2\,r \right) {\beta}^{2}- \left( r
-1 \right)  \left( r-\alpha+1 \right) \beta+r}{\beta+r}}
},\\
&&\!\!\!\!\!\!\!\!\arg(\tilde{t})=\arctan\left({\frac {\sqrt {\Xi_0 }}{ \left( 2-\alpha
 \right) \beta+2\,r}}
\right)+\Xi_1.
\end{eqnarray*}
Consequently,
\begin{eqnarray}
&&\varpi_1=\sqrt{{\frac {-{\beta}^{3}+ \left( \alpha-2\,r \right) {\beta}^{2}- \left( r
-1 \right)  \left( r-\alpha+1 \right) \beta+r}{\beta+r}}
}-1,
\label{a5}
\end{eqnarray}
\begin{eqnarray}
&&\varpi_2=\arctan\left({\frac {\sqrt {\Xi_0 }}{ \left( 2-\alpha
 \right) \beta+2\,r}}
\right)+\Xi_1-\frac{2\pi n}{m}.
\label{a6}
\end{eqnarray}
By substituting \eqref{a5} and \eqref{a6} into \eqref{a4}, we can determine the values of $T_{\pm}$ and $\mathcal{A}_{n/m}$. The proof is completed.
\end{proof}
\begin{rmk}
The $\varsigma(0)$ in Theorem \ref{th7.1} is the coefficient of the term $\bar{z}^{m-1}$ in the equation \eqref{a2} evaluated at $\varpi=(0,0)$ . Evidently, the function $\varsigma(\varpi)$ depends on $m$. Generally speaking, as the value of $m$ decreases, the expressions for $\varsigma(\varpi)$ become simpler.
\end{rmk}

According to Theorem \ref{th7.1}, if an Arnold tongue $\mathcal{A}_{n/m}$ region exists, then the system \eqref{eq2.1} has two $m$-periodic orbits on the invariant closed curve resulting from the Neimark-Sacker bifurcation: one is attracting and the other is repelling. This indicates that the mapping \eqref{eq2.1} can transition from a stable state to an oscillatory state, specifically, $m$-periodic fluctuations. Moreover, we can infer that the numbers of susceptible, infective, and recovered individuals vary periodically. Such periodic variations suggest that the population is on a path towards extinction. From a biological standpoint, it is thus imperative that we prevent the occurrence of this situation.

\section{Numerical simulations}
In this section, we will conduct numerical simulations of the dynamic phenomena of the mapping \eqref{eq2.1} to validate the results presented in the previous sections.

First, we simulate the codimension 1 bifurcation phenomena in the system \eqref{eq2.1}.
According to Theorem \ref{th3.1}, for $N>0,0<\beta<1$ and $0<r<1$, mapping \eqref{eq2.1} undergoes a transcritical bifurcation at the fixed point $E_1$ when the parameter $\alpha$ crosses the value $\alpha=\beta+r$. To detect this bifurcation phenomenon for the fixed point $E_1$, we utilize MatcontM in {\it MatlabR2022a} (see \cite{Kuznetsov1}). We set the parameter values as follows: $N=0.51$, $\beta=0.31$, $r=0.27$ and $\alpha=0.76461239$, where $\alpha$ is an active parameter. Taking an initial value $(x_{1},y_{1},z_{1})=(0.38686268,0.065814774,0.057322545)$, we select the ``Initial Point Type'' as ``Fixed Point'' and leave the other options as the system defaults in MatcontM. We halt the ``Forward'' computation once the first ``BP'' point is detected. This ``BP'' point occurs when two fixed - point curves intersect in the $(\alpha,x)$-plane (as shown in Figure \ref{BP-1}). This intersection indicates that the transcritical bifurcation has taken place.

Theorem \ref{th4.1} indicates that for $N>0$, $0<\beta<1$ and $0<r<\Psi_1$, when $\Theta_1\neq0$ and $\Theta_2>0$, mapping (\ref{eq2.1}) experiences a supercritical flip bifurcation. As the parameter $\alpha$ crosses $\alpha=\Psi_2$., a stable period-two cycle emerges near the fixed point $E_2$. Thus, under the default settings described in the previous paragraph, we halt the ``Backward'' computation once a ``PD'' point, which might be subject to a flip bifurcation, is detected (see Figure \ref{BP-1}). From the ``Output'' window in MatcontM, we observe that the normal form coefficient at the ``PD'' point is $37.1311$, signifying that the mapping \eqref{eq2.1} undergoes a supercritical flip bifurcation.
Furthermore, using the ``PD'' point as an initial value, we change the ``Initializer'' to ``FP-curve x2'' while keeping the other options as the system defaults in MatcontM. We stop the computation upon detecting the first ``PD'' point and then select this ``PD'' point as an initial value (see Figure \ref{PD-1}). This process reveals that mapping \eqref{eq2.1} generates a stable period-two cycle.
In addition, we will numerically simulate a cascade resulting from the supercritical flip bifurcation using {\it MatlabR2022a}. We choose
$\alpha\in [3.95, 4.85]$ as the bifurcation parameter and set the parameters $N=0.72$, $\beta=0.52$, $r=0.21$, and an initial value $(x_{2},y_{2},z_{2})=(0.131335764, 0.4193224693, 0.1693417666)$ (note that $\Theta_1=-0.7871437846\neq0$ and $\Theta_2=2344.468744\ne 0$).
With {\it MatlabR2022a}, we aim to plot the bifurcation diagram of the mapping \eqref{eq2.1} in the $(A,x,y,z)-$space. However, this is not feasible in four-dimensional space. As a result, we simulate the projections of the flip bifurcation diagram of the mapping \eqref{eq2.1} onto three different three-dimensional spaces: the $(A,x,y)$-space, the $(A,x,z)$-space and the $(A,y,z)$-space. These are presented in Figures \ref{flip-1}, \ref{flip-2}, and \ref{flip-3} respectively, all of which display a cascade of flip bifurcations. This implies that such a cascade also exists in the $(A,x,y,z)-$space.
In conclusion, the numerical simulation shows that the mapping \eqref{eq2.1} undergoes a supercritical flip bifurcation and generates a stable period-two cycle near the fixed point $E_2$. The other instances of flip bifurcations can be simulated in a similar manner.

\begin{figure}[ht!]
\subfigure[Detection of transcritical bifurcation points.]{
\includegraphics[width=7.2cm]{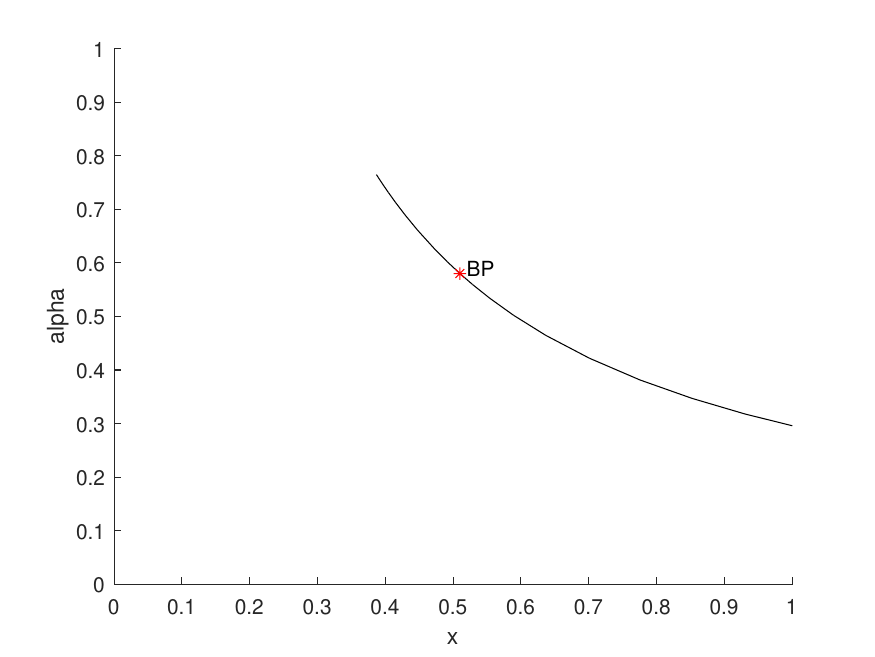}
\label{BP-1}
}
\quad
\subfigure[Detection of the stable period-two cycle for the supercritical flip bifurcation.]{
\includegraphics[width=7.2cm]{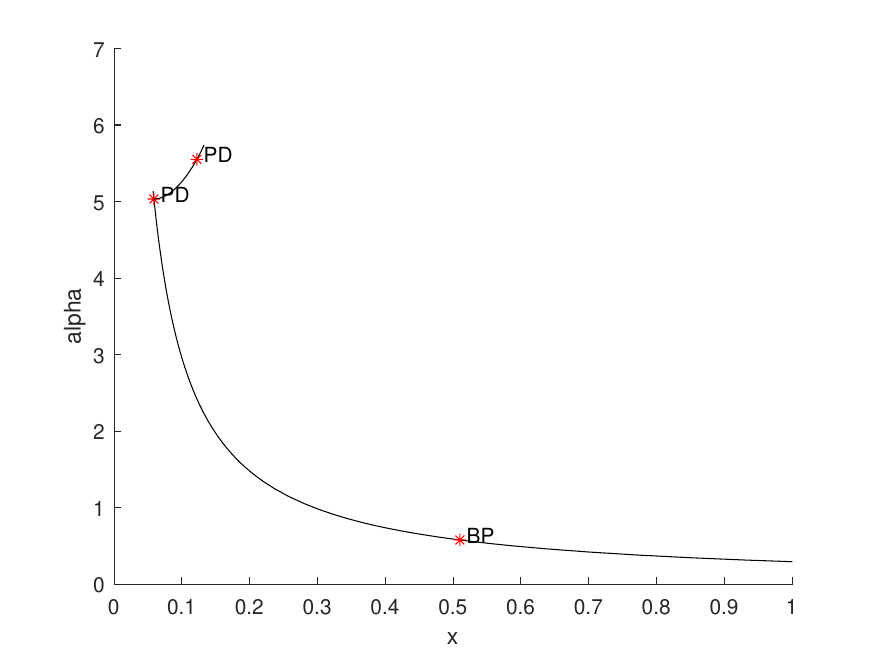}
\label{PD-1}
}
\caption{The continuation of $E_1$ and $E_2$ through MatcontM in the $(\alpha,x)$-plane.}
\end{figure}
\begin{figure}[ht!]
\T\T\T\T\T\T\T\T\T\T\T\T\T\T\T\T\T\T\T\T\T\T\T\T\T\T\T\T\T\T\T\T\T\T\T\T
\subfigure[Projection diagram in $(\alpha,x,y)$-space with $N=0.72$, $\beta=0.52$ and $r=0.21$.]{
\includegraphics[width=7.5cm]{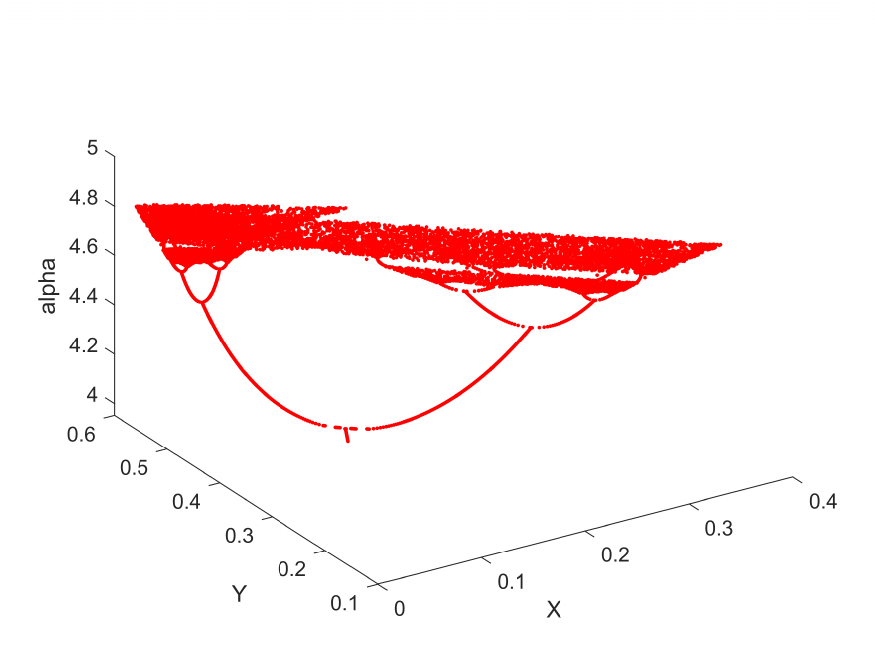}
\label{flip-1}
}
\quad
\subfigure[Projection diagram in $(\alpha,x,y)$-space with $N=0.72$, $\beta=0.52$ and $r=0.21$.]{
\includegraphics[width=7.5cm]{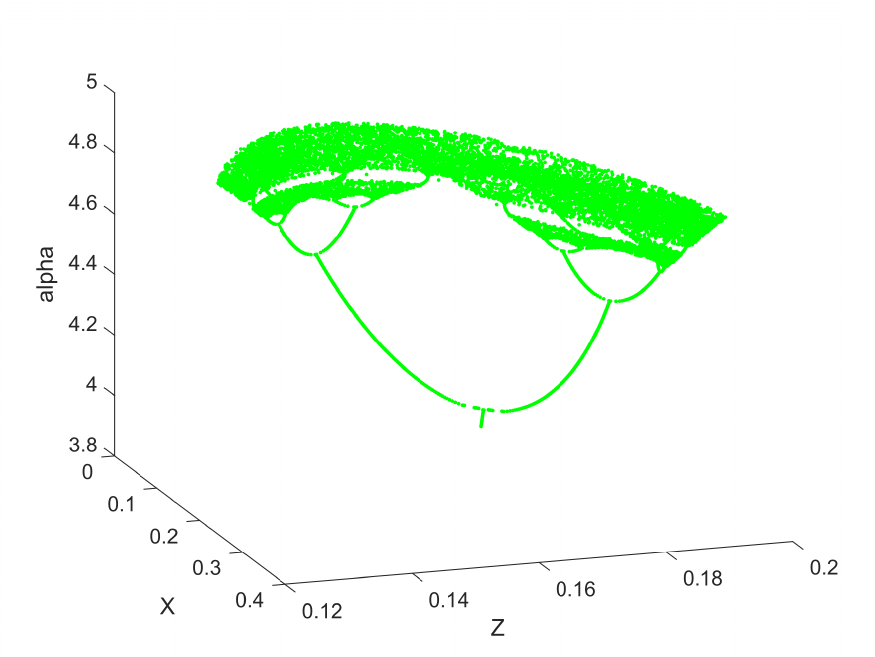}
\label{flip-2}
}
\quad
\centering
\subfigure[Projection diagram in $(\alpha,x,y)$-space with $N=0.72$, $\beta=0.52$ and $r=0.21$.]{
\includegraphics[width=7.5cm]{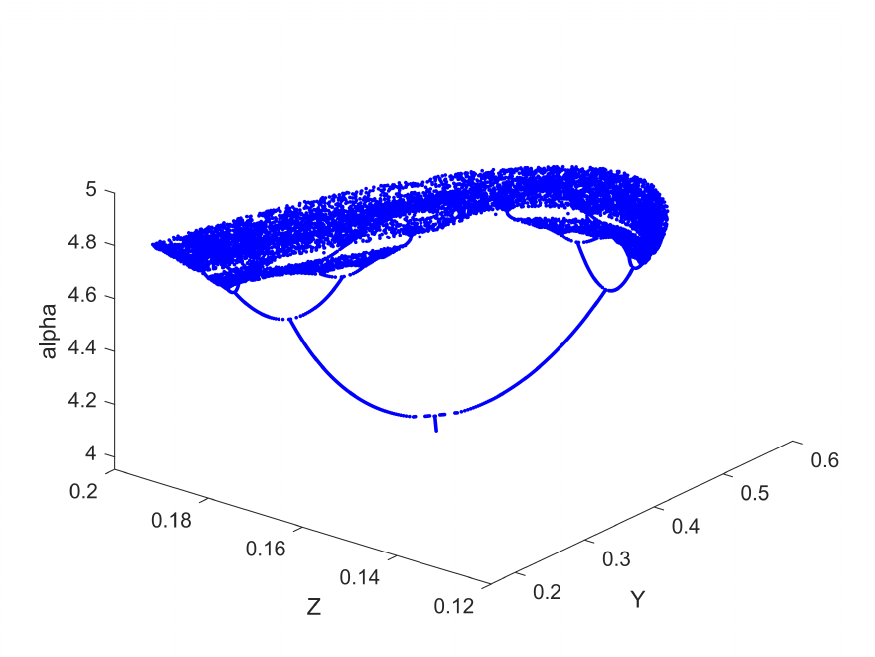}
\label{flip-3}
}
\caption{Projection diagrams of flip bifurcation diagram on different spaces as parameter $\alpha$ cross $\Psi_2$ when $N>0$, $0<\beta<1$ and $0<r<\Psi_1$.}
\end{figure}

According to Theorem \ref{th5.1}, when the parameter crosses $\mathfrak{L}_2$ and satisfies the following conditions:
\begin{eqnarray*}
&&r\neq-{\frac {{\beta}^{2}-3\,\beta+3}{\beta-3}},~r\neq-{\frac {{\beta}^{2}-2\,\beta+2}{\beta-2}}, r\neq1-\beta~~\mbox{and}~~r\neq-\frac{\beta^2-1}{\beta},
\end{eqnarray*}
mapping \eqref{eq2.1} gives rise to a Neimark-Sacker bifurcation in the vicinity of $E_2$.
Therefore, by choosing the parameters $N=1.25$, $\beta=0.32$, $r=0.7983$ and $\alpha=10.495403$, where $\alpha$ serves as an active parameter, and taking an initial value $(x_{31},y_{31},z_{31})=(0.13016993,0.32614041,0.79340337)$, setting the ``Initial Point Type'' as ``Fixed Point'' while keeping the other options at the system default in MatcontM, our computations detect a ``NS'' point that may be associated with a Neimark-Sacker bifurcation, as depicted in Figure \ref{NS-1}. From the ``Output'' window in MatcontM, we observe that the normal form coefficient at the ``NS'' point is $1.588549>0$. This indicates that the mapping \eqref{eq2.1} experiences a subcritical Neimark-Sacker bifurcation, resulting in the emergence of a unique unstable invariant circle. Furthermore, when setting the parameters $N=3.72$, $\beta=0.52$, $r=0.81$ and $\alpha=5.36$, and taking two initial values $(x_{32},y_{32},z_{32}) = (0.93896381,1.0956095,1.7)$ and $(x_{33},y_{33},z_{33}) = (0.93896382,1.0956095,1.7)$, we plot the bifurcation diagram of mapping \eqref{eq2.1} in the $(x,y,z)-$space (see Figure \ref{ns-1}). In Figure \ref{ns-1}, we can see two orbits of different colors. The red orbit moves away from the invariant circle $\Gamma$ and converges to the fixed point $E_2$, while the blue orbit departs from $\Gamma$ and goes to infinity. Thus, there exists an unstable invariant circle $\Gamma$ at the intersection of the two orbits. This implies that the mapping \eqref{eq2.1} undergoes a subcritical Neimark-Sacker bifurcation in the vicinity of $E_2$.

\begin{figure}[ht!]
\centering
{
\includegraphics[width=10cm]{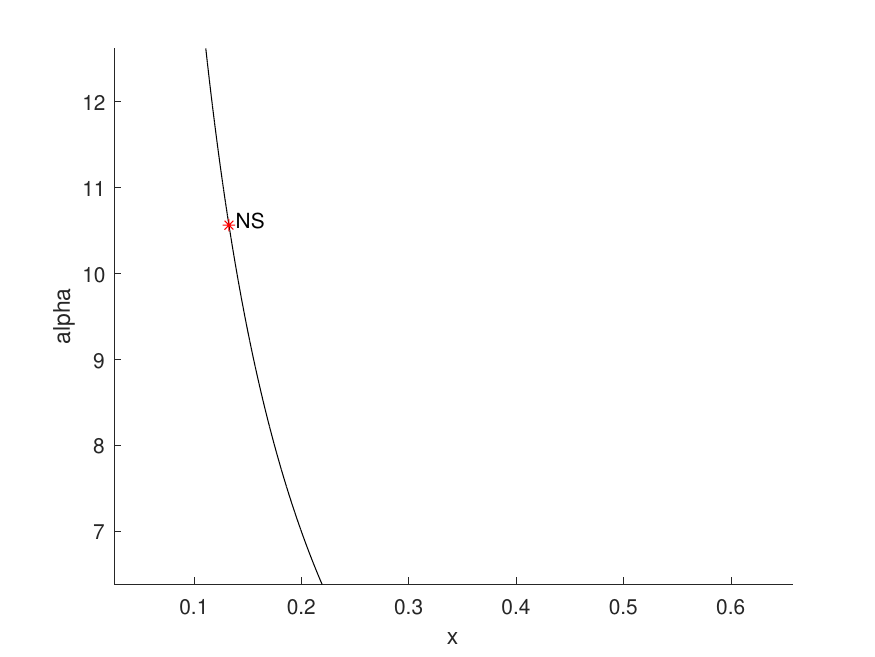}
}
\caption{Detection of the subcritical Neimark-Sacker bifurcation.}
\label{NS-1}
\end{figure}
\begin{figure}[htbp]
\centering
{
\includegraphics[width=10cm]{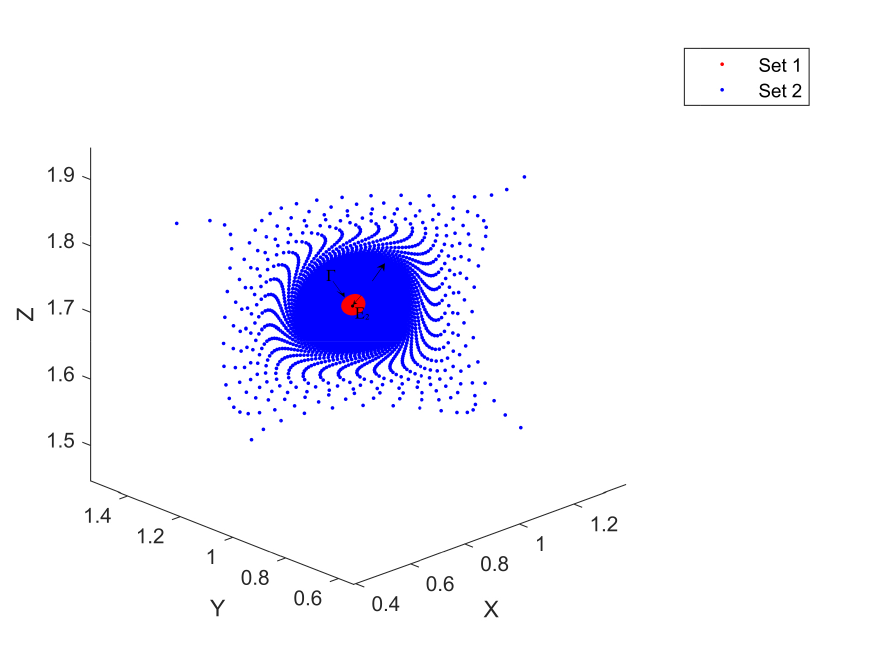}
}
\caption{An unstable invariant circle $\Gamma$ generated from the Neimark-Sacker bifurcation as the parameter crosses $\mathfrak{L}$.}
\label{ns-1}
\end{figure}
Subsequently, to further validate the accuracy of the content in Section 4, we will simulate the codimension 2 bifurcations of the mapping \eqref{eq2.1}. Taking the ``NS'' point depicted in Figure \ref{NS-1} as a basis, we commence by simulating the Neimark-Sacker bifurcation curve. This curve serves as the organizing center for codimension-2 bifurcations, enabling transitions to other bifurcation curves. Specifically, we designate a ``NS'' point within Figure \ref{NS-1} as the initial value. In MatcontM, we choose  and  as the active parameters, while maintaining the other options at the system defaults. The computation is halted once ``R2'', ``R3'', ``R4'' and ``CH'' points are detected. These points respectively correspond to the $1:2$, $1:3$ and $1:4$ resonance points, , as well as a point experiencing a Chenciner bifurcation (refer to Figure \ref{cod-2}).
Furthermore, from the ``Output'' in MatcontM, we have the following results:
\\
{\tt label=R2, x=(0.100000~0.338560~0.811440~13.586957~0.766957~-1.000000),\\
normal form coefficient of R2: [c,d]=1.969247e+01,-1.181541e+02,\\
label=R3, x=(0.133333~0.319218~0.797499~10.494403~0.799403~-0.500000),\\
normal form coefficient of R3: Re(c\_1)=4.281802e-01,\\
label=R4, x=(0.200000~0.282240~0.767760~7.440476~0.870476~0.000000),\\
normal form coefficient of R4: A=1.120520e-01+9.589064e-02i,\\
label=CH, x=(0.850000~0.040960~0.359040~4.595588~2.805000~0.764706),\\
normal form coefficient of CH=1.46786e+00.}\\
Since these normal form coefficients are non - zero, we can conclude that system \eqref{eq2.1} experiences $1:2$, $1:3$ and $1:4$ resonances, as well as a Chenciner bifurcation in. This validates the conclusions of Theorems \ref{th6.1}-\ref{th6.3}.

\begin{figure}[ht!]
\centering
{
\includegraphics[width=10cm]{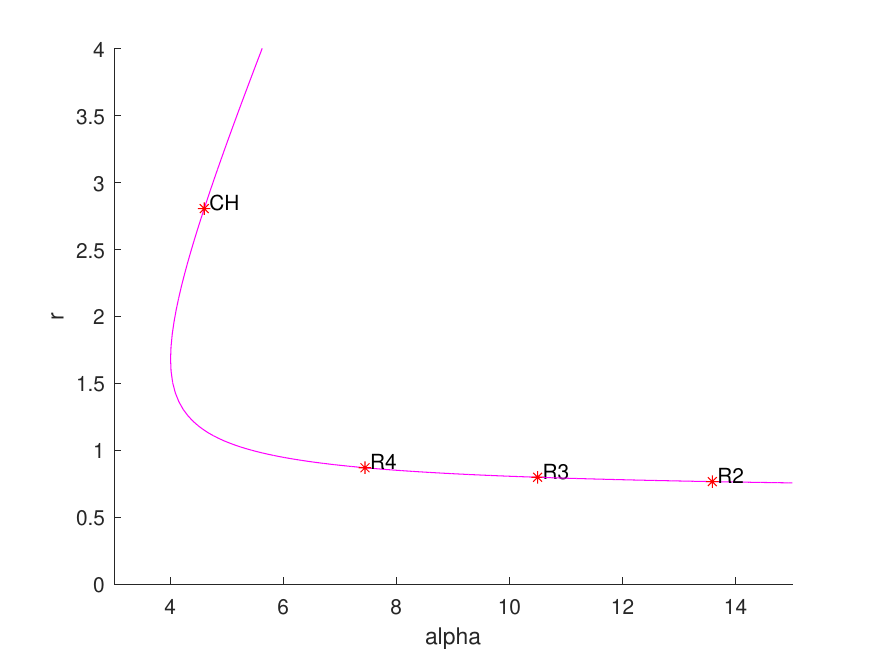}
}
\caption{Neimark-Sacker bifurcation curve including ``R2'',``R3'', ``R4'' and ``CH'' points.}
\label{cod-2}
\end{figure}
In what follows, we will verify Corollaries \ref{cor6.11} - \ref{cor6.31}.
In Figure \ref{cod-2}, a ``CH'' point is visible on the Neimark-Sacker bifurcation curve.
For the parameter $(\alpha,r)$ in a small neighborhood of the Chenciner bifurcation point$(\Psi_3, -(\beta^2-1)/\beta)$, there exists a parameter curve in system \eqref{eq2.1} along which the Neimark-Sacker bifurcation can occur. This implies that the first conclusion of Corollary \ref{cor6.12} is valid. Moreover, with parameters $N=16.32$, $\beta=0.72$, $r=0.66887$ and $\alpha=4.96031746$, and setting the initial values
\begin{eqnarray*}
&&(x_{41},y_{41},z_{44})=(5.8908,6.091217853,5.659079786),
\\
&&(x_{42},y_{42},z_{45})=(5.8909,6.091217853,5.659079786),
\end{eqnarray*}
the blue and red orbits are obtained (see Figure \ref{CH-2}). From Figure \ref{CH-2}, we can observe an invariant circle $\Gamma_1$ at the intersection of the blue and red orbits. The blue orbit starting from the invariant circle $\Gamma_1$ approaches the fixed point $E_2$, while the red orbit moves away from $\Gamma_1$ towards infinity. This indicates that $\Gamma_1$ is unstable.
Furthermore, according to Corollary \ref{cor6.12}, when $L_2>0$ and the parameters lie within the region $\mathfrak{G_1}$, the mapping \eqref{eq2.1} has two invariant circles. The outer one is unstable, and the inner one is stable. Thus, we use {\it MatlabR2022a} to plot the bifurcation diagram of mapping \eqref{eq2.1} in the $(x,y,z)-$space (see Figures \ref{CH-1}). With $N=16.32$, $\beta=0.72$, $r=0.668891$ and $\alpha=4.96031746$, and taking initial values
\begin{eqnarray*}
&&(x_{43},y_{43},z_{43})=(5.8905,6.091217853,5.659079786),
\\
&&(x_{44},y_{44},z_{44})=(5.8906,6.091217853,5.659079786),
\\
&&(x_{45},y_{45},z_{45})=(4.569702359,6.091217853,5.659079786),
\end{eqnarray*}
we numerically simulate the cyan, red, and blue orbits of mapping \eqref{eq2.1} as shown in Figure \ref{CH-1}, respectively. From Figure \ref{CH-1}, we can see that there are two invariant circles $\Gamma_u$ and $\Gamma_s$ (where $\Gamma_u$ is located outside of $\Gamma_s$),
The red orbit starting from the outer invariant circle $\Gamma_u$ moves towards infinity, the cyan orbit starting from the outer invariant circle $\Gamma_u$ approaches the inner invariant circle $\Gamma_s$, and the blue orbit starting from the fixed point $E_2$ also approaches the inner invariant circle $\Gamma_s$. This implies that the outer invariant circle $\Gamma_u$ is unstable and the inner invariant circle $\Gamma_s$ is stable. In addition, Corollary \ref{cor6.12} also states that as the parameters approach the region $\mathfrak{G_2}$, the two invariant circles of the mapping \eqref{eq2.1} merge into a single invariant circle that is unstable on the outside and stable on the inside. Hence, with parameters $N=16.32$, $\beta=0.72$, $r=0.6689363$ and $\alpha=4.96031746$, and setting initial values
\begin{eqnarray*}
&&(x_{46},y_{46},z_{46})=(5.7508,6.091217853,5.659079786),
\\
&&(x_{47},y_{47},z_{47})=(4.569702359,6.091217853,5.659079786),
\end{eqnarray*}
we obtain the blue and red orbits as shown in Figure \ref{CH-3}. From Figure \ref{CH-3}, we observe an invariant circle $\Gamma_1$ at the intersection of the blue and red orbits. The blue orbit starting from the invariant circle $\Gamma_1$ moves away from $\Gamma_1$ towards infinity, and the red orbit starting from the fixed point $E_2$ approaches the invariant circle $\Gamma_1$. This implies that $\Gamma_1$ is unstable on the outside and stable on the inside. Therefore, all the results of Corollary \ref{cor6.12} are verified. The verification of Corollary \ref{cor6.11} can be obtained in a similar manner.

\begin{figure}[ht!]
\T\T\T\T\T\T\T\T\T\T\T\T\T\T\T\T\T\T\T\T\T\T\T\T\T\T\T\T\T\T\T\T\T\T\T\T
\subfigure[Two invariant circles $\Gamma_u$ and $\Gamma_s$.]{
\includegraphics[width=7.5cm]{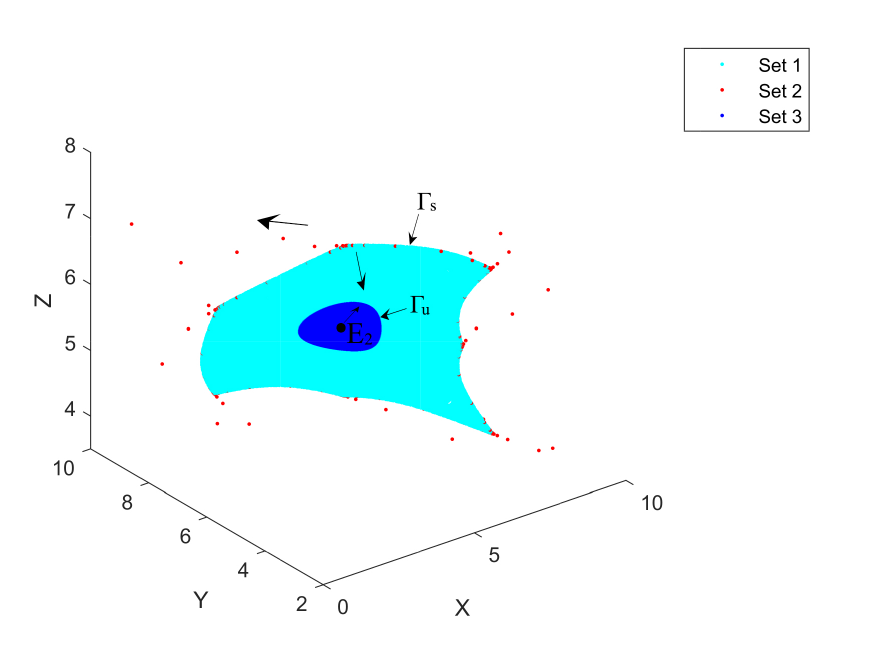}
\label{CH-1}
}
\quad
\subfigure[An unstable invariant circle $\Gamma_1$.]{
\includegraphics[width=7.5cm]{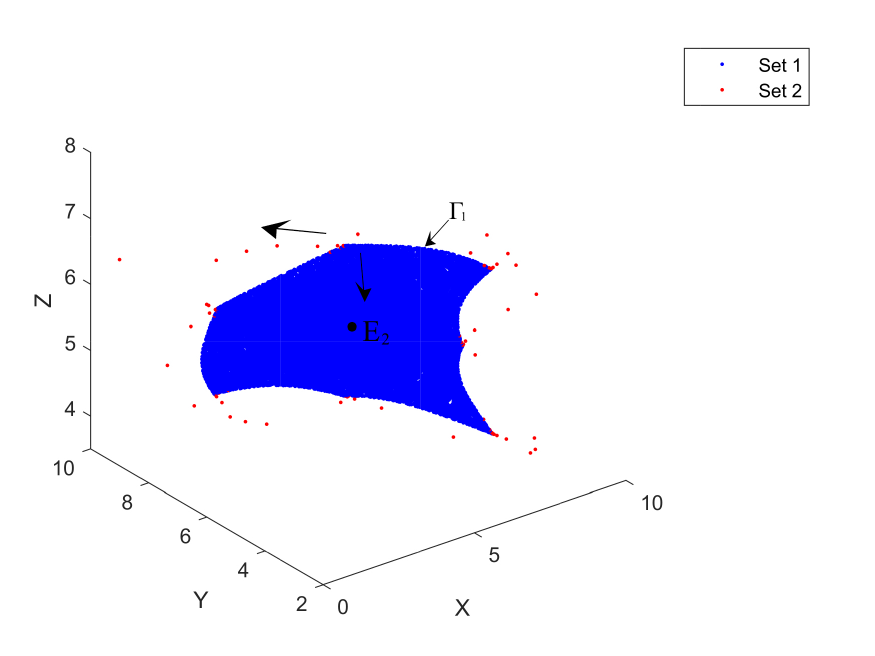}
\label{CH-2}
}
\quad
\centering
\subfigure[An unstable from the outside and stable from the inside invariant circle $\Gamma_1$.]{
\includegraphics[width=7.5cm]{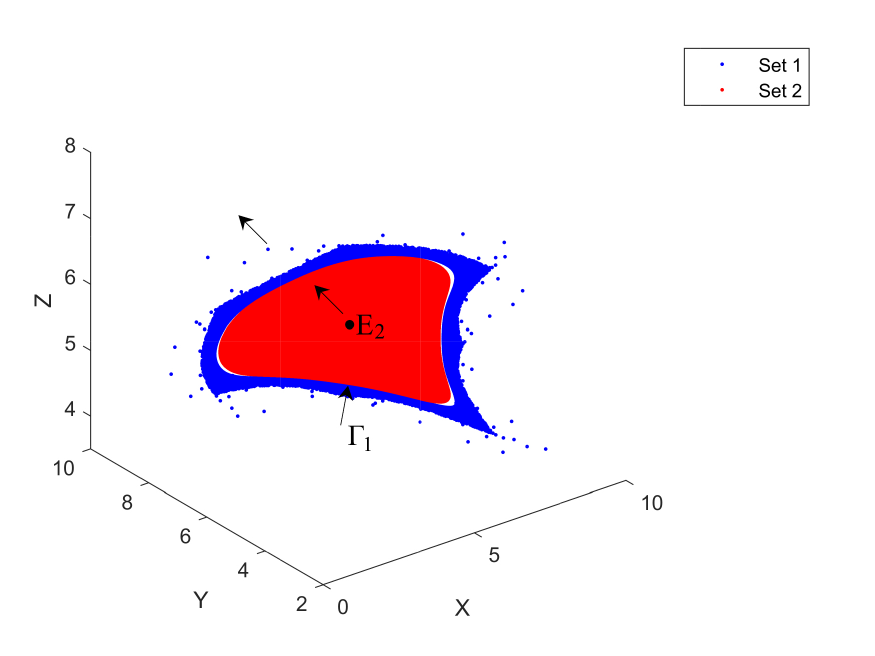}
\label{CH-3}
}
\caption{Invariant circles arising from Chenciner bifurcation.}
\end{figure}
Besides, we will analyze the bifurcation phenomena of the mapping \eqref{eq2.1} in the vicinity of the $1:3$ resonance point to verify the correctness of Corollary \ref{cor6.21}.
As depicted in Figure \ref{cod-2}, a ``R3'' point is present on the Neimark-Sacker bifurcation curve. Consequently, the first conclusion of Corollary \ref{cor6.21} is valid.
Furthermore, by choosing the ``R3'' point in Figure \ref{cod-2} as the initial value, selecting $A$ and $d$ as the active parameters, setting ``NS-curve x3'' as the ``initializer'', and modifying ``Amplitude'', ``VarTolerance'', ``FunTolerance'' and ``TestTolerance'' as $0.0001$, $10e-03$, $10e-03$ and $10e-02$, respectively, while keeping the other options at the system defaults in MatcontM, we detect a curve (denoted as $\mathcal{F}_1$) near the ``R3'' point (see Figure \ref{1bi3-1}). In the output panel of MatcontM, we can observe that the parameters $r$ and $\alpha$ of the last point of curve $\mathcal{F}_1$ are $10.104305$ and $0.76024475$ respectively.
Moreover, when we change ``MaxNumPoint'' to $5000$, that is, increase the number of iterations, and keep the other options unchanged as described above in MatcontM, we also detect a curve (denoted as $\mathcal{F}_2$) in Figure \ref{1bi3-2}. In the output panel of MatcontM, we can observe that the parameters $r$ and $\alpha$ of the last point of curve $\mathcal{F}_2$ are $10.196342$ and $0.0.76378977$ respectively. We note that the last point of curve $\mathcal{F}_2$ is closer to the ``R3'' point than the last point of curve . As we further increase the value of ``MaxNumPoint'', we find that the last point of the detected curve gets closer to the ``R3'' point. Hence, we can state that the curve detected by the above-mentioned method starting from the ``R3'' point eventually returns to the ``R3'' point, which corresponds to a neutral saddle cycle of period-3. This verifies the second conclusion of Corollary \ref{cor6.21}.
Furthermore, we numerically simulate the dynamic properties by employing {\it MatlabR2022a} to further verify Corollary \ref{cor6.21}. Selecting the parameters $N=1.25$, $\beta=0.31998$, $r=0.7983$ and $\alpha=10.5721$, and taking the initial values
\begin{eqnarray*}
&&(x_{51},y_{51},z_{51})=(0.133238,0.3197913967,0.7977795999),
\\
&&(x_{52},y_{52},z_{52})=(0.133239,0.3197913967,0.7977795999),
\\
&&(x_{53},y_{53},z_{53})=(0.13018,0.32614,0.793394),
\end{eqnarray*}
in Matlab R2022a, we obtain the cyan, blue, and red orbits respectively (see Figures \ref{1bi3-3} and \ref{1bi3-4}). From Figures \ref{1bi3-3} and \ref{1bi3-4}, we can see that there exists a period-3 saddle point cycle $\{T_{11}, T_{12}, T_{13}\}$ located at the junction of the red and blue orbits and an invariant circle $\Gamma$. Since the blue orbit moves away from $\Gamma$  towards infinity and the cyan orbit from $\Gamma$ approaches the fixed point $E_2$, the invariant circle $\Gamma$ is unstable. This implies that the period-3 saddle point cycle $\{T_{11}, T_{12}, T_{13}\}$ and the invariant circle $\Gamma$ coexist.
Additionally, by setting parameters $N=1.25$, $\beta=0.31998$, $r=0.79829$ and $\alpha=10.5721$, and letting the initial values
\begin{eqnarray*}
&&(x_{54},y_{54},z_{54})=(0.133258,0.3197913967,0.7977795999),
\\
&&(x_{55},y_{55},z_{55})=(0.13018,0.32614,0.793394),
\end{eqnarray*}
we obtain the cyan and blue orbits respectively as shown in Figures \ref{1bi3-5} (or \ref{1bi3-6}). From Figure \ref{1bi3-5}, we observe that a period-3 saddle point cycle $\{T_{11}, T_{12}, T_{13}\}$ lies on the invariant circle $\Gamma_2$, forming a homoclinic structure.

\begin{figure}[ht!]
\T\T\T\T\T\T\T\T\T\T\T\T\T\T\T\T\T\T\T\T\T\T\T\T\T\T\T\T\T\T\T\T\T\T\T\T
\subfigure[Bifurcation curve near $1:3$ resonance point by keeping ``MaxNumPoint'' as the system default.]{
\includegraphics[width=7.5cm]{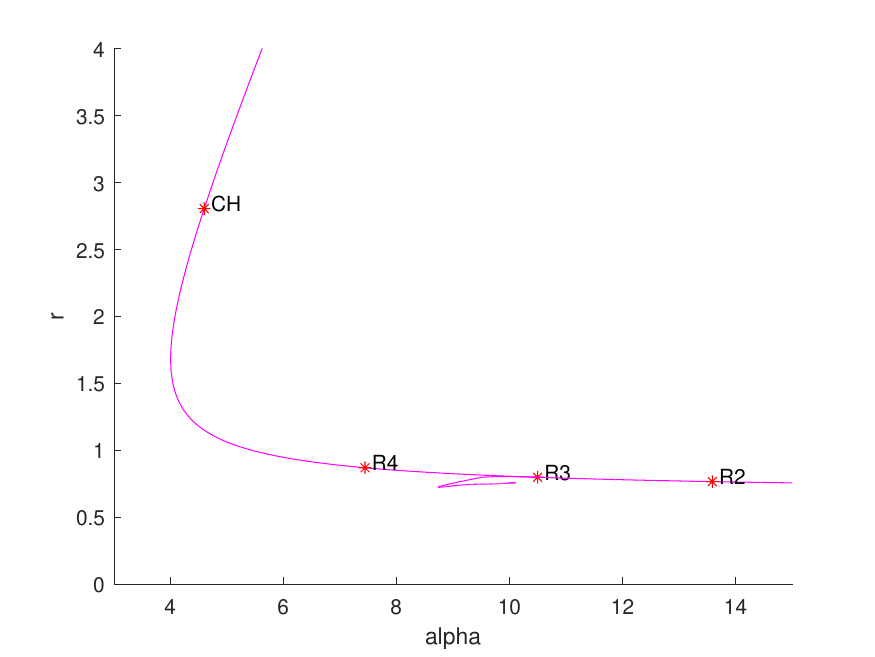}
\label{1bi3-1}
}
\quad
\subfigure[Bifurcation curve near $1:3$ resonance point by changing ``MaxNumPoint'' as $5000$.]{
\includegraphics[width=7.5cm]{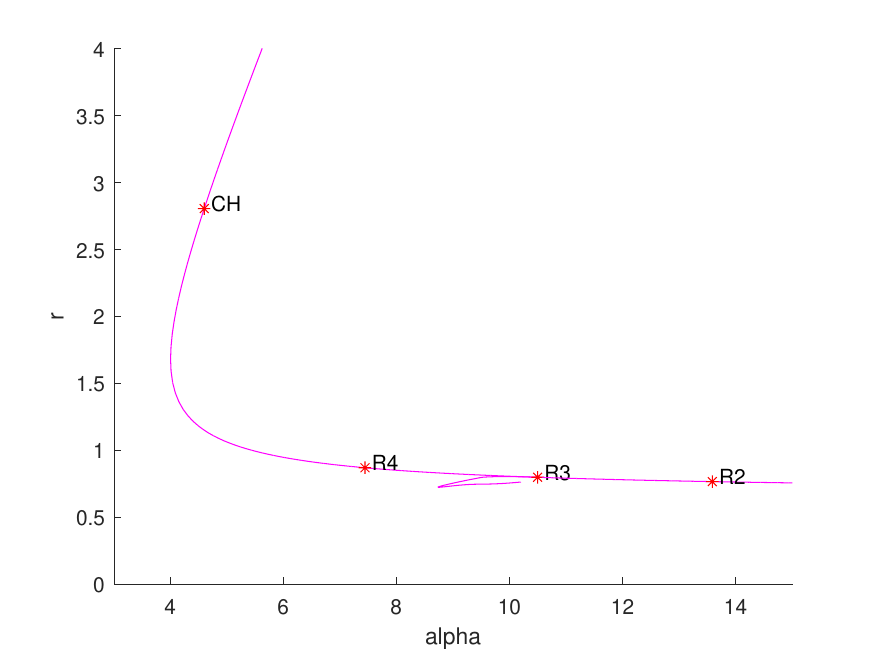}
\label{1bi3-2}
}
\caption{Detection portraits of system (1.2) in MatcontM.}
\end{figure}
\begin{figure}[ht!]
\T\T\T\T\T\T\T\T\T\T\T\T\T\T\T\T\T\T\T\T\T\T\T\T\T\T\T\T\T\T\T\T\T\T\T\T
\subfigure[$N=1.25$, $\beta=0.31998$, $r=0.7983$ and $\alpha=10.5721$.]{
\includegraphics[width=7.5cm]{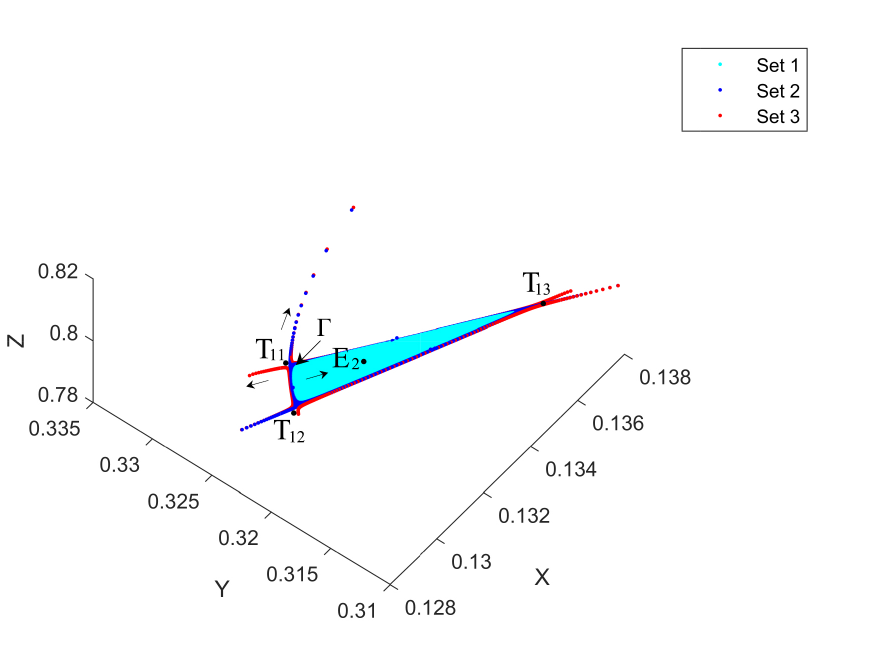}
\label{1bi3-3}
}
\quad
\subfigure[Zoom at Figure \ref{1bi3-3}.]{
\includegraphics[width=7.5cm]{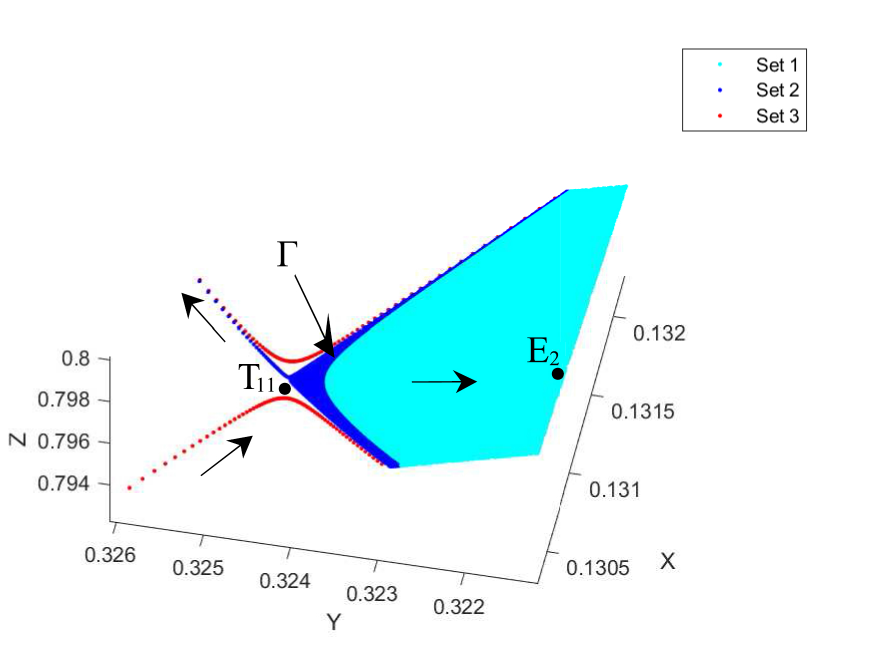}
\label{1bi3-4}
}
\caption{phase portraits of system \eqref{eq2.1} in $(r, \alpha)-$plane.}
\end{figure}
\begin{figure}[ht!]
\T\T\T\T\T\T\T\T\T\T\T\T\T\T\T\T\T\T\T\T\T\T\T\T\T\T\T\T\T\T\T\T\T\T\T\T
\subfigure[$N=1.25$, $\beta=0.31998$, $r=0.79829$ and $\alpha=10.5721$.]{
\includegraphics[width=7.5cm]{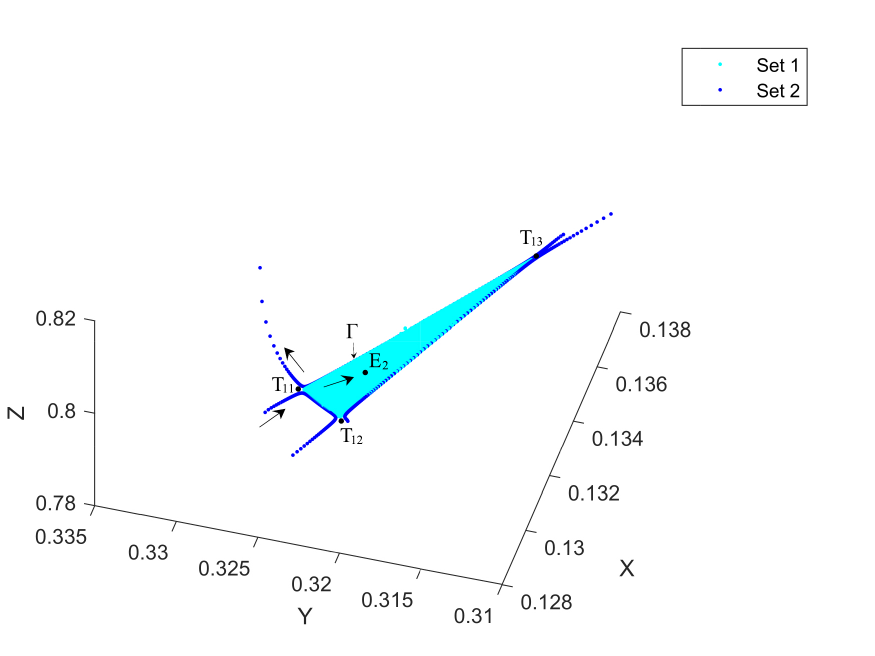}
\label{1bi3-5}
}
\quad
\subfigure[Zoom at Figure \ref{1bi3-5}.]{
\includegraphics[width=7.5cm]{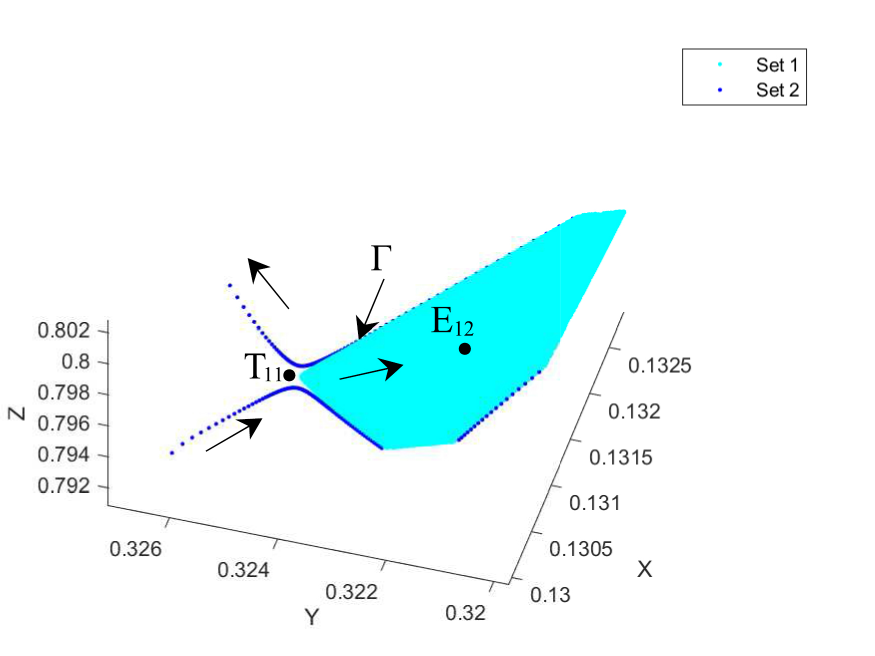}
\label{1bi3-6}
}
\caption{phase portraits of system \eqref{eq2.1} in $(r, \alpha)-$plane.}
\end{figure}

Now, we further verify the correctness of Corollary \ref{cor6.31}.
In Figure \ref{cod-2}, a ``R4'' point is detected on the Neimark-Sacker bifurcation curve. Thus, the first conclusion of Corollary \ref{cor6.31} has been verified.
In addition, according to Corollary \ref{cor6.31}, we know that if $(a_0,b_0)$ belongs to region II in Figure \ref{partion4-1} and $(r,\alpha)$ is close to the curves $H_{41}$ and $H_{42}$, then there exists a period-4 saddle node on the invariant circle $\Gamma$, forming a homoclinic structure.
Therefore, by choosing the parameters $N=10.25$, $\beta=0.80685$, $r=0.86940958$ and $\alpha=4.155094036$, and selecting the initial values
\begin{eqnarray*}
&&(x_{61},y_{61},z_{61})=(4.135,2.943,3.171),
~~~(x_{62},y_{62},z_{62})=(3,4,3.3),
\end{eqnarray*}
in Matlab R2022a, we respectively simulate a red and a cyan orbit as shown in Figures \ref{1bi4-3} and \ref{1bi4-4}, where Figure \ref{1bi4-4} is an enlargement of Figure \ref{1bi4-3}. From Figure \ref{1bi4-4}, we can see that there is a period-4 saddle-node $\{Q_{1}, Q_{2}, Q_{3}, Q_{4}\}$ at the junction of the red orbit and the cyan orbit.
Next, by setting $N=10.25$, $\beta=0.80685$, $r=0.869401$ and $\alpha=4.155094036$, while keeping the initial values unchanged, we respectively obtain the blue and magenta orbits in Figures \ref{1bi4-5} and \ref{1bi4-6}. From Figure \ref{1bi4-6}, we observe a period-4 saddle $\{S_{1}, S_{2}, S_{3}, S_{4}\}$ and a period-4 focus $\{N_{1}, N_{2}, N_{3}, N_{4}\}$.
In addition, we choose the parameters $N=10.24$, $\beta=0.80685$, $r=0.86942$ and $\alpha=4.155094036$, and take the $(x_{62},y_{62},z_{62})$ and
\begin{eqnarray*}
&&(x_{63},y_{63},z_{63})=(4.13,2.94044885,3.168476221),
\end{eqnarray*}
as the initial values, the blue and red orbits are obtained respectively, as shown in Figures \ref{1bi4-7} and \ref{1bi4-8}. From Figure \ref{1bi4-8}, we can observe that there is an invariant circle $\Gamma_2$ generated by the Neimark-Sacker bifurcation at the junction of the blue and red orbits. Since the red orbit inside the $\Gamma_2$ approaches the $\Gamma_2$ from the fixed point $E_2$, and the blue orbit outside $\Gamma_2$ also approaches $\Gamma_2$, this implies that the invariant circle $\Gamma_2$ is stable. It is worth noting that at this time, there are no saddle points outside the invariant circle in Figure \ref{1bi4-8}, which is also a distinctive feature when the parameter $(a_0,b_0)$ belongs to region II in Figure \ref{partion4-1}.
Thus, Corollary \ref{cor6.31} is validated.

\begin{figure}[ht!]
\T\T\T\T\T\T\T\T\T\T\T\T\T\T\T\T\T\T\T\T\T\T\T\T\T\T\T\T\T\T\T\T\T\T\T\T
\subfigure[$\beta=1.5201$, $r=2.645$ and $\alpha=5.48245614$.]{
\includegraphics[width=7.5cm]{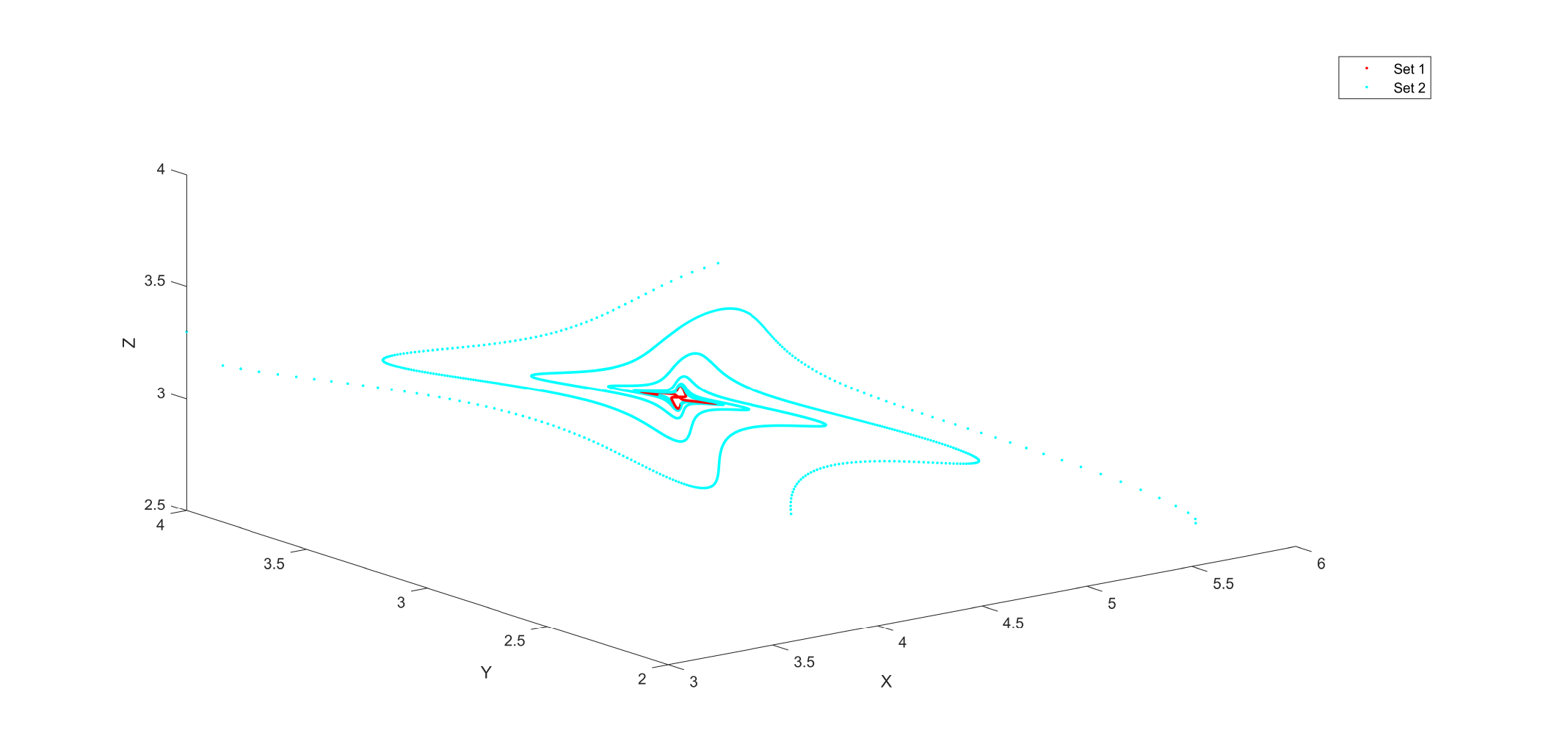}
\label{1bi4-3}
}
\quad
\subfigure[Zoom at Figure \ref{1bi4-3}.]{
\includegraphics[width=7.5cm]{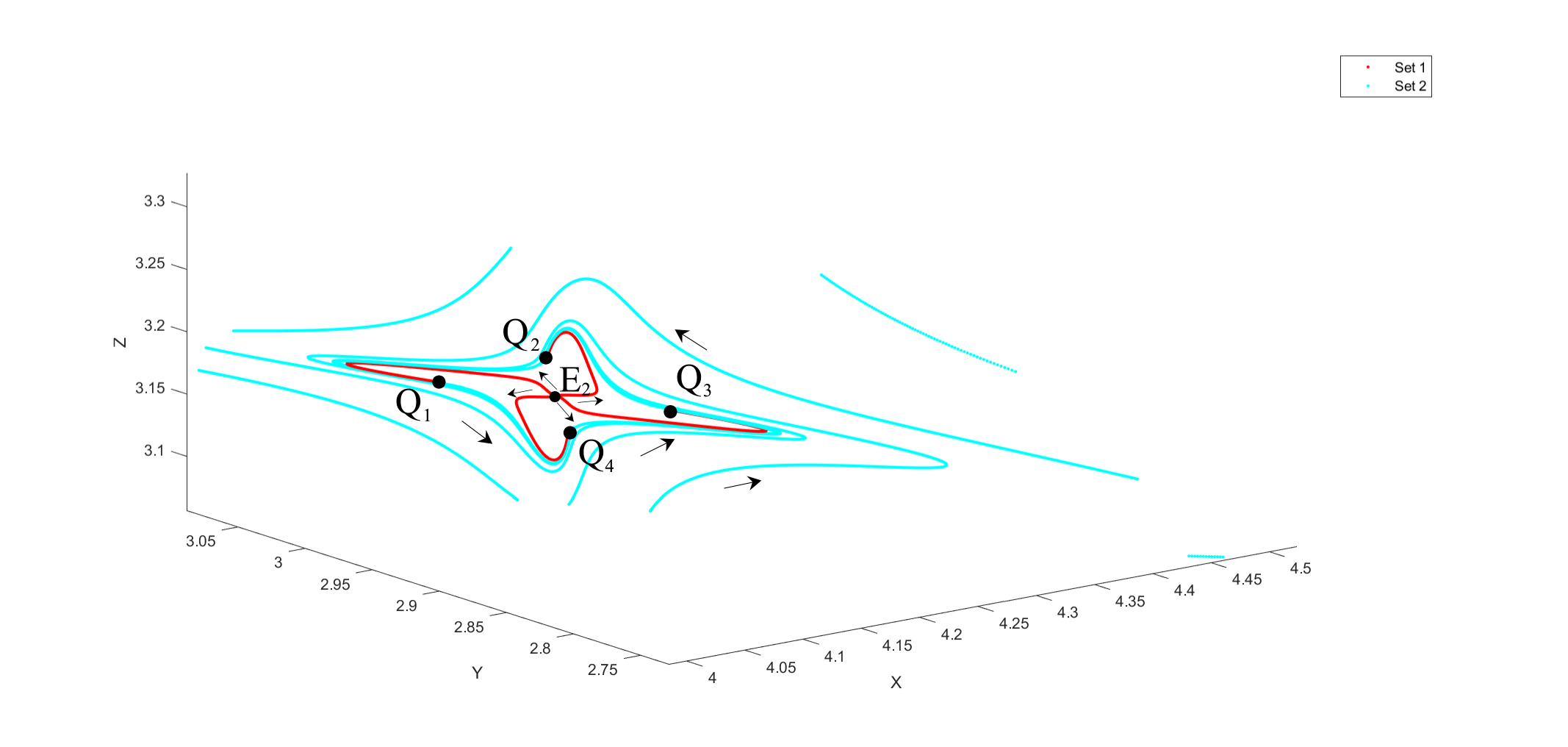}
\label{1bi4-4}
}
\caption{phase portraits of system \eqref{eq2.1} in $(r, \alpha)-$plane.}
\end{figure}
\begin{figure}[ht!]
\T\T\T\T\T\T\T\T\T\T\T\T\T\T\T\T\T\T\T\T\T\T\T\T\T\T\T\T\T\T\T\T\T\T\T\T
\subfigure[$N=10.25$, $\beta=0.80685$, $r=0.869401$ and $\alpha=4.155094036$.]{
\includegraphics[width=7.5cm]{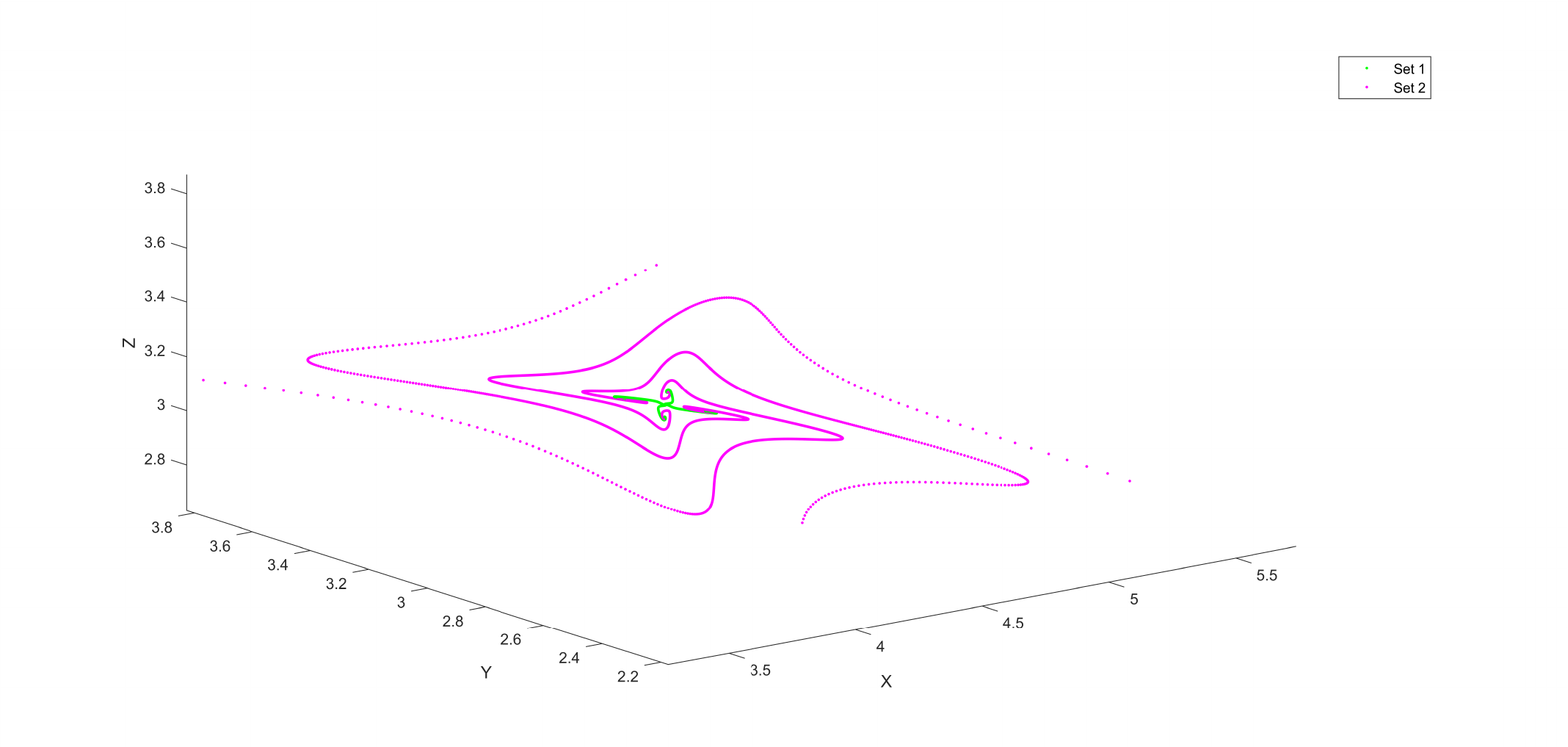}
\label{1bi4-5}
}
\quad
\subfigure[Zoom at Figure \ref{1bi4-5}.]{
\includegraphics[width=7.5cm]{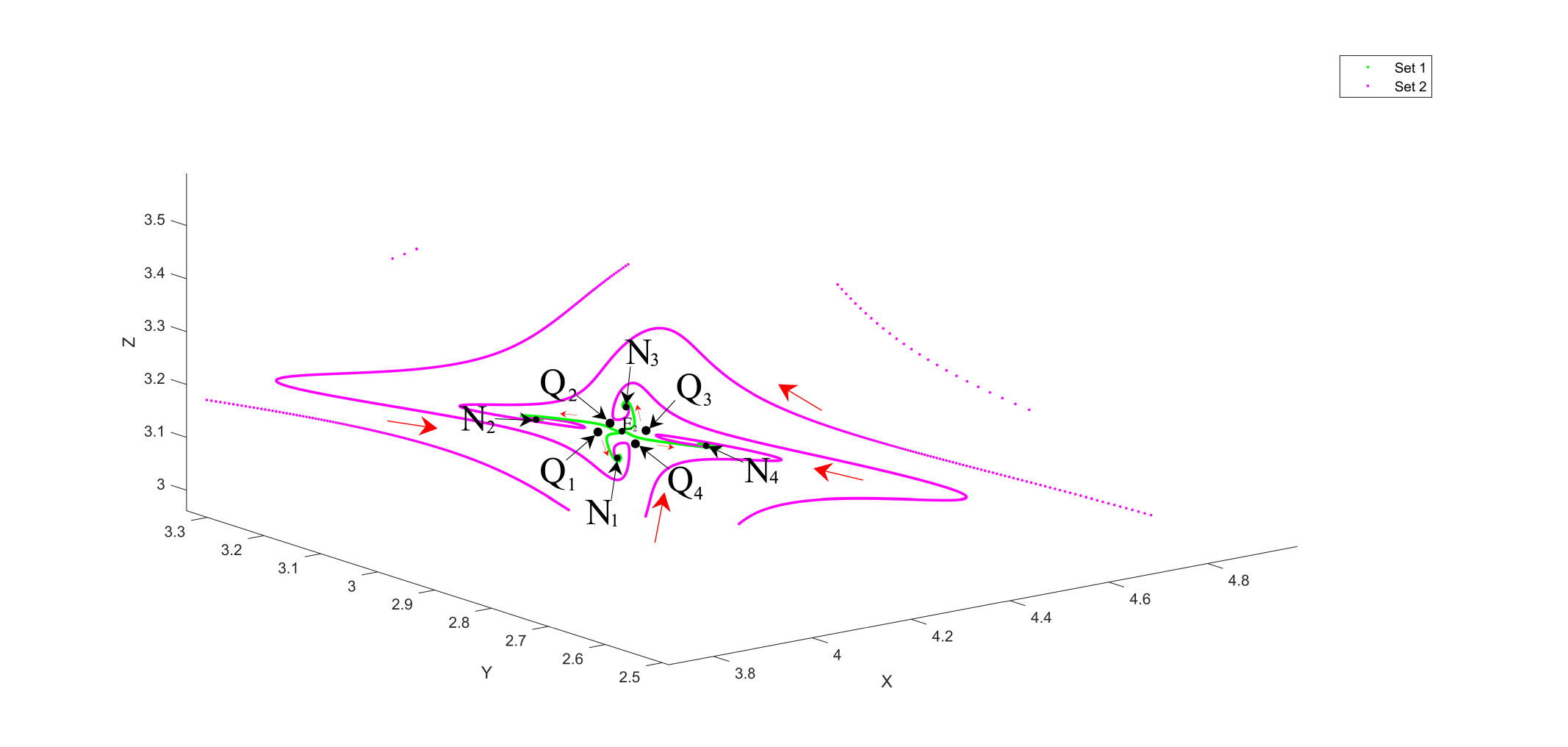}
\label{1bi4-6}
}
\caption{phase portraits of system \eqref{eq2.1} in $(r, \alpha)-$plane.}
\end{figure}
\begin{figure}[ht!]
\T\T\T\T\T\T\T\T\T\T\T\T\T\T\T\T\T\T\T\T\T\T\T\T\T\T\T\T\T\T\T\T\T\T\T\T
\subfigure[$N=10.24$, $\beta=0.80685$, $r=0.86942$ and $\alpha=4.155094036$.]{
\includegraphics[width=7.5cm]{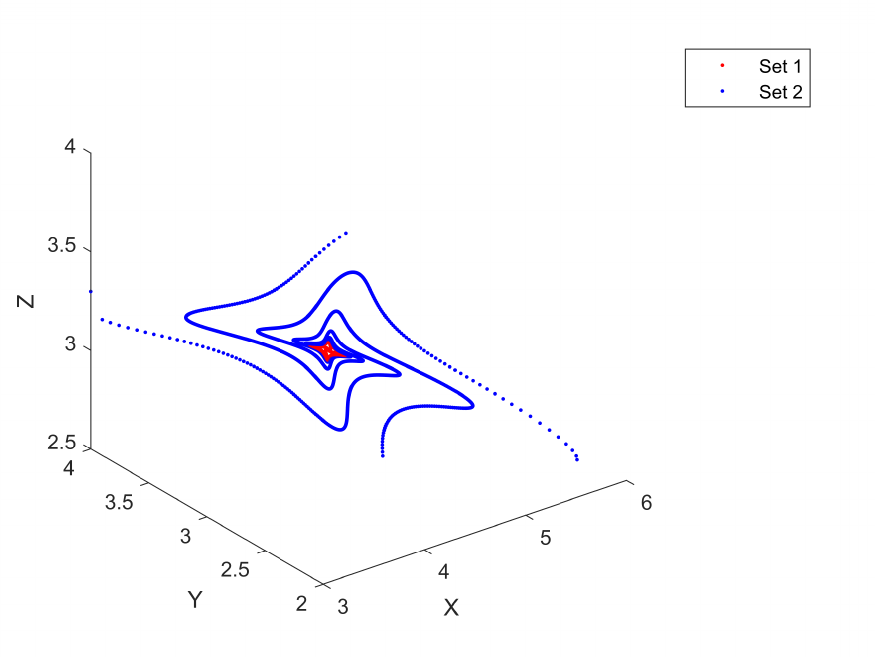}
\label{1bi4-7}
}
\quad
\subfigure[Zoom at Figure \ref{1bi4-5}.]{
\includegraphics[width=7.5cm]{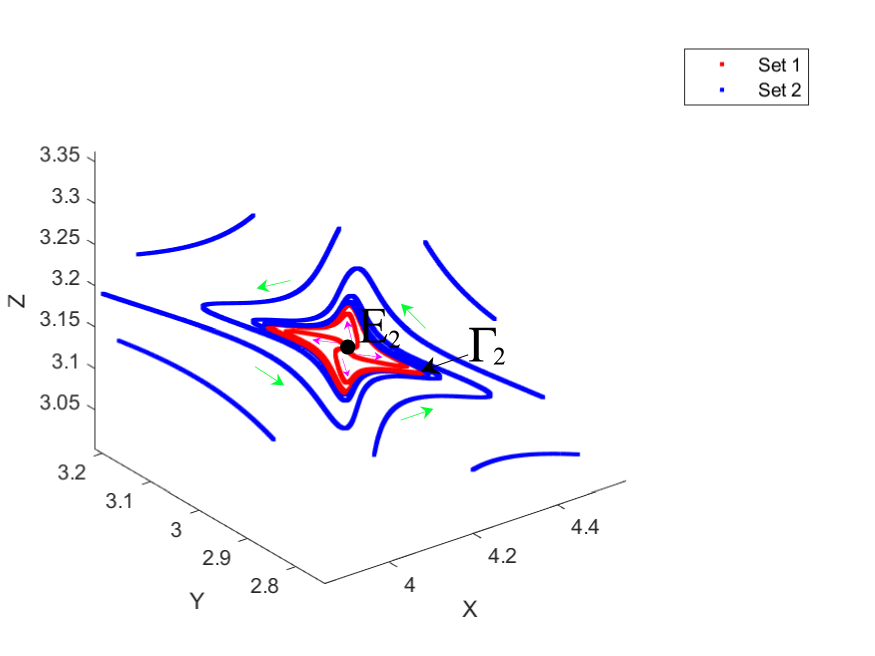}
\label{1bi4-8}
}
\caption{phase portraits of system \eqref{eq2.1} in $(r, \alpha)-$plane.}
\end{figure}

\begin{figure}[ht!]
\T\T\T\T\T\T\T\T\T\T\T\T\T\T\T\T\T\T\T\T\T\T\T\T\T\T\T\T\T\T\T\T\T\T\T\T
\subfigure[An Arnold tongue corresponding to the rotation number $2/5$.]{
\includegraphics[width=7.0cm]{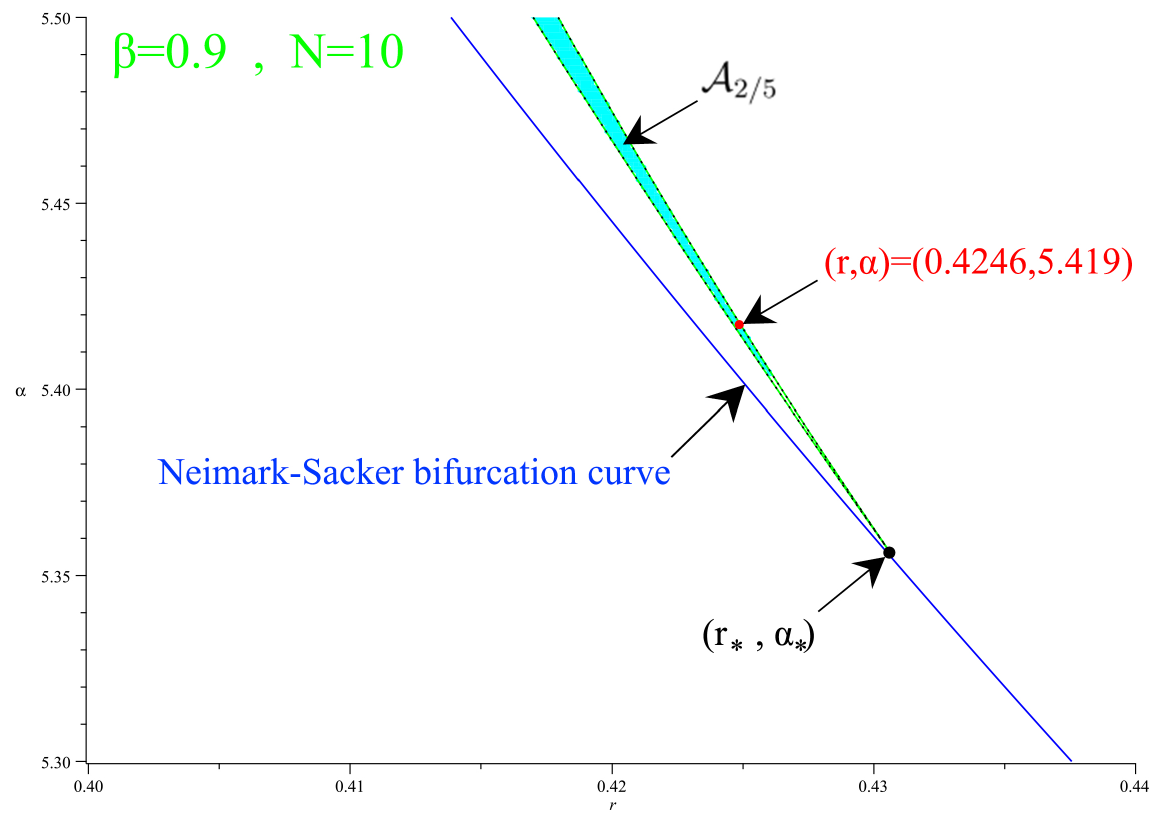}
\label{arnold-1}
}
\subfigure[A stable 5-periodic orbit $\{U_1,U_2,U_3,U_4,U_5\}$ on the invariant circle.]{
\includegraphics[width=7.0cm]{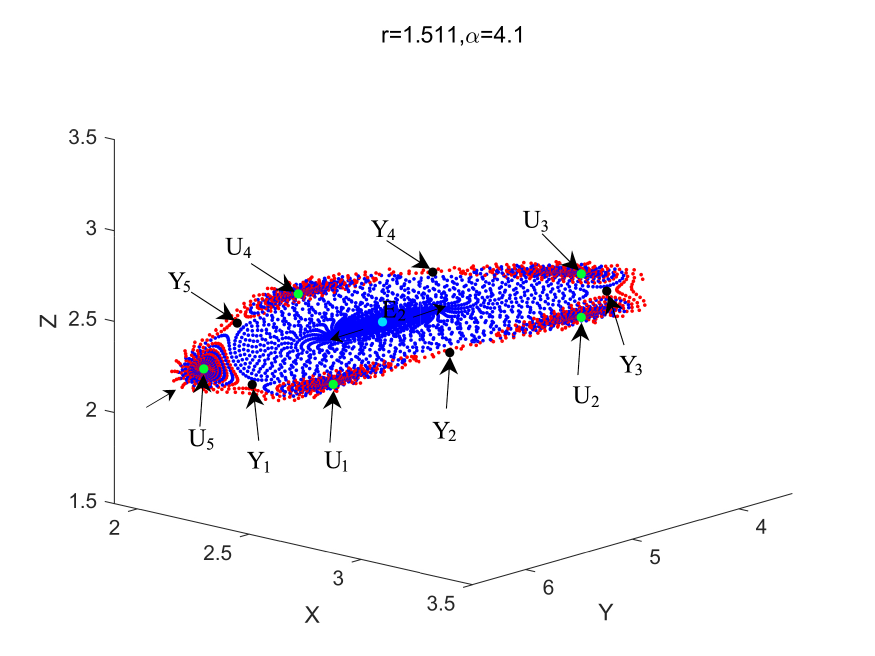}
\label{wr+1}
}
\T\T\T\T\T\T\T\T\T\T\T\T\T\T\T\T\T\T\T\T\T\T\T\T\T\T\T\T\T\T\T\T\T\T\T\T
\subfigure[The red orbit approaches the invariant circle from the outside of the invariant circle.]{
\includegraphics[width=7.0cm]{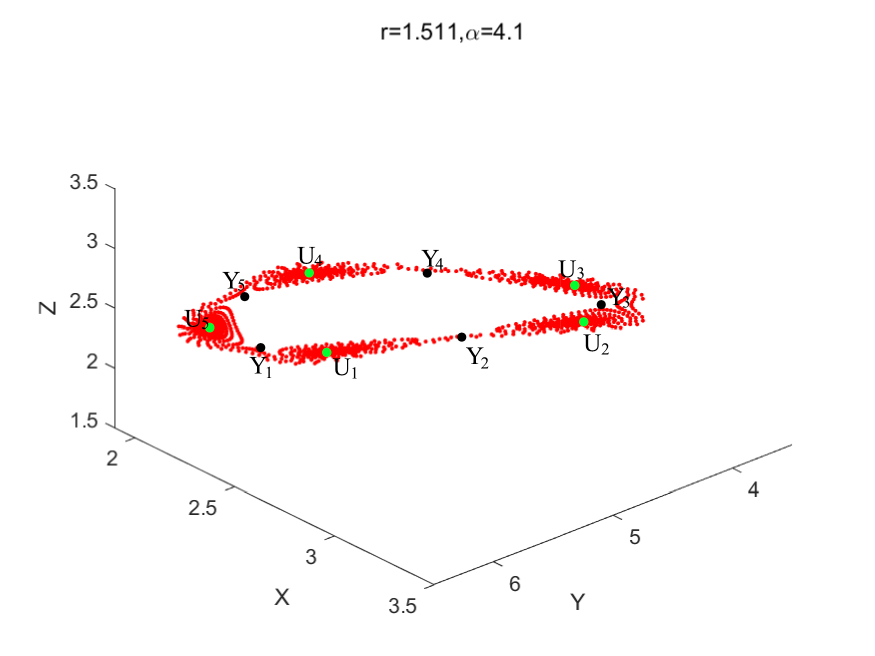}
\label{wr+2}
}
\quad
\quad
\quad
\subfigure[The blue orbit approaches the invariant circle from the fixed point $E_2$.]{
\includegraphics[width=7.0cm]{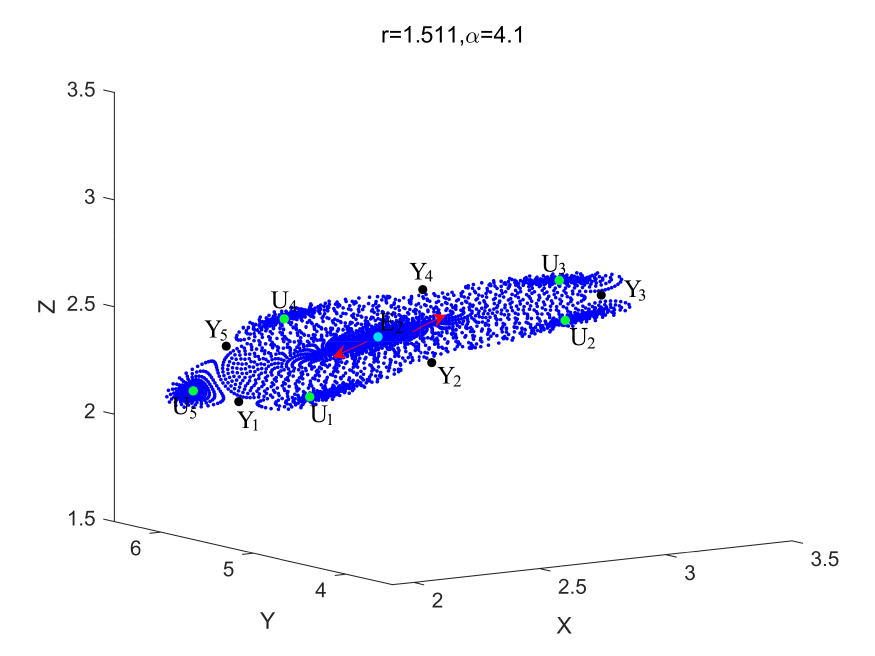}
\label{wr3}
}
\caption{Arnold tongue and stable 5-periodic orbit}
\label{Arnold-YJQ}
\end{figure}

In addition, we perform a numerical simulation of an Arnold tongue $\mathcal{A}_{2/5}$ and 5-periodic orbits on the invariant closed curve resulting from the Neimark-Sacker bifurcation when the parameters are set as $\beta=0.9$ and $N=10$.  Given $n=2$ and $m=5$, we obtain $r*\approx0.4311216871$ and $\alpha_*\approx5.351159455$. In this particular case, mapping \eqref{a2} can be written as
\begin{eqnarray*}
&&z\mapsto\tilde{t}\,z+\varrho_{2,1}\,z^{2}\,\bar{z}^{1}+\varsigma\,\bar{z}^{4}+O(|z|^5),
\end{eqnarray*}
where
\begin{eqnarray*}
&&\tilde{t}\approx-0.8090169949 + 0.5877852523\,{\bf i},\\
&&\varrho_{2,1}\approx0.05167849067 + 0.06713669603\,{\bf i},\\
&&\varsigma\approx-0.004804598419 - 0.009002834572\,{\bf i}.
\end{eqnarray*}
Additionally, it can be verified that
\begin{eqnarray*}
&&\check{\varrho_3}\approx-0.002346818430,~~\tilde{\varrho_2}\approx-0.084690582253,~~|\varsigma|\approx0.01020466542.
\end{eqnarray*}
According to Theorem \ref{th7.1}, the Arnold tongue $\mathcal{A}_{2/5}$ is expressed as
\begin{eqnarray*}
&&\mathcal{A}_{1/5}=\{(\lambda,\mu)|T_{-}-\frac{1\,\pi}{5}<arctan(-{\frac {\sqrt {\Xi_{0}}}{ \left( 4\,\mu-3 \right) {\lambda}^{2}-2\,\lambda\,\mu+{\mu}^{2}}})<T_{+}-\frac{1\,\pi}{5}\},
\end{eqnarray*}
where
\begin{eqnarray*}
&&T_{\pm}\approx36.08740291\,\left(\sqrt{{\frac {-{\beta}^{3}+ \left( \alpha-2\,r \right) {\beta}^{2}- \left( r
-1 \right)  \left( r-\alpha+1 \right) \beta+r}{\beta+r}}
}-1\right)\\
&&~~~~~~~\pm89.75931609\,\left(\sqrt{{\frac {-{\beta}^{3}+ \left( \alpha-2\,r \right) {\beta}^{2}- \left( r
-1 \right)  \left( r-\alpha+1 \right) \beta+r}{\beta+r}}
}-1\right)^{\frac{3}{2}},
\end{eqnarray*}
By using Maple 2021, an Arnold tongue $\mathcal{A}_{2/5}$, represented by the cyan region in Figure \ref{arnold-1}, is simulated in the
$(\lambda,\mu)$-space. Besides, when we select the parameters $r=0.4246$ and $\alpha=5.419$, which are located within the Arnold tongue $\mathcal{A}_{2/5}$ (as shown in Figure \ref{arnold-1}), and utilize Matlab R2022a to simulate system \eqref{eq2.1}, after $1\times10^5$ iterations, we obtain a stable 5-periodic orbit $\{U_1,U_2,U_3,U_4,U_5\}$, depicted as five green points in Figure \ref{wr+1}, on the invariant circle resulting from the Neimark-Sacker bifurcation. Moreover, from Figures \ref{wr+1}, \ref{wr+2} and \ref{wr3}, it can be observed that the period-5 orbit serves as the $\omega$-limit set of the blue orbit with the initial value
$$(x_{70}, y_{70},z_{70})=(2.444362428, 5.133680971, 2.4219566)$$
near the fixed point $E_*$ and the red - blue orbit with the initial value
$$(x_{71}, y_{72}, z_{73})=(2.18, 5.78, 2.12)$$
outside the invariant circle, respectively. Furthermore, there exists a saddle 5-periodic orbit $\{Y_1,Y_2,Y_3,Y_4,Y_5\}$ near the black points on the invariant circle.


\section*{{\fontsize{13}{10}\selectfont Appendix: Expressions of $\Upsilon_6$ and $\Upsilon_7$}}
{\fontsize{9}{10}\selectfont\begin{eqnarray*}
&&\!\!\!\!\!\!\!\!\!\!\!\!\!\!\!\!
\Upsilon_6:=({\beta}^{2}\epsilon-\beta\,\epsilon-1 ) ^{
2} ( {\epsilon}^{4} ( \epsilon-\varepsilon ) ^
{2}{\beta}^{16}-6\, ( \epsilon-\varepsilon )  ( \frac{1}{2}\,{
\varepsilon}^{4}-\frac{4}{3}\,\epsilon\,{\varepsilon}^{3}+ ( \frac{5}{6}\,{
\epsilon}^{2}-{\frac {5\,\epsilon}{12}} ) {\varepsilon}^{2}-{
\frac {7\,{\epsilon}^{2}\varepsilon}{12}}+{\epsilon}^{2} (
\epsilon+\frac{1}{3} )  ) {\epsilon}^{2}{\beta}^{15}\\
&&~+ ( -4\,
\epsilon\,{\varepsilon}^{7}+ ( 13\,{\epsilon}^{2}-5\,\epsilon
 ) {\varepsilon}^{6}+ ( -14\,{\epsilon}^{3}+2\,{\epsilon}^{
2} ) {\varepsilon}^{5}+ ( 5\,{\epsilon}^{4}+41\,{\epsilon}^{3}+{\frac {33\,{\epsilon}^{2}}{2}} ) {\varepsilon}^{4}+ (
-68\,{\epsilon}^{4}-{\frac {65\,{\epsilon}^{3}}{2}} ) {
\varepsilon}^{3}\\
&&~+ ( -\frac{23}{2}\,{\epsilon}^{3}+30\,{\epsilon}^{5}+26\,
{\epsilon}^{4} ) {\varepsilon}^{2}+ ( -25\,{\epsilon}^{5}+8
\,{\epsilon}^{4} ) \varepsilon+15\,{\epsilon}^{6}+{\epsilon}^{4}
+\frac{7}{2}\,{\epsilon}^{5} ) {\beta}^{14}+ (  ( 1+19\,
\epsilon ) {\varepsilon}^{7}+ ( -68\,{\epsilon}^{2}\\
&&~-8\,
\epsilon ) {\varepsilon}^{6}+ ( -37\,\epsilon+79\,{\epsilon
}^{3}+{\frac {37\,{\epsilon}^{2}}{2}} ) {\varepsilon}^{5}+
 ( 35\,{\epsilon}^{2}-30\,{\epsilon}^{4}-{\frac {183\,{\epsilon}^
{3}}{2}} ) {\varepsilon}^{4}+ ( 95\,{\epsilon}^{3}+{\frac {
87\,{\epsilon}^{2}}{2}}+155\,{\epsilon}^{4} ) {\varepsilon}^{3}\\
&&~+
 ( -136\,{\epsilon}^{4}-{\frac {63\,{\epsilon}^{3}}{2}}-75\,{
\epsilon}^{5} ) {\varepsilon}^{2}+ ( -{\frac {49\,{\epsilon
}^{4}}{2}}-{\frac {31\,{\epsilon}^{3}}{2}}+{\frac {125\,{\epsilon}^{5}
}{2}} ) \varepsilon-20\,{\epsilon}^{6}+10\,{\epsilon}^{5}+9\,{
\epsilon}^{4} ) {\beta}^{13}+ ( -{\varepsilon}^{8}\\
&&~+ (-28\,\epsilon-2 ) {\varepsilon}^{7}+ ( 134\,{\epsilon}^{2}+
91\,\epsilon+12 ) {\varepsilon}^{6}+ ( -180\,{\epsilon}^{3}
-218\,{\epsilon}^{2}+32\,\epsilon ) {\varepsilon}^{5}+ ( 75
\,{\epsilon}^{4}+259\,{\epsilon}^{3}\\
&&~-176\,{\epsilon}^{2}-113\,\epsilon
 ) {\varepsilon}^{4}+ ( -230\,{\epsilon}^{4}-\frac{11}{2}\,{
\epsilon}^{3}+93\,{\epsilon}^{2} ) {\varepsilon}^{3}+ ( 100
\,{\epsilon}^{5}+255\,{\epsilon}^{4}+91\,{\epsilon}^{3}+{\frac {135\,{
\epsilon}^{2}}{2}} ) {\varepsilon}^{2}+ ( -130\,{\epsilon}^
{5}\\
&&~-{\frac {61\,{\epsilon}^{4}}{2}}-{\frac {59\,{\epsilon}^{3}}{2}}
 ) \varepsilon+15\, ( {\epsilon}^{3}-{\frac {11\,{\epsilon}
^{2}}{6}}-\frac{3}{2}\,\epsilon-{\frac{13}{30}} ) {\epsilon}^{3}
 ) {\beta}^{12}+ ( 6\,{\varepsilon}^{8}+ ( -7\,
\epsilon-8 ) {\varepsilon}^{7}+ ( -104\,{\epsilon}^{2}-130
\,\epsilon\\
&&~-29 ) {\varepsilon}^{6}+ ( 55+205\,{\epsilon}^{3}+{\frac {961\,{\epsilon}^{2}}{2}}+241\,\epsilon ) {\varepsilon}^
{5}+ ( -100\,{\epsilon}^{4}-{\frac {1125\,{\epsilon}^{3}}{2}}-{
\frac {221\,{\epsilon}^{2}}{2}}+120\,\epsilon ) {\varepsilon}^{4
}+ ( 295\,{\epsilon}^{4}\\
&&~-49\,{\epsilon}^{3}-527\,{\epsilon}^{2}-{
\frac {379\,\epsilon}{2}} ) {\varepsilon}^{3}+ ( -75\,{
\epsilon}^{5}-231\,{\epsilon}^{4}+{\frac {445\,{\epsilon}^{3}}{2}}+95
\,{\epsilon}^{2} ) {\varepsilon}^{2}+ ( {\frac {349\,{
\epsilon}^{5}}{2}}+90\,{\epsilon}^{4}+116\,{\epsilon}^{3}\\
&&~+{\frac {111
\,{\epsilon}^{2}}{2}} ) \varepsilon-{\frac {45\,{\epsilon}^{4}}{
2}}-{\frac {39\,{\epsilon}^{3}}{2}}+5\,{\epsilon}^{5}-6\,{\epsilon}^{6
} ) {\beta}^{11}+ ( -15\,{\varepsilon}^{8}+ ( 70\,
\epsilon+41 ) {\varepsilon}^{7}+ ( -20\,{\epsilon}^{2}-97\,
\epsilon-14 ) {\varepsilon}^{6}\\
&&~+ ( -110\,{\epsilon}^{3}-292
\,{\epsilon}^{2}-401\,\epsilon-129 ) {\varepsilon}^{5}+ (
75\,{\epsilon}^{4}+646\,{\epsilon}^{3}+{\frac{263}{2}}+{\frac {1495\,{
\epsilon}^{2}}{2}}+542\,\epsilon ) {\varepsilon}^{4}+ ( -
328\,{\epsilon}^{4}\\
&&~-{\frac {715\,{\epsilon}^{3}}{2}}+{\frac {297\,{
\epsilon}^{2}}{2}}+{\frac {319\,\epsilon}{2}} ) {\varepsilon}^{3
}+ ( 30\,{\epsilon}^{5}+186\,{\epsilon}^{4}-611\,{\epsilon}^{3}-
607\,{\epsilon}^{2}-{\frac {379\,\epsilon}{2}} ) {\varepsilon}^{
2}+ ( -137\,{\epsilon}^{5}+33\,{\epsilon}^{4}\\
&&~+{\frac {291\,{
\epsilon}^{3}}{2}}+61\,{\epsilon}^{2} ) \varepsilon+{\epsilon}^{
2} ( 75\,{\epsilon}^{2}+18+59\,\epsilon+{\frac {89\,{\epsilon}^{3
}}{2}}+{\epsilon}^{4} )  ) {\beta}^{10}+ ( 20\,{
\varepsilon}^{8}+ ( -91\,\epsilon-85 ) {\varepsilon}^{7}+
 ( 92\,{\epsilon}^{2}\\
&&~+440\,\epsilon+122 ) {\varepsilon}^{6}
+ ( -\frac{17}{2}+9\,{\epsilon}^{3}-{\frac {589\,{\epsilon}^{2}}{2}}-227
\,\epsilon ) {\varepsilon}^{5}+ ( -30\,{\epsilon}^{4}-{
\frac {617\,{\epsilon}^{3}}{2}}-274-404\,{\epsilon}^{2}-{\frac {1885\,
\epsilon}{2}} ) {\varepsilon}^{4}\\
&&~+ ( 253\,{\epsilon}^{4}+{
\frac{367}{2}}+685\,{\epsilon}^{3}+{\frac {2081\,{\epsilon}^{2}}{2}}+{
\frac {1713\,\epsilon}{2}} ) {\varepsilon}^{3}+ ( -5\,{
\epsilon}^{5}-266\,{\epsilon}^{4}+{\frac {373\,{\epsilon}^{3}}{2}}+{
\frac {213\,{\epsilon}^{2}}{2}}+{\frac {213\,\epsilon}{2}} ) {
\varepsilon}^{2}\\
&&~+ ( {\frac {115\,{\epsilon}^{5}}{2}}-{\frac {215
\,\epsilon}{2}}-{\frac {321\,{\epsilon}^{4}}{2}}-504\,{\epsilon}^{3}-{\frac {773\,{\epsilon}^{2}}{2}} ) \varepsilon-56\,{\epsilon}^{5}
+4\,{\epsilon}^{4}+55\,{\epsilon}^{3}+26\,{\epsilon}^{2} ) {
\beta}^{9}+ ( -15\,{\varepsilon}^{8}+ ( 56\,\epsilon\\
&&~+104
 ) {\varepsilon}^{7}+ ( -62\,{\epsilon}^{2}-507\,\epsilon-
203 ) {\varepsilon}^{6}+ ( 16\,{\epsilon}^{3}+554\,{
\epsilon}^{2}+1080\,\epsilon+258 ) {\varepsilon}^{5}+ ( 5\,
{\epsilon}^{4}-41\,{\epsilon}^{3}-{\frac{91}{2}}-850\,{\epsilon}^{2}\\
&&~-180\,\epsilon ) {\varepsilon}^{4}+ ( -110\,{\epsilon}^{4}-
320-{\frac {565\,{\epsilon}^{3}}{2}}-698\,{\epsilon}^{2}-1291\,
\epsilon ) {\varepsilon}^{3}+ ( {\frac{307}{2}}+286\,{
\epsilon}^{4}+409\,{\epsilon}^{3}+{\frac {2479\,{\epsilon}^{2}}{2}}\\
&&~+{
\frac {1603\,\epsilon}{2}} ) {\varepsilon}^{2}+ ( -10\,{
\epsilon}^{5}+{\frac {49\,\epsilon}{2}}+{\frac {79\,{\epsilon}^{4}}{2}
}+{\frac {351\,{\epsilon}^{3}}{2}}-\frac{15}{2}\,{\epsilon}^{2} )
\varepsilon-{\frac {223\,{\epsilon}^{4}}{2}}-110\,{\epsilon}^{2}-182\,
{\epsilon}^{3}+{\frac {55\,{\epsilon}^{5}}{2}}\\
&&~-{\frac {53\,\epsilon}{2
}} ) {\beta}^{8}+ ( 6\,{\varepsilon}^{8}+ ( -17\,
\epsilon-78 ) {\varepsilon}^{7}+ ( 16\,{\epsilon}^{2}+294\,\epsilon+215 ) {\varepsilon}^{6}+ ( -293-5\,{\epsilon}^{3}-
{\frac {659\,{\epsilon}^{2}}{2}}-1149\,\epsilon ) {\varepsilon}^
{5}\\
&&~+ ( {\frac {187\,{\epsilon}^{3}}{2}}+{\frac{879}{2}}+{\frac {
2587\,{\epsilon}^{2}}{2}}+1453\,\epsilon ) {\varepsilon}^{4}+
 ( 20\,{\epsilon}^{4}-{\frac{285}{2}}-{\frac {429\,{\epsilon}^{3}
}{2}}-1020\,{\epsilon}^{2}-{\frac {311\,\epsilon}{2}} ) {
\varepsilon}^{3}+ ( -210\\
&&~-150\,{\epsilon}^{4}-{\frac {277\,{
\epsilon}^{3}}{2}}-898\,{\epsilon}^{2}-923\,\epsilon ) {
\varepsilon}^{2}+ ( {\frac{145}{2}}+417\,\epsilon+113\,{\epsilon}
^{4}+297\,{\epsilon}^{3}+804\,{\epsilon}^{2} ) \varepsilon+{
\frac {155\,{\epsilon}^{4}}{2}}-35\,{\epsilon}^{2}\\
&&~+34\,{\epsilon}^{3}-
5\,{\epsilon}^{5}-\frac{13}{2}\,\epsilon ) {\beta}^{7}+ ( -{
\varepsilon}^{8}+ ( 2\,\epsilon+33 ) {\varepsilon}^{7}+
 ( -{\epsilon}^{2}-90\,\epsilon-164 ) {\varepsilon}^{6}+ ( 82\,{\epsilon}^{2}+627\,\epsilon+216 ) {\varepsilon}^{5}
\\
&&~+ ( -25\,{\epsilon}^{3}-707\,{\epsilon}^{2}-1380\,\epsilon-298
 ) {\varepsilon}^{4}+ ( {\frac{1019}{2}}+214\,{\epsilon}^{3
}+{\frac {3065\,{\epsilon}^{2}}{2}}+{\frac {2579\,\epsilon}{2}}
 ) {\varepsilon}^{3}+ ( -{\frac{353}{2}}+30\,{\epsilon}^{4}
\\
&&~-{\frac {615\,{\epsilon}^{3}}{2}}-643\,{\epsilon}^{2}-{\frac {605\,
\epsilon}{2}} ) {\varepsilon}^{2}+ ( -69-307\,\epsilon-90\,
{\epsilon}^{4}-{\frac {227\,{\epsilon}^{3}}{2}}-507\,{\epsilon}^{2}
 ) \varepsilon+5\,{\epsilon}^{4}+132\,{\epsilon}^{3}+219\,{
\epsilon}^{2}\\
&&~+93\,\epsilon+15 ) {\beta}^{6}+ ( -6\,{
\varepsilon}^{7}+ ( 76+12\,\epsilon ) {\varepsilon}^{6}+
 ( -{\frac{347}{2}}-6\,{\epsilon}^{2}-196\,\epsilon ) {
\varepsilon}^{5}+ ( {\frac{161}{2}}+170\,{\epsilon}^{2}+{\frac {
1371\,\epsilon}{2}} ) {\varepsilon}^{4}\\
&&~+ ( -196-50\,{
\epsilon}^{3}-773\,{\epsilon}^{2}-{\frac {1893\,\epsilon}{2}} )
{\varepsilon}^{3}+ ( {\frac{681}{2}}+241\,{\epsilon}^{3}+{\frac {
1993\,{\epsilon}^{2}}{2}}+{\frac {1607\,\epsilon}{2}} ) {
\varepsilon}^{2}+ ( -{\frac{201}{2}}-269\,\epsilon\\
&&~+20\,{\epsilon}
^{4}-{\frac {361\,{\epsilon}^{3}}{2}}-{\frac {547\,{\epsilon}^{2}}{2}}
 ) \varepsilon-7-20\,{\epsilon}^{4}-91\,{\epsilon}^{2}-59\,{
\epsilon}^{3}-{\frac {51\,\epsilon}{2}} ) {\beta}^{5}+ ( -
15\,{\varepsilon}^{6}+ ( 95+30\,\epsilon ) {\varepsilon}^{5
}+ ( -15\,{\epsilon}^{2}\\
&&~-225\,\epsilon-89 ) {\varepsilon}^{
4}+ ( -{\frac{97}{2}}+180\,{\epsilon}^{2}+392\,\epsilon ) {
\varepsilon}^{3}+ ( -50\,{\epsilon}^{3}-442\,{\epsilon}^{2}-{
\frac {711\,\epsilon}{2}}-53 ) {\varepsilon}^{2}+ ( 134\,{
\epsilon}^{3}+{\frac {703\,{\epsilon}^{2}}{2}}\\
&&~+{\frac {601\,\epsilon}{
2}}+111 ) \varepsilon-75\,{\epsilon}^{2}+5\,{\epsilon}^{4}-{
\frac{43}{2}}-{\frac {75\,{\epsilon}^{3}}{2}}-79\,\epsilon ) {
\beta}^{4}+ ( -20\,{\varepsilon}^{5}+ ( 40\,\epsilon+70
 ) {\varepsilon}^{4}+ ( -20\,{\epsilon}^{2}-145\,\epsilon\\
&&~-11 ) {\varepsilon}^{3}+ ( -{\frac{139}{2}}+100\,{\epsilon}^
{2}+99\,\epsilon ) {\varepsilon}^{2}+ ( -25\,{\epsilon}^{3}
-{\frac {235\,{\epsilon}^{2}}{2}}-{\frac {103\,\epsilon}{2}}+12
 ) \varepsilon+56\,{\epsilon}^{2}+{\frac {87\,\epsilon}{2}}+\frac{23}{2}
\\
&&~+{\frac {59\,{\epsilon}^{3}}{2}} ) {\beta}^{3}+ ( -15\,{
\varepsilon}^{4}+ ( 31+30\,\epsilon ) {\varepsilon}^{3}+( -15\,{\epsilon}^{2}-52\,\epsilon+8 ) {\varepsilon}^{2}
 + ( -{\frac{59}{2}}+26\,{\epsilon}^{2}+\epsilon )
\varepsilon+\frac{9}{2}\,\epsilon+7-5\,{\epsilon}^{3}\\
&&~-9\,{\epsilon}^{2}
 ) {\beta}^{2}+ ( -6\,{\varepsilon}^{3}+ ( 8+12\,
\epsilon ) {\varepsilon}^{2}+ ( \frac{5}{2}
 -6\,{\epsilon}^{2}-10\,
\epsilon ) \varepsilon-\frac{5}{2}\,\epsilon-4+2\,{\epsilon}^{2}
 ) \beta- ( \epsilon-\varepsilon+1 )  ( \epsilon
-\varepsilon )  ),\\
&&\!\!\!\!\!\!\!\!\!\!\!\!\!\!\!\!
\Upsilon_7:=4\, (  ( \epsilon-\varepsilon ) \beta-\epsilon+
\varepsilon-1 ) ^{2} ( \beta-1 ) ^{4}{N}^{2} (
\beta\,\varepsilon+1 ) ^{2}{\beta}^{5} ( {\beta}^{5}{
\epsilon}^{2}+ ( -4\,\epsilon\,{\varepsilon}^{2}+4\,{\varepsilon}
^{3}-2\,{\epsilon}^{2} ) {\beta}^{4}+ ( -8\,{\varepsilon}^{
3}+ ( 8\,\epsilon\\
&&~+12 ) {\varepsilon}^{2}-8\,\epsilon\,
\varepsilon+{\epsilon}^{2}-2\,\epsilon ) {\beta}^{3}+ ( 4\,
{\varepsilon}^{3}+ ( -4\,\epsilon-20 ) {\varepsilon}^{2}+
 ( 16\,\epsilon+12 ) \varepsilon-2\,\epsilon ) {\beta
}^{2}+ ( 8\,{\varepsilon}^{2}+ ( -8\,\epsilon-16 )
\varepsilon+8\,\epsilon\\
&&~+5 ) \beta-4\,\epsilon+4\,\varepsilon-4
 ) .
\end{eqnarray*}}
\end{document}